%% file: 2021_06_30_arxiv.tex
\documentclass{amsart}

\usepackage[T2A,T1]{fontenc}
\usepackage[lutf8]{luainputenc}
\usepackage[russian,USenglish]{babel}
\usepackage{amsmath,amssymb,amsthm,upgreek}
\usepackage{colortbl,color,xcolor} % colors
\usepackage{graphicx,tikz,pgfplots} % graphics
\usepackage{enumitem}
\pgfplotsset{compat=1.6}
\usepackage{subfigure,float,longtable}
\usepackage[bookmarks=true,bookmarksnumbered,plainpages=false,linktocpage,colorlinks=true,citecolor=green!80!black,linkcolor=red!70!black,filecolor=magenta,urlcolor=magenta,hidelinks,breaklinks,pdfauthor={Thomas Jahn, Tino Ullrich},pdftitle={On the optimal constants in the two-sided Stechkin inequalities}]{hyperref}
\usepackage[nameinlink,capitalize,noabbrev]{cleveref}

\usetikzlibrary{decorations.fractals,calc,intersections,through,backgrounds,arrows,patterns,shapes.geometric,shapes.misc,spy}

\newcommand{\rus}[1]{\selectlanguage{russian}{\fontfamily{erewhon-TLF}\fontsize{9pt}{11pt}\selectfont #1}\selectlanguage{USenglish}}
\newcommand{\bibrus}[1]{\selectlanguage{russian}{\fontfamily{erewhon-TLF}\fontsize{8pt}{10pt}\selectfont #1}\selectlanguage{USenglish}}

\newcommand{\ii}{\mathrm{i}}
\newcommand{\ee}{\mathrm{e}}
\newcommand{\dd}{\mathrm{d}}
\newcommand{\NN}{\mathbb{N}}                                     % natural numbers
\newcommand{\RR}{\mathbb{R}}                                     % real numbers
\newcommand{\ZZ}{\mathbb{Z}}                                     % integer numbers
\newcommand{\fall}{\:\forall\:}                                  % for all
\newcommand{\abs}[1]{\left\lvert#1\right\rvert}                  % absolute value, variable sized delimiters
\newcommand{\abss}[1]{\lvert#1\rvert}                            % absolute value, small delimiters
\newcommand{\mnorm}[1]{\left\lVert#1\right\rVert}                % another norm, variable sized delimiters
\newcommand{\setn}[1]{\left\{#1\right\}}                         % set, variable sized delimiters
\newcommand{\setcond}[2]{\left\{#1 \:\middle\vert\: #2\right\}}  % set with condition, variable sized delimiters
\newcommand{\setconds}[2]{\{#1 \:\vert\: #2\}}                   % set with condition, small delimiters
\newcommand{\defeq}{\mathrel{\mathop:}=}                         % defining equality
\newcommand{\lr}[1]{\!\left(#1\right)}                           % variable sized parentheses
\newcommand{\p}{\partial}                                        % subdifferential and partial derivatives
\newcommand{\skpr}[2]{\left\langle#1 \,\middle\vert\, #2\right\rangle}               % inner product
\newcommand{\norel}{\mathrel{\phantom{=}}}                                           % phantom equal sign
\newcommand{\eqdef}{=\mathrel{\mathop:}}                                             % reversed defining equation

\newcommand{\enquote}[1]{``#1''}                                                   

\newcommand{\floor}[1]{\left\lfloor #1\right\rfloor}

\let \eps \varepsilon
\let \piup \uppi

\theoremstyle{plain} 
\newtheorem{Satz}{Theorem}[section]\newtheorem{Kor}[Satz]{Corollary}
\newtheorem{Lem}[Satz]{Lemma}
\newtheorem{Prop}[Satz]{Proposition}

\theoremstyle{definition} % definition, condition, problem, example

\newtheorem{Bem}[Satz]{Remark}

\crefname{Satz}{Theorem}{Theorems}
\crefname{Prop}{Proposition}{Propositions}
\crefname{Lem}{Lemma}{Lemmas}
\crefname{Kor}{Corollary}{Corollaries}
\crefname{Bem}{Remark}{Remarks}
\crefname{Bsp}{Example}{Examples}
\crefname{Def}{Definition}{Definitions}
\crefname{Alg}{Algorithm}{Algorithms}

\numberwithin{equation}{section}

\makeatletter
\renewcommand*{\eqref}[1]{%
  \hyperref[{#1}]{\textup{\tagform@{\ref*{#1}}}}%
}
\makeatother

\makeatletter
\@ifundefined{textcommabelow}{%
  \DeclareTextCommandDefault\textcommabelow[1]
    {\hmode@bgroup\ooalign{\null#1\crcr\hidewidth\raise-.31ex
     \hbox{\check@mathfonts\fontsize\ssf@size\z@
     \math@fontsfalse\selectfont,}\hidewidth}\egroup}%
}{}
\makeatother

\begin{document}
\allowdisplaybreaks
\parindent 0pt

\title{On the optimal constants in the two-sided Stechkin inequalities}

%    Remove any unused author tags.

%    author one information
\author{Thomas Jahn}
\address[T. Jahn]{Faculty of Mathematics, Technische Universität Chemnitz, 09107 Chemnitz}
\email{thomas.jahn@mathematik.tu-chemnitz.de}
\thanks{}

%    author two information
\author{Tino Ullrich}
\address[T. Ullrich]{Faculty of Mathematics, Technische Universität Chemnitz, 09107 Chemnitz}
\email{tino.ullrich@mathematik.tu-chemnitz.de}
\thanks{}



\date{}

\begin{abstract}
We address the optimal constants in the strong and the weak Stechkin inequalities, both in their discrete and continuous variants. 
These inequalities appear in the characterization of approximation spaces which arise from sparse approximation or have applications to interpolation theory.
An elementary proof of a constant in the strong discrete Stechkin inequality given by Bennett is provided, and we improve the constants given by Levin and Stechkin and by Copson.
Finally, the minimal constants in the weak discrete Stechkin inequalities and both continuous Stechkin inequalities are presented.
\end{abstract}

\maketitle

\section{Introduction}\label{chap:intro}
In the present paper, we address the minimal constants $c_1(q)$, $C_1(q)$, $c_{1,\infty}(q)$, $C_{1,\infty}(q)$, $\bar{c}_1(q)$, $\bar{C}_1(q)$, $\bar{c}_{1,\infty}(q)$, and $\bar{C}_{1,\infty}(q)>0$ in the inequalities
\begin{equation}
\frac{1}{c_1(q)}\sum_{n=1}^\infty\lr{\frac{1}{n}\sum_{k=n}^\infty a_k^q}^{\frac{1}{q}}\leq\sum_{n=1}^\infty a_n\leq C_1(q)\sum_{n=1}^\infty\lr{\frac{1}{n}\sum_{k=n}^\infty a_k^q}^{\frac{1}{q}},\label{eq:general-stechkin-q-intro}
\end{equation}
\begin{equation}
\frac{1}{c_{1,\infty}(q)}\sup_{n\in\NN}n\lr{\frac{1}{n}\sum_{k=n}^\infty a_k^q}^{\frac{1}{q}}\leq\sup_{n\in\NN} na_n\leq C_{1,\infty}(q)\sup_{n\in\NN}n\lr{\frac{1}{n}\sum_{k=n}^\infty a_k^q}^{\frac{1}{q}},\label{eq:general-weak-l1-q-intro}
\end{equation}
\begin{equation}
\frac{1}{\bar{c}_1(q)}\int_0^\infty\lr{\frac{1}{t}\int_t^\infty f(s)^q \dd s}^{\frac{1}{q}}\dd t\leq \int_0^\infty f(t)\dd t\leq \bar{C}_1(q)\int_0^\infty \lr{\frac{1}{t}\int_t^\infty f(s)^q \dd s}^{\frac{1}{q}}\dd t,\label{eq:general-stechkin-continuous-intro}
\end{equation}
and
\begin{equation}
\frac{1}{\bar{c}_{1,\infty}(q)}\sup_{t>0}t\lr{\frac{1}{t}\int_t^\infty f(s)^q \dd s}^{\frac{1}{q}}\leq\sup_{t>0} tf(t)\leq \bar{C}_{1,\infty}(q)\sup_{t>0}t\lr{\frac{1}{t}\int_t^\infty f(s)^q \dd s}^{\frac{1}{q}},\label{eq:general-weak-l1-continuous-intro}
\end{equation}
for sequences $(a_n)_{n\in\NN}$ with $a_1\geq a_2\geq\ldots\geq 0$, monotonically decreasing functions $f:(0,\infty)\to [0,\infty)$, and $1<q\leq\infty$.
These inequalities are henceforth referred to as the \emph{strong discrete}, the \emph{weak discrete}, the \emph{strong continuous}, and the \emph{weak continuous Stechkin inequality}, respectively.

The right-hand side inequalities of \eqref{eq:general-stechkin-q-intro}, \eqref{eq:general-weak-l1-q-intro}, and \eqref{eq:general-weak-l1-continuous-intro} also allow for $q=1$.
Note that in case $q=\infty$, the expressions $\lr{\frac{1}{n}\sum_{k=n}^\infty a_k^q}^{\frac{1}{q}}$ and $\lr{\frac{1}{t}\int_x^\infty f(s)^q \dd s}^{\frac{1}{q}}$ are replaced by $\sup\setcond{a_k}{k\geq n}=a_n$ and $\sup\setcond{f(s)}{s\geq t}=f(t)$.

Copson \cite[Theorem~2.3]{Copson1928} proves that $C_1(q)\leq q^{\frac{1}{q}}$, cf. also Hardy, Littlewood, and Pólya \cite[Theorem~345]{HardyLiPo1934}.
Levin and Stechkin \cite[{\rus{Д}}.61]{HardyLiPo1948} improve Copson's result \cite[Theorem~2.3]{Copson1928} when $1<q<\infty$, showing that $C_1(q)=(q-1)^\frac{1}{q}$ when $3\leq q<\infty$ and giving upper bounds in the remaining cases.
Stechkin revisits \eqref{eq:general-stechkin-q-intro} for $q=2$ in \cite[\rus{Лемма}~3]{Stechkin1951} and \cite[\rus{Лемма}~1]{Stechkin1955} where it is proved that $C_1(2)\leq\frac{2}{\sqrt{3}}$.
Gao \cite[Theorem~1]{Gao2011} provides further improvement by showing that $C_1(q)=(q-1)^{\frac{1}{q}}$ even for $q_0\leq q<\infty$, where $q_0\approx 2.8855$ is a solution of the equation
$2^{\frac{1}{q-1}}\lr{(q-1)^{\frac{q}{q-1}}-(q-1)}-\lr{1+\frac{3-q}{2}}^{\frac{q}{q-1}}=0$.
De Bruijn \cite[p.~174]{DeBruijn1958} reports that $C_1(2)=1.1064957714$ \enquote{with an error of at most $9$ units at the last decimal place.}

Stechkin \cite[\rus{Лемма}~1]{Stechkin1955} is first to address the constant $c_1(2)$; he asserts $c_1(2)\leq 2$ and conjectures an improvement to $c_1(2)\leq\frac{\piup}{2}$ but proves neither claim; see also \cite[Section~7.4]{DungTeUl2018} for a historical discussion.
The existence of constants validating the inequalities \eqref{eq:general-stechkin-q-intro} and \eqref{eq:general-weak-l1-q-intro} is due to Pietsch \cite[Example~1 on p.~123]{Pietsch1981}, see also \cite[Theorem~4]{DeVore1998}.
Bennett \cite[Theorem~3]{Bennett1988} shows that 
\begin{equation*}
c_1(q)=\frac{\piup}{q\sin(\frac{\piup}{q})}
\end{equation*}
in \eqref{eq:general-stechkin-q-intro}, thus confirming Stechkin's conjecture for $q=2$, see \cref{fig:left-stechkin-q} for an illustration.
Hardy, Littlewood, and Pólya \cite[Theorem~337]{HardyLiPo1934} prove 
\begin{equation*}
\bar{C}_1(q)=(q-1)^{\frac{1}{q}}
\end{equation*}
in \eqref{eq:general-stechkin-continuous-intro}, depicted in \cref{fig:stechkin-continuous}.

The contribution of the present paper is as follows.
First, we give an alternative proof for the optimality of $c_1(2)=\frac{\piup}{2}$ in \eqref{eq:general-stechkin-q-intro} which uses an elementary insight from convex optimization.

Second, we extend the upper bound for $C_1(q)\leq 2\lr{2\frac{q}{q-1}-1}^{-\frac{q-1}{q}}$ proved by Levin and Stechkin for $\frac{5}{3}\leq q<3$ to $1<q<\infty$ via \cref{result:right-stechkin-q}.
A more detailed analysis of the same argument leads to $C_1(2)\leq 1.1086983$ in \cref{result:right-stechkin-upper-bound}.

Third, we improve the upper bounds for $C_1(q)$ from the literature when $1\leq q\leq\frac{2+\ln(2)}{2-\ln(2)}$ and $q\neq 2$.
Summarizing, the currently best known bounds for the constant $C_1(q)$ are
\begin{equation*}
C_1(q)\begin{cases}\leq\lr{\frac{\ee\ln(2)}{\sqrt{2}}}^{1-\frac{1}{q}},&1\leq q\leq\frac{2+\ln(2)}{2-\ln(2)},\\
\approx 1.1064957714,&q=2,\\
\leq 2\lr{2\frac{q}{q-1}-1}^{-\frac{q-1}{q}},&\frac{2+\ln(2)}{2-\ln(2)}<q< q_0,\\
=(q-1)^{\frac{1}{q}}&q_0\leq q\leq \infty
\end{cases}
\end{equation*}
with $q_0\approx 2.8855$, see \cref{result:right-stechkin-q,result:improved-right-stechkin-q} and \cref{fig:right-stechkin-q} for an illustration.

Fourth, we determine the optimal constants in \eqref{eq:general-weak-l1-q-intro} as 
\begin{equation*}
c_{1,\infty}(q)=\zeta(q)^{\frac{1}{q}}\text{\quad and\quad}C_{1,\infty}(q)=\lr{\frac{1}{q}}^{-\frac{1}{q}}\lr{1-\frac{1}{q}}^{-\lr{1-\frac{1}{q}}}
\end{equation*}
where $\zeta(\cdot)$ denotes the Riemann zeta-function, see \cref{result:left-weak-q,result:right-weak-q,fig:left-weak-q,fig:right-weak-q}.
Next, we show that
\begin{equation*}
\bar{c}_1(q)=c_1(q)=\frac{\piup}{q\sin(\frac{\piup}{q})}
\end{equation*}
in \eqref{eq:general-stechkin-continuous-intro}, see \cref{result:left-stechkin-continuous,fig:stechkin-continuous}, as well as
\begin{equation*}
\bar{c}_{1,\infty}(q)=(q-1)^{-\frac{1}{q}}\text{\quad and\quad}\bar{C}_{1,\infty}(q)=C_{1,\infty}(q)=\lr{\frac{1}{q}}^{-\frac{1}{q}}\lr{1-\frac{1}{q}}^{-\lr{1-\frac{1}{q}}},
\end{equation*}
see \cref{result:left-weak-continuous,result:right-weak-continuous,fig:weak-l1-continuous}.

The inequalities discussed in this paper have applications in interpolation theory and nonlinear approximation.
On the one hand, \eqref{eq:general-stechkin-continuous-intro} can be used in the proof of the Marcinkiewicz interpolation theorem, see \cite[Theorem~1.3.1]{BerghLo1976}.
On the other hand \eqref{eq:general-stechkin-q-intro} and \eqref{eq:general-weak-l1-q-intro} play a role in the characterization of the approximation spaces $\mathcal{A}_r^\alpha(\mathcal{H})$, i.e., the set of elements $f$ of the infinite-dimensional separable Hilbert space $\mathcal{H}$ for which the quasi-norm
\begin{equation*}
\mnorm{f}_{\mathcal{A}_r^\alpha(\mathcal{H})}\defeq\begin{cases}
\lr{\sum_{n=1}^\infty(n^\alpha E_n(f)_{\mathcal{H}})^r \frac{1}{n}}^{\frac{1}{r}},&0<r<\infty,\\
\sup_{n\in\NN}n^\alpha E_n(f)_{\mathcal{H}},&r=\infty
\end{cases}
\end{equation*}
is finite.
Here, $\alpha>0$ and $E_n(f)$ denote the infimal distance of $f$ to elements of the form $\sum_{k\in\Lambda}\lambda_k e_k$, where $\Lambda\subset \NN$ is a set of cardinality $n-1$ and $(e_k)_{k=1}^\infty$ is an orthonormal basis of $\mathcal{H}$.
The consequences for the optimal constants in the inequalities stated by DeVore in \cite[Theorem~4]{DeVore1998} on sparse approximation in infinite-dimensional separable real Hilbert spaces are outlined in \cref{chap:devore}.

A recurring technique in this paper is that we prove the inequalities under consideration for finite sequences first, which then yields the general claim through a limiting process.
These finite-dimensional versions will be used to gain some geometric insight to \eqref{eq:general-stechkin-q-intro}.

\section{The strong discrete Stechkin inequality}\label{chap:stechkin}
In this section, we are concerned with the minimal constants $c_1(q)$ and $C_1(q)>0$ in the inequality
\begin{equation}
\frac{1}{c_1(q)}\sum_{n=1}^\infty\lr{\frac{1}{n}\sum_{k=n}^\infty a_k^q}^{\frac{1}{q}}\leq \sum_{n=1}^\infty a_n\leq C_1(q)\sum_{n=1}^\infty\lr{\frac{1}{n}\sum_{k=n}^\infty a_k^q}^{\frac{1}{q}}\label{eq:general-stechkin-q}
\end{equation}
for sequences $(a_n)_{n\in\NN}$ with $a_1\geq a_2\geq\ldots\geq 0$ and for $1<q\leq\infty$ with the appropriate modification for $q=\infty$, as indicated in \cref{chap:intro}.
The monotonicity assumption on $(a_n)_{n\in\NN}$ gives $\sup\setcond{a_k}{k\geq n}=a_n$, so $c_1(\infty)=C_1(\infty)=1$.
For $q=1$, we have $C_1(1)=1$ because
\begin{equation*}
\sum_{k=1}^\infty a_k=\sum_{k=1}^\infty \sum_{n=1}^k \frac{a_k}{k}\leq \sum_{k=1}^\infty \sum_{n=1}^k \frac{a_k}{n}= \sum_{n=1}^\infty \frac{1}{n}\sum_{k=n}^\infty a_k
\end{equation*}
which holds as an equality when $a_n=0$ for all $n\geq 2$.

\subsection{On the optimal lower constant}
We give a rather elementary proof of the optimality of $c_1(2)=\frac{\piup}{2}$ in \eqref{eq:general-stechkin-q}, which is due to Bennett \cite[Theorem~3]{Bennett1988}.
Our proof uses an elementary insight from convex optimization.
\begin{Satz}\label{result:left-stechkin}
The minimal constant $c_1(2)>0$ for which
\begin{equation*}
\sum_{n=1}^\infty\lr{\frac{1}{n}\sum_{k=n}^\infty a_k^2}^{\frac{1}{2}}\leq c_1(2)\sum_{n=1}^\infty a_n
\end{equation*}
holds for all sequences $(a_n)_{n\in\NN}$ with $a_1\geq a_2\geq\ldots\geq 0$ is $c_1(2)=\frac{\piup}{2}$.
\end{Satz}
\begin{proof}
\emph{Step 1. We prove the claim for finite sequences.}
Consider
\begin{align*}
&\norel\sup\setcond{\frac{\sum_{n=1}^\infty\lr{\frac{1}{n}\sum_{k=n}^\infty a_k^2}^{\frac{1}{2}}}{\sum_{n=1}^\infty a_n}}{a_1\geq a_2\geq\ldots\geq 0,a_{N+1}=0,\sum_{n=1}^\infty a_n<\infty}\\
&=\sup\setcond{\frac{\sum_{n=1}^N\lr{\frac{1}{n}\sum_{k=n}^N a_k^2}^{\frac{1}{2}}}{\sum_{n=1}^N a_n}}{a_1\geq a_2\geq\ldots\geq 0,a_{N+1}=0}\\
&=\sup\setcond{\sum_{n=1}^N\lr{\frac{1}{n}\sum_{k=n}^N a_k^2}^{\frac{1}{2}}}{(a_n)_{n\in\NN}\in\Delta_N}.
\end{align*}
The set $\Delta_N\defeq\setcond{(a_n)_{n\in\NN}}{a_1\geq a_2\geq\ldots\geq 0,a_{N+1}=0,\sum_{n=1}^N a_n=1}$ is determined by a single linear equality and $n$ linear inequalities in $a_1,\ldots,a_n$ and thus is a $(n-1)$-dimensional simplex in its $N$-dimensional linear span $V_N$.
Therefore, the restriction of the convex function 
\begin{align*}
V_N&\to\RR,\\(a_n)_{n\in\NN}&\mapsto\sum_{n=1}^N\lr{\frac{1}{n}\sum_{k=n}^N a_k^2}^{\frac{1}{2}}
\end{align*}
to $\Delta_N$ attains its supremum at one of the vertices of $\Delta_N$.
Similarly to the discussion in \cite[Section~2.1]{FoucartRa2013}, the vertices of $\Delta_N$ are precisely those points for which all but one of the defining inequalities are actually equalities.
This means that for each of the $N$ vertices of $\Delta_N$, there is a number $k_0\in\setn{1,\ldots,N}$ such that $a_1=\ldots =a_{k_0}> a_{k_0+1}=\ldots=a_N=0$.
Taking $\sum_{n=1}^Na_n=1$ into account, we obtain $a_1=\ldots =a_{k_0}=\frac{1}{k_0}$ and
\begin{equation*}
\sum_{n=1}^N\lr{\frac{1}{n}\sum_{k=n}^N a_k^2}^{\frac{1}{2}}=\sum_{n=1}^{k_0}\lr{\frac{1}{n}\sum_{k=n}^{k_0} \frac{1}{k_0^2}}^{\frac{1}{2}}=\sum_{n=1}^{k_0}n^{-\frac{1}{2}}\lr{\frac{k_0-n+1}{k_0^2}}^{\frac{1}{2}}.
\end{equation*}
Therefore
\begin{align*}
&\norel\sup\setcond{\frac{\sum_{n=1}^\infty\lr{\frac{1}{n}\sum_{k=n}^\infty a_k^2}^{\frac{1}{2}}}{\sum_{n=1}^\infty a_n}}{a_1\geq a_2\geq\ldots\geq 0,a_{N+1}=0,\sum_{n=1}^\infty a_n<\infty}\\
&=\sup\setcond{\sum_{n=1}^N\lr{\frac{1}{n}\sum_{k=n}^N a_k^2}^{\frac{1}{2}}}{(a_n)_{n\in\NN}\in\Delta_N}\\
&=\sup\setcond{\sum_{n=1}^{k_0}n^{-\frac{1}{2}}\lr{\frac{k_0-n+1}{k_0^2}}^{\frac{1}{2}}}{k_0\in\setn{1,\ldots,N}}.
\end{align*}
Clearly, this quantity is monotonically increasing in $N$ because not only the set $\setcond{(a_n)_{n\in\NN}}{a_1\geq a_2\geq\ldots\geq 0,a_{N+1}=0}$ is, but also $\Delta_N$ and its vertex set are.
We will show that the sequence $\lr{\sum_{n=1}^{k_0}n^{-\frac{1}{2}}\lr{\frac{k_0-n+1}{k_0^2}}^{\frac{1}{2}}}_{k_0\in\NN}$ is bounded above and the supremum is $\frac{\piup}{2}$.
Then we automatically know that
\begin{align*}
\frac{\piup}{2}&=\sup_{k_0\in\NN}\sum_{n=1}^{k_0}n^{-\frac{1}{2}}\lr{\frac{k_0-n+1}{k_0^2}}^{\frac{1}{2}}\\
&=\sup_{N\in\NN}\sup\setcond{\sum_{n=1}^{k_0}n^{-\frac{1}{2}}\lr{\frac{k_0-n+1}{k_0^2}}^{\frac{1}{2}}}{k_0\in\setn{1,\ldots,N}}\\
&=\sup_{N\in\NN}\sup\setcond{\frac{\sum_{n=1}^\infty\lr{\frac{1}{n}\sum_{k=n}^\infty a_k^2}^{\frac{1}{2}}}{\sum_{n=1}^\infty a_n}}{a_1\geq a_2\geq\ldots\geq 0,a_{N+1}=0}.
\end{align*}
The function $g:(0,k_0+1)\to [0,\infty)$, $g(t)\defeq t^{-\frac{1}{2}}\lr{\frac{k_0-t+1}{k_0^2}}^{\frac{1}{2}}$ is monotonically decreasing, so
\begin{equation}
f_{k_0}\defeq\int_1^{k_0+1}g(t)\dd t\leq\sum_{n=1}^{k_0} g(n) \leq g(1)+\int_1^{k_0} g(t)\dd t\eqdef h_{k_0}.\label{eq:sum-integral-estimate}
\end{equation}

For the computation of the antiderivative $\int t^{-\frac{1}{2}}(k_0-t+1)^{\frac{1}{2}}\dd t$, the change of variables $u=\sqrt{\frac{t}{k_0-t+1}}$ or $t=\frac{u^2(k_0+1)}{u^2+1}$ yields $\frac{\dd t}{\dd u}=\frac{2u(k_0+1)}{(u^2+1)^2}$ and
\begin{align*}
&\int t^{-\frac{1}{2}}(k_0-t+1)^{\frac{1}{2}}\dd t\\
&=2(k_0+1)\int \frac{1}{(u^2+1)^2}\dd u\\
%&=2(k_0+1)\lr{\frac{u}{2(u^2+1)}+\frac{1}{2}\int\frac{1}{u^2+1}\dd u}\\
&=2(k_0+1)\lr{\frac{u}{2(u^2+1)}+\frac{1}{2}\arctan(u)}\\
%&=\frac{u(k_0+1)}{u^2+1}+(k_0+1)\arctan(u)\\
%&=\frac{t^{\frac{1}{2}}(k_0+1)}{(k_0-t+1)^{\frac{1}{2}}\lr{\frac{t}{k_0-t+1}+1}}+(k_0+1)\arctan\lr{\lr{\frac{t}{k_0-t+1}}^{\frac{1}{2}}}\\
&=t^{\frac{1}{2}}(k_0-t+1)^{\frac{1}{2}}+(k_0+1)\arctan\lr{\lr{\frac{t}{k_0-t+1}}^{\frac{1}{2}}}.
\end{align*}
Plugging in the integration bounds, we arrive at
\begin{align*}
\int_1^{k_0+1} t^{-\frac{1}{2}}(k_0-t+1)^{\frac{1}{2}}\dd t&=(k_0+1)\frac{\piup}{2}-k_0^{\frac{1}{2}}-(k_0+1)\arctan\lr{k_0^{-\frac{1}{2}}},\\
%\int_0^{k_0} t^{-\frac{1}{2}}(k_0-t+1)^{\frac{1}{2}}\dd t&=k_0^{\frac{1}{2}}+(k_0+1)\arctan\lr{k_0^{\frac{1}{2}}}\\
\int_1^{k_0} t^{-\frac{1}{2}}(k_0-t+1)^{\frac{1}{2}}\dd t&=\lr{\arctan\lr{k_0^{\frac{1}{2}}}-\arctan\lr{k_0^{-\frac{1}{2}}}}(k_0+1).
\end{align*}
It follows that $\lim_{k_0\to\infty}f_{k_0}=\lim_{k_0\to\infty}h_{k_0}=\frac{\piup}{2}$.
Now, if we can show that the sequences $(f_{k_0})_{k_0\in\NN}$ and $(h_{k_0})_{k_0\in\NN}$ are monotonically increasing, then \eqref{eq:sum-integral-estimate} implies
\begin{align*}
\frac{\piup}{2}&=\lim_{k_0\to\infty}f_{k_0}=\sup_{k_0\in\NN}f_{k_0}\leq \sup_{k_0\in\NN}\sum_{n=1}^{k_0}n^{-\frac{1}{2}}\lr{\frac{k_0-n+1}{k_0^2}}^{\frac{1}{2}}\\
&\leq\sup_{k_0\in\NN}h_{k_0}=\lim_{k_0\to\infty}h_{k_0}=\frac{\piup}{2}
\end{align*}
and we are done.
Indeed, one has
\begin{align*}
\frac{\p}{\p x}\int_1^{x+1} \lr{\frac{x-t+1}{tx^2}}^{\frac{1}{2}}\dd t&=\frac{\p}{\p x}\frac{(x+1)\frac{\piup}{2}-x^{\frac{1}{2}}-(x+1)\arctan\lr{x^{-\frac{1}{2}}}}{x}\\
&=\frac{x^{\frac{1}{2}}-\arctan\lr{x^{\frac{1}{2}}}}{x^2}>0
\end{align*}
and
\begin{align*}
&\norel\frac{\p}{\p x}\lr{x^{-\frac{1}{2}}+\int_1^x \lr{\frac{x-t+1}{tx^2}}^{\frac{1}{2}}\dd t}\\
&=\frac{\p}{\p x}\lr{x^{-\frac{1}{2}}+\lr{\arctan\lr{x^{\frac{1}{2}}}-\arctan\lr{x^{-\frac{1}{2}}}}\frac{x+1}{x}}\\
&=\frac{x^{\frac{1}{2}}-4\arctan\lr{x^{\frac{1}{2}}}+\piup}{2x^2}>0
\end{align*}
for all $x\geq 1$.
For the latter claim, note that the function $(0,\infty)\ni x\mapsto \sqrt{x}-4\arctan\lr{\sqrt{x}}+\piup$ has a global minimizer at $x=3$ with minimum $\sqrt{3}-\frac{\piup}{3}>0$, which can be read off the signs of its derivative $x\mapsto\frac{x-3}{2\sqrt{x}(x+1)}$.

\emph{\qquad Step 2. We prove the claim for all sequences.}
It remains to show that
\begin{equation*}
\frac{\sum_{n=1}^\infty\lr{\frac{1}{n}\sum_{k=n}^\infty a_k^2}^{\frac{1}{2}}}{\sum_{n=1}^\infty a_n}\leq\frac{\piup}{2}
\end{equation*}
for all sequences $(a_n)_{n\in\NN}$ with $a_1\geq a_2\geq \ldots\geq 0$ and $\sum_{n=1}^\infty a_n\leq\infty$.
Let $\eps>0$.
Then there exists $N_1,N_2,N_3\in\NN$ such that
\begin{equation*}
\sum_{n=N+1}^\infty\lr{\frac{1}{n}\sum_{k=n}^\infty a_k^2}^{\frac{1}{2}}<\frac{\eps}{2}
\end{equation*}
for all $N> N_1$, and
\begin{equation*}
\sum_{k=\floor{\frac{N}{2}}}^\infty a_k<\frac{\eps}{8}
\end{equation*}
for all $N> N_2$, and
\begin{equation*}
\sqrt{N}\floor{\frac{N}{2}}^{-\frac{1}{2}}\leq 2
\end{equation*}
for all $N>N_3$.
Also, for fixed $N\in\NN$ and $M\geq N+1$, apply \cite[Proposition~2.3]{FoucartRa2013} with $p=1$, $q=2$, and $s=\frac{N}{2}$ to the sequence $(a_n)_{n\geq\floor{\frac{N}{2}}}$ to obtain
\begin{equation*}
\lr{\sum_{k=N+1}^\infty a_k^2}^{\frac{1}{2}}\leq \floor{\frac{N}{2}}^{-\frac{1}{2}}\sum_{k=\floor{\frac{N}{2}}}^\infty a_k.
\end{equation*}
For $N>\max\setn{N_1,N_2,N_3}$, we conclude
\begin{align*}
\sum_{n=1}^N\lr{\frac{1}{n}\sum_{k=N+1}^\infty a_k^2}^{\frac{1}{2}}&\leq\sum_{n=1}^N n^{-\frac{1}{2}} \floor{\frac{N}{2}}^{-\frac{1}{2}}\sum_{k=\floor{\frac{N}{2}}}^\infty a_k\leq\int_0^N x^{-\frac{1}{2}}\dd x\floor{\frac{N}{2}}^{-\frac{1}{2}}\sum_{k=\floor{\frac{N}{2}}}^\infty a_k\\
&=2\sqrt{N}\floor{\frac{N}{2}}^{-\frac{1}{2}}\sum_{k=\floor{\frac{N}{2}}}^\infty a_k\leq 4\sum_{k=\floor{\frac{N}{2}}}^\infty a_k<\frac{\eps}{2}
\end{align*}
and
\begin{align*}
&\norel\sum_{n=1}^\infty\lr{\frac{1}{n}\sum_{k=n}^\infty a_k^2}^{\frac{1}{2}}\\
&=\sum_{n=1}^N\lr{\frac{1}{n}\sum_{k=n}^\infty a_k^2}^{\frac{1}{2}}+\sum_{n=N+1}^\infty\lr{\frac{1}{n}\sum_{k=n}^\infty a_k^2}^{\frac{1}{2}}\\
&\leq \sum_{n=1}^N\lr{\frac{1}{n}\sum_{k=n}^N a_k^2}^{\frac{1}{2}}+\sum_{n=1}^N\lr{\frac{1}{n}\sum_{k=N+1}^\infty a_k^2}^{\frac{1}{2}}+\sum_{n=N+1}^\infty\lr{\frac{1}{n}\sum_{k=n}^\infty a_k^2}^{\frac{1}{2}}\\
&<\frac{\piup}{2}\sum_{n=1}^Na_n+\frac{\eps}{2}+\frac{\eps}{2}\\
&=\frac{\piup}{2}\sum_{n=1}^\infty a_n+\eps.
\end{align*}
Taking the limit $\eps\downarrow 0$ proves the assertion.
\end{proof}

The precise values for $c_1(q)$ from \cite[Theorem~3]{Bennett1988} are illustrated in \cref{fig:left-stechkin-q}.
\begin{figure}
\begin{center}
\begin{tikzpicture}[line cap=round,line join=round,>=stealth,x=4.0cm,y=0.5cm]
\draw[black!50] (0,1)--(1,1);
\draw[black!50] (0,3)--(1,3);
\draw[black!50] (0,5)--(1,5);
\draw[black!50] (0.5,0)--(0.5,5);
\draw[black!50] (1,0)--(1,5);
\draw[line width=1pt,->] (0,0)--(1.2,0);
\draw[line width=1pt,->] (0,0)--(0,6.5);
\draw (0,6.5) node[left]{$c_1(q)$};
\draw (1.2,0) node[below]{$1/q$};
\draw (0,1) node[left]{$1$};
\draw (0,3) node[left]{$3$};
\draw (0,5) node[left]{$5$};
\draw (1,0) node[below]{$1$};
\draw (0.5,0) node[below]{$0.5$};

\draw[color=black,line width=1pt,smooth,samples=200,domain=0.001:0.8] plot(\x,{3.141592653589793*(\x)/sin((3.141592653589793*(\x))*180/pi)});
\end{tikzpicture}
\end{center}
\caption{The function $1/q\mapsto \frac{\piup}{q\sin(\frac{\piup}{q})}$.\label{fig:left-stechkin-q}}
\end{figure}
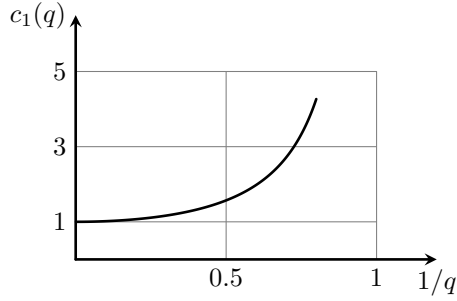

\begin{Bem}
Our proof of \cref{result:left-stechkin} already emphasizes the importance of finite sequences for our considerations which will be encountered again in the proof of \cref{result:right-weak-q}.
Let us draw a geometric picture.
Fix $N\in \NN$ and $q\in (1,\infty)$.
The quantities
\begin{align*}
\mnorm{a}_1&=\sum_{n=1}^N\abs{a_n}\\
\intertext{and}
\gamma_N(a)&\defeq\sum_{n=1}^N\lr{\frac{1}{n}\sum_{k=n}^N \abs{a_k}^q}^{\frac{1}{q}}
\end{align*}
define norms on $\RR^N$, whose closed unit balls shall be denoted by $B_1^N$ and $B_{1,q}^N$, respectively.
\cref{fig:q-balls} illustrates $B_{1,q}^N$ for several values of $q$.
\begin{figure}[H]
\begin{center}
\begin{tikzpicture}[line cap=round,line join=round,>=stealth,x=2cm,y=2cm]
\clip(-1.1,-0.75) rectangle (1.1,0.75);

\draw [dotted,domain=-1.1:1.1] plot(\x,\x);
\draw [dotted,domain=-1.1:1.1] plot(\x,{-\x});
\draw [dotted,domain=-1.1:1.1] plot(\x,0);
\draw [dotted,domain=-1.1:1.1] plot(0,\x);

%q=1
\draw (0,0.70710678)--(1,0)--(0,-0.70710678)--(-1,0)--cycle;

%q=infty
\draw (1,0)--(0.5,0.5)--(-0.5,0.5)--(-1,0)--(-0.5,-0.5)--(0.5,-0.5)--cycle;

%q=2
\draw plot[samples=250,domain=-1:1)] (\x,{-sqrt(4-2*(\x)^2)+sqrt(2)})--plot[samples=250,domain=1:-1)] (\x,{sqrt(4-2*(\x)^2)-sqrt(2)});

%q=3
\begin{scope}[yscale=1,xscale=1]
\input{q3.tex}
\end{scope}
\begin{scope}[yscale=1,xscale=-1]
\input{q3.tex}
\end{scope}
\begin{scope}[yscale=-1,xscale=-1]
\input{q3.tex}
\end{scope}
\begin{scope}[yscale=-1,xscale=1]
\input{q3.tex}
\end{scope}

%q=6
\begin{scope}[yscale=1,xscale=1]
\input{q6.tex}
\end{scope}
\begin{scope}[yscale=1,xscale=-1]
\input{q6.tex}
\end{scope}
\begin{scope}[yscale=-1,xscale=-1]
\input{q6.tex}
\end{scope}
\begin{scope}[yscale=-1,xscale=1]
\input{q6.tex}
\end{scope}

%q=1.4
\begin{scope}[yscale=1,xscale=1]
\input{q1-4.tex}
\end{scope}
\begin{scope}[yscale=1,xscale=-1]
\input{q1-4.tex}
\end{scope}
\begin{scope}[yscale=-1,xscale=-1]
\input{q1-4.tex}
\end{scope}
\begin{scope}[yscale=-1,xscale=1]
\input{q1-4.tex}
\end{scope}
\end{tikzpicture}
\end{center}\caption{The unit balls of the norms $\gamma_N$ for $N=2$ and $q\in\setn{1,1.4,2,3,6,\infty}$.\label{fig:q-balls}}
\end{figure}
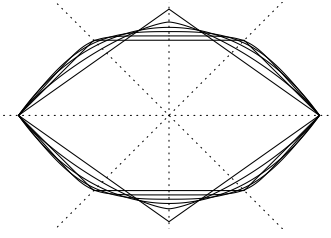

The inequality
\begin{equation}
\frac{1}{c_1(q)}\sum_{n=1}^N\lr{\frac{1}{n}\sum_{k=n}^N a_k^q}^{\frac{1}{q}}\leq\sum_{n=1}^N a_n\leq C_1(q)\sum_{n=1}^N\lr{\frac{1}{n}\sum_{k=n}^N a_k^q}^{\frac{1}{q}}\label{eq:finite-dimensional-general-stechkin}
\end{equation}
for all $a=(a_1,\ldots,a_N)\in\RR^N$ with $a_1\geq \ldots a_N\geq 0$ then translates to the following chain of set inclusions of convex bodies:
\begin{equation}
\frac{1}{c_1(q)}(B_1^N\cap K_N)\subset B_{1,q}^N\cap K_N\subset C_1(q)(B_1^N\cap K_N).\label{eq:optimal-containment-stechkin}
\end{equation}
Here $K_N\defeq \setcond{(a_1,\ldots,a_N)}{a_1\geq \ldots a_N\geq 0}$.
For understanding the shape of the convex bodies $B_{1,q}^N$, we note that norm $\gamma_N$ is the pointwise sum of the functions $\gamma_{n,N}:\RR^n\to\RR$ given by $\gamma_{n,N}(a)=\lr{\frac{1}{n}\sum_{k=n}^N a_k^q}^{\frac{1}{q}}$ for $n\in\setn{1,\ldots,N}$.
With ${}^\circ$ denoting the polar set with respect to the standard inner product, it follows that $B_{1,q}^N=\lr{B_{1,N}^\circ+\ldots+B_{N,N}^\circ}^\circ$, which is similar to the construction of the harmonic mean of convex bodies introduced by Firey in \cite{Firey1961}.
For $N=q=2$, the chain of set inclusions stated in \eqref{eq:optimal-containment-stechkin} with the optimal constants is illustrated in \cref{fig:optimal-containment-stechkin}.
This figure may also be used to convince oneself that for the left-hand side inequality in \eqref{eq:finite-dimensional-general-stechkin}, it is relevant to have $a\in K_N$, and thus monotonicity is also relevant for the left-hand side inequality in \eqref{eq:general-stechkin-q-intro}.
\end{Bem}

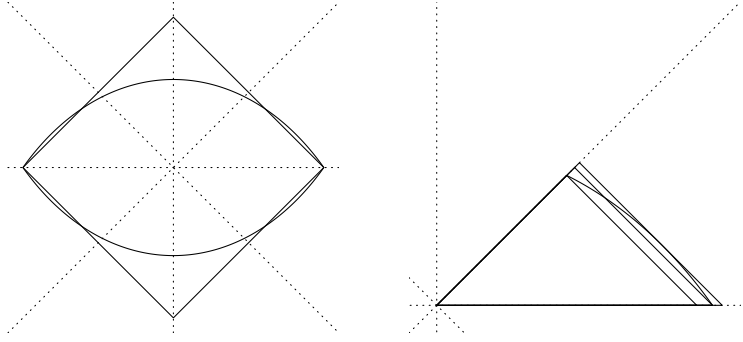
\begin{figure}[h!]
\begin{center}
\subfigure{
\begin{tikzpicture}[line cap=round,line join=round,>=stealth,x=2cm,y=2cm]
\clip(-1.1,-1.1) rectangle (1.1,1.1);

\draw [dotted,domain=-1.1:1.1] plot(\x,\x);
\draw [dotted,domain=-1.1:1.1] plot(\x,{-\x});
\draw [dotted,domain=-1.1:1.1] plot(\x,0);
\draw [dotted,domain=-1.1:1.1] plot(0,\x);
\draw plot[samples=250,domain=-1:1)] (\x,{-sqrt(4-2*(\x)^2)+sqrt(2)})--plot[samples=250,domain=1:-1)] (\x,{sqrt(4-2*(\x)^2)-sqrt(2)});
\draw (0,1)--(1,0)--(0,-1)--(-1,0)--cycle;
\end{tikzpicture}
}\qquad
\subfigure{
\begin{tikzpicture}[line cap=round,line join=round,>=stealth,x=3.666cm,y=3.666cm]
\clip(-0.1,-0.1) rectangle (1.1,1.1);

\draw [dotted,domain=-1.5:1.5] plot(\x,\x);
\draw [dotted,domain=-1.5:1.5] plot(\x,{-\x});
\draw [dotted,domain=-1.5:1.5] plot(\x,0);
\draw [dotted,domain=-1.5:1.5] plot(0,\x);
\draw (0,0)--plot[samples=250,domain={sqrt(2)/3}:1)] (\x,{sqrt(4-2*(\x)^2)-sqrt(2)})--cycle;
\draw (0,0)--(0.5,0.5)--(1,0)--cycle;
\draw[scale={2*sqrt(2)/3}] (0,0)--(0.5,0.5)--(1,0)--cycle;
\draw[scale={sqrt(2)*(sqrt(3)-1)}] (0,0)--(0.5,0.5)--(1,0)--cycle;
\end{tikzpicture}
}
\end{center}\caption{The unit balls of the norms $\gamma_N$ and $\mnorm{\cdot}_1$ for $N=q=2$ (left) and an optimally scaled version of their intersections with the cone $K_N$ (right).\label{fig:optimal-containment-stechkin}}
\end{figure}

\subsection{On the optimal upper constant}
In this section, we improve the known upper bounds on $C_1(q)$ for $1<q\leq\frac{2+\ln(2)}{2-\ln(2)}$.
The following result adapts Stechkin's proof technique in \cite[\rus{Лемма}~3]{Stechkin1951} and produces upper bounds for $C_1(q)$ from auxiliary sequences whose entrywise inverses in $\ell_{q^\prime}$ where $q^\prime$ is the Hölder conjugate of $q$.
\begin{Prop}\label{result:right-stechkin-q}
Let $1<q<\infty$, set $q^\prime=\frac{q}{q-1}$, and assume that $b=(b_n)_{n\in\NN_0}$ is a strictly monotonically increasing sequence with $b_0=0$ and $\sum_{n=1}^\infty\frac{1}{b_n^{q^\prime}}<\infty$.
Then for all $(a_n)_{n\in\NN}\in\ell_q$, we have
\begin{equation*}
\sum_{n=1}^\infty a_n\leq C_b(q)\sum_{n=1}^\infty\lr{\frac{1}{n}\sum_{k=n}^\infty a_k^q}^{\frac{1}{q}}
\end{equation*}
with $C_b(q)\defeq\sup\setcond{\lr{n^{\frac{q^\prime}{q}}(b_n-b_{n-1})^{q^\prime}\sum_{k=n}^\infty\frac{1}{b_k^{q^\prime}}}^{\frac{1}{q^\prime}}}{n\in\NN}$.
\end{Prop}
\begin{proof}
Consider
\begin{align*}
&\norel\sum_{n=1}^\infty a_n=\sum_{n=1}^\infty\frac{a_n}{b_n}b_n=\sum_{n=1}^\infty\frac{a_n}{b_n}\sum_{k=1}^n(b_k-b_{k-1})\\
&=\sum_{n=1}^\infty(b_n-b_{n-1})\sum_{k=n}^\infty\frac{a_k}{b_k}\leq\sum_{n=1}^\infty(b_n-b_{n-1})\lr{\frac{1}{n}\sum_{k=n}^\infty a_k^q}^{\frac{1}{q}}\lr{n^{\frac{q^\prime}{q}}\sum_{k=n}^\infty\frac{1}{b_k^{q^\prime}}}^{\frac{1}{q^\prime}}\\
&=\sum_{n=1}^\infty\lr{\frac{1}{n}\sum_{k=n}^\infty a_k^q}^{\frac{1}{q}}\lr{n^{\frac{q^\prime}{q}}(b_n-b_{n-1})^{q^\prime}\sum_{k=n}^\infty\frac{1}{b_k^{q^\prime}}}^{\frac{1}{q^\prime}}\leq C_b(q)\sum_{n=1}^\infty\lr{\frac{1}{n}\sum_{k=n}^\infty a_k^q}^{\frac{1}{q}}
\end{align*}
where $C_b(q)$ is chosen as stated in the assertion.
\end{proof}
We investigate the choice $b_k=(k(k+1))^p$ for $p\in (\frac{1}{2q^\prime},1]$.
If we set
\begin{equation*}
A_n\defeq n^{\frac{q^\prime}{q}}((n(n+1))^p-(n(n-1))^p)^{q^\prime}\sum_{k=n}^\infty\frac{1}{(k(k+1))^{q^\prime p}}
\end{equation*}
for $n\in\NN$, then $C(p,q)\defeq\sup\setconds{A_n^\frac{1}{q^\prime}}{n\in\NN}$ the constant defined in \cref{result:right-stechkin-q} for our particular choice of the sequence $(b_n)_{n\in\NN}$.
\begin{Lem}\label{result:upper-bound-q}
Let $p>\frac{1}{2q^\prime}$.
For all $n\in\NN$, we have
\begin{equation}
A_n\leq\lr{\frac{(n+1)^p-(n-1)^p}{n^{p-1}}}^{q^\prime}\frac{1}{2q^\prime p-1}.\label{eq:upper-bound-q}
\end{equation}
\end{Lem}
\begin{proof}
The assertion follows from
\begin{align*}
\sum_{k=n}^\infty\frac{1}{b_k^{q^\prime}}&=\sum_{k=n}^\infty\frac{1}{(k(k+1))^{q^\prime p}}=\sum_{k=n}^\infty\lr{\int_k^{k+1}x^{-2}\dd x}^{q^\prime p}\\
&\leq \sum_{k=n}^\infty\int_k^{k+1}x^{-2q^\prime p}\dd x=\frac{1}{2q^\prime p-1}n^{-2q^\prime p+1},
\end{align*}
for which $p>\frac{1}{2q^\prime}$ is crucial.
\end{proof}

We can say even more about the right-hand side of \eqref{eq:upper-bound-q}.
\begin{Lem}\label{result:monotonicity}
Let $0<p\leq 1$.
The sequence $(A_n^\prime)_{n\in\NN}$ defined by
\begin{equation*}
A_n^\prime\defeq\frac{(n+1)^p-(n-1)^p}{n^{p-1}}
\end{equation*}
is monotonically decreasing.
\end{Lem}
\begin{proof}
For $p=1$ the claim is trivial.
Otherwise, consider the functions $f_1,g_1:(0,\infty)\to\RR$ defined by $g_1(x)=x^p$ and
\begin{equation*}
f_1(x)=\frac{g_1(x+1)-g_1(x-1)}{x^{p-1}}.
\end{equation*}
Note that $f_1(n)=A_n^\prime$.
We will show that $f_1$ is monotonically decreasing on $[2,\infty)$ and that $A_1^\prime>A_2^\prime$.
The Taylor expansion of $g_1$ at $x\geq 2$ is given by
\begin{equation*}
g_1(x+h)=\sum_{k=0}^n \frac{h^k}{k!}x^{p-k} \prod_{m=0}^{k-1}(p-m)+R(x,h,n),
\end{equation*}
and the corresponding approximation error is
\begin{equation*}
R(x,h,n)\defeq \int_x^{x+h}\frac{(x+h-t)^n}{n!}t^{p-n-1}\prod_{m=0}^n(p-m)\dd t.
\end{equation*}
Next, note that
\begin{align*}
\abs{R(x,1,n)}&\leq\int_x^{x+1}\frac{\abs{x+1-t}^n}{n!}t^{p-n-1}\prod_{m=0}^n\abs{p-m}\dd t\\
&\leq\int_x^{x+1}t^{p-n-1}\dd t=\frac{1}{p-n}((x+1)^{p-n}-x^{p-n})\leq\frac{1}{n-p}\\
\intertext{and}
\abs{R(x,-1,n)}&\leq\int_{x-1}^x\frac{(x-1-t)^n}{n!}t^{p-n-1}\prod_{m=0}^n(p-m)\dd t\\
&\leq\int_{x-1}^x t^{p-n-1}\dd t=\frac{1}{p-n}(x^{p-n}-(x-1)^{p-n})\leq\frac{1}{n-p}.\\
\end{align*}
This follows from $\frac{1}{n!}\prod_{m=0}^n\abs{p-m}=p\prod_{m=1}^n\frac{m-p}{m}\leq 1$ and $\abs{x+h+t}^n \leq 1$ for $h\in\setn{-1,1}$ and $t$ between $x$ and $x+h$.
As $\lim_{n\to\infty}\frac{-1}{p-n}=0$, we know that $\lim_{n\to\infty}R(x,h,n)=0$ for $h\in\setn{-1,1}$, and we may write
\begin{align*}
g_1(x+1)&=\sum_{k=0}^\infty\frac{1}{k!}x^{p-k} \prod_{m=0}^{k-1}(p-m)\\
\intertext{and}
-g_1(x-1)&=\sum_{k=0}^\infty\frac{(-1)^{k+1}}{k!}x^{p-k} \prod_{m=0}^{k-1}(p-m).
\end{align*}
Setting $h_n\defeq \frac{1}{(2n-1)!}\prod_{m=0}^{2n-2}(p-m)$ for $n\in\NN$ and noticing $h_n>0$, this gives
\begin{align*}
g_1(x+1)-g_1(x-1)&=2\sum_{n=1}^\infty\frac{1}{(2n-1)!}x^{p-(2n-1)} \prod_{m=0}^{2n-2}(p-m)\\
&=2px^{p-1}+2\sum_{n=2}^\infty\frac{1}{(2n-1)!}x^{p-(2n-1)} \prod_{m=0}^{2n-2}(p-m)\\
&=2px^{p-1}+2\sum_{n=2}^\infty h_n x^{p-(2n-1)}.
\end{align*}
As a function of $x$, the expression
\begin{equation*}
f_1(x)=\frac{g_1(x+1)-g_1(x-1)}{x^{p-1}}=2p+\sum_{n=2}^\infty\frac{h_n}{x^{2n-2}}
\end{equation*}
is thus monotonically decreasing on $[2,\infty)$.
In order to show that
\begin{equation}
A_1^\prime=2^p\geq \frac{3^p-1}{2^{p-1}}=A_2^\prime \label{eq:monotonicity-start}
\end{equation}
for all $p\in (0,1]$, consider the functions $f_2,g_2:\RR\to\RR$ defined by $f_2(x)=3^x-1$ and $g_2(x)=2^{2x-1}$.
Then $f_2^\prime(x)=3^x\ln(3)$, $g_2^\prime(x)=4^x\ln(2)$, $f(1)=g(1)$, and $f_2,g_2,f_2^\prime$, and $g_2^\prime$ are monotonically increasing.
Therefore $f_1^\prime(1)>g_2^\prime(1)$ shows that $f(x)\leq g(x)$ for all $x\leq 1$ (with equality only for $x=1$).
This implies \eqref{eq:monotonicity-start}.
\end{proof}
In \cref{result:improved-right-stechkin-q}, we will show how \cref{result:upper-bound-q,result:monotonicity} and an in some sense optimal choice of $p$ in $b_n=(n(n+1))^p$ yield $C_1(2)\leq \sqrt{\frac{\ee\ln(2)}{\sqrt{2}}}$ which is already an improvement to Stechkin's $C_1(2)\leq\frac{2}{\sqrt{3}}$.
More detailed analysis for a specific choice of $p$ in this construction yields further improvement when $q=2$, coinciding with de Bruijn's result \cite[p.~174]{DeBruijn1958} in the first two decimal places.
\begin{Kor}\label{result:right-stechkin-upper-bound}
The minimal constant $C_1(2)>0$ for which
\begin{equation*}
\sum_{n=1}^\infty a_n\leq C_1(2)\sum_{n=1}^\infty\lr{\frac{1}{n}\sum_{k=n}^\infty a_k^2}^{\frac{1}{2}}
\end{equation*}
holds for all $(a_n)_{n\in\NN}\in\ell_2$ is at most $1.1086983$.
\end{Kor}
\begin{proof}
For $N\in\NN$ with $N\geq 2$, \cref{result:upper-bound-q,result:monotonicity} give 
\begin{align}
C(p,2)&\leq \max\setn{\max\limits_{n=1,\ldots,N-1}\sqrt{A_n},\sup\limits_{n\geq N}\frac{1}{\sqrt{4p-1}}\frac{(n+1)^p-(n-1)^p}{n^{p-1}}}\nonumber\\
&=\max\setn{\max\limits_{n=1,\ldots,N-1}\sqrt{A_n},\frac{1}{\sqrt{4p-1}}\frac{(N+1)^p-(N-1)^p}{N^{p-1}}}.\label{eq:improved-constant}
\end{align}
The last expression evaluated at $N=100$ and $p=0.88$ can be bounded above by $1.1086983$.
For the computation of $A_1,\ldots,A_{100}$, the series $\sum_{k=n}^\infty\frac{1}{(k(k+1))^{2p}}$ have been truncated to $\sum_{k=n}^M \frac{1}{(k(k+1))^{2p}}$ with $M=2\cdot 10^5$.
The proof of \cref{result:upper-bound-q} then shows $\sum_{k=M}^\infty\frac{1}{(k(k+1))^{2p}}\leq \frac{1}{4p-1}M^{-4p+1}$.
For $n\in\NN$, we also have
\begin{equation*}
n((n(n+1))^p-(n(n-1))^p)^2=n^{4p-1}\lr{\frac{(n+1)^p-(n-1)^p}{n^{p-1}}}^2\leq n^{4p-1}
\end{equation*}
by \cref{result:monotonicity}.
The truncation error is therefore at most $\frac{1}{4p-1}N^{4p-1}M^{-4p+1}\approx 1.9055\cdot 10^{-9}$, which can be neglected in the computation of the maximum in \eqref{eq:improved-constant}.
\end{proof}

For $p=1$, \cref{result:upper-bound-q} gives $A_n^{\frac{1}{q^\prime}}\leq 2(2q^\prime-1)^{-\frac{1}{q^\prime}}$, and \cref{result:right-stechkin-q} yields
\begin{equation}
\sum_{n=1}^\infty a_n\leq 2(2q^\prime-1)^{-\frac{1}{q^\prime}}\sum_{n=1}^\infty\lr{\frac{1}{n}\sum_{k=n}^\infty a_k^q}^{\frac{1}{q}}\label{eq:stechkin-choice-bound-q}
\end{equation}
This extends the bound obtained by Levin and Stechkin \cite[{\rus{Д}}.61]{HardyLiPo1948} for $\frac{5}{3}\leq q< 3$ to arbitrary $1<q<\infty$.
The $q=2$ case $C_1(2)\leq \frac{2}{\sqrt{3}}$ has been addressed again in \cite[\rus{Лемма}~3]{Stechkin1951}.
For $1<q<\frac{2+\ln(2)}{2-\ln(2)}$, we can achieve better bounds by choosing the parameter $p$ optimally in $b_k=(k(k+1))^p$.
\begin{Satz}\label{result:improved-right-stechkin-q}
Let $1<q\leq\frac{2+\ln(2)}{2-\ln(2)}$.
The minimal constant $C_1(q)>0$ for which
\begin{equation*}
\sum_{n=1}^\infty a_n\leq C_1(q)\sum_{n=1}^\infty\lr{\frac{1}{n}\sum_{k=n}^\infty a_k^q}^{\frac{1}{q}}
\end{equation*}
holds for all $(a_n)_{n\in\NN}\in\ell_q$ is at most $\lr{\frac{\ee\ln(2)}{\sqrt{2}}}^{\frac{1}{q^\prime}}$.
\end{Satz}
\begin{proof}
Choose $p\in (\frac{1}{2q^\prime},1]$ and set $b_k=(k(k+1))^p$ in \cref{result:right-stechkin-q}.
Then \cref{result:upper-bound-q,result:monotonicity} show that
\begin{equation*}
\sum_{n=1}^\infty a_n\leq \frac{2^p}{(2q^\prime p-1)^{\frac{1}{q^\prime}}}\sum_{n=1}^\infty\lr{\frac{1}{n}\sum_{k=n}^\infty a_k^q}^{\frac{1}{q}}
\end{equation*}
for all $(a_n)_{n\in\NN}\in\ell_q$.
For fixed $q\in (1,\frac{2+\ln(2)}{2-\ln(2)}]$, we find a minimizer of $p\mapsto\frac{2^p}{(2q^\prime p-1)^{\frac{1}{q^\prime}}}$.
Through the change of variables $\lambda\defeq q^\prime p$, this expression becomes $\lr{\frac{2^\lambda}{2\lambda-1}}^{\frac{1}{q^\prime}}$.
The latter is minimized at $\lambda=\frac{2+\ln(2)}{\ln(4)}$, so $\frac{1}{2q^\prime}<p=\frac{\lambda}{q^\prime}\leq 1$, and the minimum is $\frac{\ee\ln(2)}{\sqrt{2}}$.
\end{proof}

The conclude this section by comparing the bounds from \eqref{eq:stechkin-choice-bound-q} and \cref{result:improved-right-stechkin-q} to those from the literature.

Copson \cite[Theorem~2.3]{Copson1928} shows that $C_1(q)\leq q^{\frac{1}{q}}$.
This result is also reported by Hardy, Littlewood, and Pólya \cite[Theorem~345]{HardyLiPo1934}.
The bound from \eqref{eq:stechkin-choice-bound-q} improves the one from \cite[Theorem~2.3]{Copson1928} when $2(2q^\prime-1)^{-\frac{1}{q^\prime}}<q^{\frac{1}{q}}$.
As
\begin{align*}
\lim_{q\to 1}q^{\frac{1}{q}}=\lim_{q\to\infty}q^{\frac{1}{q}}&=1,\\
\lim_{q\to 1}2(2q^\prime-1)^{-\frac{1}{q^\prime}}=\lim_{q\to\infty}2(2q^\prime-1)^{-\frac{1}{q^\prime}}&=2,
\end{align*}
and 
$2(2q^\prime-1)^{-\frac{1}{q^\prime}}<q^{\frac{1}{q}}$ at $q=2$, there are real numbers $q_1$ and $q_2$ such that $1<q_1<q_2$ and $2(2q^\prime-1)^{-\frac{1}{q^\prime}}<q^{\frac{1}{q}}$ for all $q\in(q_1,q_2)$.
Furthermore, the function $q\mapsto q^{\frac{1}{q}}$ is monotonically increasing on $(1,\ee)$ and monotonically decreasing on $(\ee,\infty)$.
Also, for $q_3\approx 1.7718$, the function $q\mapsto 2(2q^\prime-1)^{-\frac{1}{q^\prime}}$ is monotonically decreasing on $(1,q_3)$ and monotonically increasing on $(q_3,\infty)$.
Thus for $q$ sufficiently close to $1$ or $\infty$, the bound from \cite[Theorem~2.3]{Copson1928} is smaller than the one from \eqref{eq:stechkin-choice-bound-q}.
It turns out that $q_1$ and $q_2$ can be chosen such that $2(2q^\prime-1)^{-\frac{1}{q^\prime}}<q^{\frac{1}{q}}$ \emph{if and only if} $q\in(q_1,q_2)$.
Analytical expressions for $q_1$ and $q_2$ are not available through the inequality $2(2q^\prime-1)^{-\frac{1}{q^\prime}}<q^{\frac{1}{q}}$.
However, we have $q_1\approx 1.3229$ and $q_2\approx 4.4124$.

An improvement to \cite[Theorem~2.3]{Copson1928} is reported by Levin and Stechkin in their appendix to the Russian 1948 edition \cite{HardyLiPo1948} of Hardy, Littlewood, and Pólya's monograph.
Levin and Stechkin's bound \cite[{\rus{Д}}.61]{HardyLiPo1948} translates to our notation as
\begin{equation}
C_1(q)\begin{cases}\leq 2^{\frac{1}{q}-2}\lr{3-\frac{1}{q}}q(2-\frac{1}{q})^{\frac{1}{q}-1},&1<q<\frac{5}{3},\\
\leq 2\lr{2\frac{q}{q-1}-1}^{-\frac{q-1}{q}},&\frac{5}{3}\leq q<3,\\
=(q-1)^\frac{1}{q},&3\leq q<\infty.\end{cases}\label{eq:levin-bound}
\end{equation}
Note that at first glance, \eqref{eq:levin-bound} is not what is stated in \cite[{\rus{Д}}.61]{HardyLiPo1948} or its English translation \cite[D.61]{LevinSt1960}.
The mismatch is $2^{\frac{1}{q}-2}\lr{3-\frac{1}{q}}q(2-\frac{1}{q})^{\frac{1}{q}-1}$ versus $2^{\frac{1}{q}-2}\lr{3-\frac{1}{q}}q(1-\frac{1}{q})^{\frac{1}{q}-1}$ in the $1<q<\frac{5}{3}$ case but the latter would not be an improvement over \cite[Theorem~345]{HardyLiPo1934}, and it would not fit in as a special case of the two-parameter inequality \cite[{\rus{Д}}.62]{HardyLiPo1948} or \cite[D.62(v)]{LevinSt1960}, regardless of the discrepancies between the formulas in the different versions and regardless of the fact that it is not the $r=-p$ special case as claimed by Levin and Stechkin but actually the $r=+p$ one.

The bound from \eqref{eq:stechkin-choice-bound-q} is larger than the one from \cite[{\rus{Д}}.61]{HardyLiPo1948} when $3<q<\infty$.
(In this case the latter result provides the optimal constant.)
The bounds coincide for $\frac{5}{3}\leq q\leq 3$.
In the remaining case $1<q<\frac{5}{3}$, our bound from \eqref{eq:stechkin-choice-bound-q} is monotonically decreasing in $q$ while the Levin and Stechkin's bound \cite[{\rus{Д}}.61]{HardyLiPo1948} is monotonically increasing, reversing orders at the boundaries.
Thus the bounds coincide at most once for $1<q<\frac{5}{3}$, and they do for $q_4\approx 1.3725$, which means that the bound from \eqref{eq:stechkin-choice-bound-q} is smaller for $q_4<q\leq\frac{5}{3}$ and larger for $1<q<q_4$ than the one from \cite[{\rus{Д}}.61]{HardyLiPo1948}.

By construction, the bound from \cref{result:improved-right-stechkin-q} outperforms the one from \eqref{eq:stechkin-choice-bound-q}.
Moreover, \cref{result:improved-right-stechkin-q} is an improvement over \cite[{\rus{Д}}.61]{HardyLiPo1948} for $1<q\leq \frac{2+\ln(2)}{2-\ln(2)}\approx 2.0608$.

Gao \cite[Theorem~1]{Gao2011} provides further improvement by showing that $C_1(q)=(q-1)^{\frac{1}{q}}$ not only for $3\leq q<\infty$ as stated in \cite[{\rus{Д}}.61]{HardyLiPo1948}, but even for $q_0\leq q<\infty$, where $q_0\approx 2.8855$ is a solution of the equation
\begin{equation*}
2^{\frac{1}{q-1}}\lr{(q-1)^{\frac{q}{q-1}}-(q-1)}-\lr{1+\frac{3-q}{2}}^{\frac{q}{q-1}}=0.
\end{equation*}

De Bruijn \cite[p.~174]{DeBruijn1958} reports that $C_1(2)=1.1064957714$ \enquote{with an error of at most $9$ units at the last decimal place.}
This is an improvement over Levin and Stechkin's bound $C_1(2)\leq\frac{2}{\sqrt{3}}\approx 1.1547$ established in \cite[{\rus{Д}}.61]{HardyLiPo1948}, and over our bound $C_1(2)\leq \sqrt{\frac{\ee\ln(2)}{\sqrt{2}}}\approx 1.1542$ from \cref{result:improved-right-stechkin-q}.
A visualization of the various upper bounds on $C_1(q)$ is given in \cref{fig:right-stechkin-q}.

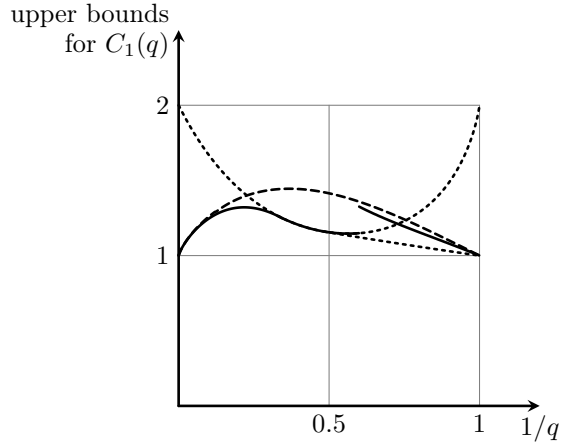
\begin{figure}[h!]
\begin{center}
\begin{tikzpicture}[line cap=round,line join=round,>=stealth,x=4.0cm,y=2.0cm]
\draw[black!50] (0,1)--(1,1);
\draw[black!50] (0,2)--(1,2);
\draw[black!50] (1,0)--(1,2);
\draw[black!50] (0.5,0)--(0.5,2);
\draw[line width=1pt,->] (0,0)--(1.2,0);
\draw[line width=1pt,->] (0,0)--(0,2.5);
\draw (1.2,0) node[below]{$1/q$};
\draw (0,2.6) node[left]{upper bounds};
\draw (0,2.4) node[left]{for $C_1(q)$};
\draw (0,1) node[left]{$1$};
\draw (0,2) node[left]{$2$};
\draw (1,0) node[below]{$1$};
\draw (0.5,0) node[below]{$0.5$};
\draw[densely dashed,line width=1pt,smooth,samples=200,domain=1.8713492058682709E-6:1] plot(\x,{(\x)^(-(\x))});
\draw[dotted,line width=1pt,smooth,samples=200,domain=0:0.9999982192917657] plot(\x,{2*((-(\x)-1)/((\x)-1))^((\x)-1)});
\draw[dotted,line width=1pt,smooth,samples=200,domain=0.4852511696625:1] plot(\x,{(2.718281828459045*ln(2)/sqrt(2))^(1-(\x))});
%\draw[line width=2pt,smooth,samples=100,domain=0.0001:0.346552] plot(\x,{((1.0-(\x))/(\x))^((\x))});
%\draw[line width=2pt,color=red,smooth,samples=100,domain=0.6:1] plot(\x,{2.0^(1.0/(1.0/(\x))-2.0)*(3.0-1.0/(1.0/(\x)))*1.0/(\x)*(2.0-1.0/(1.0/(\x)))^(1.0/(1.0/(\x))-1.0)});

\draw[line width=1pt,smooth,samples=100,domain=0.0001:0.346552] plot(\x,{((1.0-(\x))/(\x))^((\x))});
\draw[line width=1pt,smooth,samples=200,domain=1/3:0.6] plot(\x,{2*((-(\x)-1)/((\x)-1))^((\x)-1)});
\draw[line width=1pt,smooth,samples=100,domain=0.6:1] plot(\x,{2.0^(1.0/(1.0/(\x))-2.0)*(3.0-1.0/(1.0/(\x)))*1.0/(\x)*(2.0-1.0/(1.0/(\x)))^(1.0/(1.0/(\x))-1.0)});
\end{tikzpicture}
\end{center}
\caption{The upper bounds on $C_1(q)$ given by Copson \cite[Theorem~2.3]{Copson1928} (dashed line), Levin and Stechkin \cite[{\rus{Д}}.61]{HardyLiPo1948} and Gao \cite[Theorem~1]{Gao2011} (solid line), and \eqref{eq:stechkin-choice-bound-q} and \cref{result:improved-right-stechkin-q} from the paper at hand (dotted line).\label{fig:right-stechkin-q}}
\end{figure}

\section{The weak discrete Stechkin inequality}\label{chap:weak}
Here we compute the optimal constants $c_{1,\infty}(q)$ and $C_{1,\infty}(q)>0$, for which the inequality
\begin{equation}
\frac{1}{c_{1,\infty}(q)}\sup_{n\in\NN}n\lr{\frac{1}{n}\sum_{k=n}^\infty a_k^q}^{\frac{1}{q}}\leq\sup_{n\in\NN} na_n\leq C_{1,\infty}(q)\sup_{n\in\NN}n\lr{\frac{1}{n}\sum_{k=n}^\infty a_k^q}^{\frac{1}{q}}\label{eq:general-weak-l1-q}
\end{equation}
holds true for all sequences $(a_n)_{n\in\NN}$ with $a_1\geq a_2\geq\ldots\geq 0$, when $1<q<\infty$.
With the modification indicated in \cref{chap:intro}, inequality \eqref{eq:general-weak-l1-q} holds true also for $q=\infty$.
The monotonicity assumption on $(a_n)_{n\in\NN}$ gives $\sup\setcond{a_k}{k\geq n}=a_n$, so $c_{1,\infty}(\infty)=C_{1,\infty}(\infty)=1$.
For $q=1$, we have $C_{1,\infty}(1)=1$ because
\begin{equation*}
\sup_{n\in\NN} na_n = \sup_{n\in\NN}\sum_{k=1}^n a_n\leq\sum_{k=1}^\infty a_k=\sup_{n\in\NN}\sum_{k=n}^\infty a_k,
\end{equation*}
which holds as an equality when $a_n=0$ for all $n\geq 2$.

Notable results in the direction of \eqref{eq:general-weak-l1-q} are the inequalities 
\begin{equation*}
n^{1-1/q}\lr{\sum_{k=n}^\infty a_k^q}^{1/q}\leq\sum_{n=1}^\infty a_n
\end{equation*}
and
\begin{equation*}
n^{1-1/q}\lr{\sum_{k=n+1}^\infty a_k^q}^{1/q}\leq (q-1)^{-1/q}\sup_{n\in\NN}n a_n
\end{equation*}
proved in \cite[{\rus{Лемма}~IV.2.1}]{Temlyakov1986} and \cite[Proposition~2.11]{FoucartRa2013}, respectively.

The following estimate on the Riemann zeta-function is required in our calculation of $c_{1,\infty}(q)$.
\begin{Lem}\label{result:zeta-estimate}
Let $1<q<\infty$.
Then $\zeta(q)>\frac{1}{q-1}+\frac{1}{2}$.
\end{Lem}
\begin{proof}
Consider
\begin{align*}
\zeta(q)-\frac{1}{q-1}&=\zeta(q)-\int_1^\infty x^{-q}\dd x>\zeta(q)-\sum_{n=1}^\infty \frac{n^{-q}+(n+1)^{-q}}{2}\\
&=\zeta(q)-\frac{2\zeta(q)-1}{2}=\frac{1}{2}.
\end{align*}
\end{proof}
The estimate in \cref{result:zeta-estimate} is not best possible.
In fact, the constant $\frac{1}{2}$ can be replaced by the Euler--Mascheroni constant, see \cite[(2.1.16)]{Titchmarsh1986}.
Nonetheless this estimate enables the computation of the precise constant on the left-hand side of \eqref{eq:general-weak-l1-q}.
\begin{Satz}\label{result:left-weak-q}
Let $1<q<\infty$.
The minimal constant $c_{1,\infty}(q)>0$ for which
\begin{equation*}
\sup_{n\in\NN}n\lr{\frac{1}{n}\sum_{k=n}^\infty a_k^q}^{\frac{1}{q}}\leq c_{1,\infty}(q)\sup_{n\in\NN} na_n
\end{equation*}
holds for all $(a_n)_{n\in\NN}$ with $a_1\geq a_2\geq \ldots\geq 0$ is $c_{1,\infty}(q)=\zeta(q)^{\frac{1}{q}}$.
\end{Satz}
\begin{proof}
The supremum
\begin{equation}
\sup\setcond{\sup_{n\in\NN}n^{1-\frac{1}{q}}\lr{\sum_{k=n}^\infty a_k^q}^{\frac{1}{q}}}{a_1\geq a_2\geq\ldots\geq 0,\sup_{n\in\NN}na_n=1}\label{eq:left-weak-q}
\end{equation}
is attained at the sequence $(a_n)_{n\in\NN}$ defined by $a_n=\frac{1}{n}$ for all $n\in\NN$.
Indeed, for any sequence $(a_n)_{n\in\NN}$ with $a_1\geq a_2\geq\ldots\geq 0$ and $\sup_{n\in\NN}na_n=1$, we have $0\leq a_n\leq \frac{1}{n}$ for all $n\in\NN$.
Therefore, $\lr{\sum_{k=n}^\infty a_k^q}^{\frac{1}{q}}\leq \lr{\sum_{k=n}^\infty\frac{1}{k^q}}^{\frac{1}{q}}$ for all $n\in\NN$ with equality if and only if $a_n=\frac{1}{n}$ for all $n\in\NN$.
Also, we have $\sup_{n\in\NN} n\frac{1}{n}=1$, and \eqref{eq:left-weak-q} evaluates to $\sup_{n\in\NN}n^{1-\frac{1}{q}}\lr{\sum_{k=n}^\infty\frac{1}{k^q}}^{\frac{1}{q}}=\zeta(q)^\frac{1}{q}$.
Note that the supremum $\sup_{n\in\NN}n^{1-\frac{1}{q}}\lr{\sum_{k=n}^\infty\frac{1}{k^q}}^{\frac{1}{q}}$ is attained at $n=1$ because
\begin{equation*}
\sum_{k=n}^\infty k^{-p}=n^{-p}+\sum_{k=n+1}^\infty k^{-p}\leq n^{-p}+\int_n^\infty x^{-p} \dd x=n^{-p}+\frac{n^{-p+1}}{p-1}
\end{equation*}
and \cref{result:zeta-estimate} give
\begin{equation*}
n^{1-\frac{1}{q}}\lr{\sum_{k=n}^\infty k^{-q}}^{\frac{1}{q}}\leq n^{1-\frac{1}{q}}\lr{n^{-q}+\frac{n^{-q+1}}{q-1}}^{\frac{1}{q}}=\lr{n^{-1}+\frac{1}{q-1}}^{\frac{1}{q}}\leq\zeta(q)^{\frac{1}{q}}.
\end{equation*}
\end{proof}

The result of \cref{result:left-weak-q} is depicted in \cref{fig:left-weak-q}.
\begin{figure}[H]
\begin{center}
\begin{tikzpicture}[line cap=round,line join=round,>=stealth,x=4cm,y=0.4cm]
\draw[black!50] (0,1)--(1,1);
\draw[black!50] (0,3)--(1,3);
\draw[black!50] (0,5)--(1,5);
\draw[black!50] (0,7)--(1,7);
\draw[black!50] (1,0)--(1,7);
\draw[black!50] (0.5,0)--(0.5,7);
\draw[line width=1pt,->] (0,0)--(1.2,0);
\draw[line width=1pt,->] (0,0)--(0,8.5);
\draw (0,8.5) node[left]{$c_{1,\infty}(q)$};
\draw (1.2,0) node[below]{$1/q$};
\draw (0,1) node[left]{$1$};
\draw (0,3) node[left]{$3$};
\draw (0,5) node[left]{$5$};
\draw (0,7) node[left]{$7$};
\draw (1,0) node[below]{$1$};
\draw (0.5,0) node[below]{$0.5$};

\draw[line width=1pt] (0.0,1.0)--
(0.005,1.0)--
(0.01,1.0)--
(0.015,1.0)--
(0.02,1.0)--
(0.025,1.0000000000000226)--
(0.03,1.000000000002772)--
(0.035,1.0000000000877436)--
(0.04,1.0000000011921402)--
(0.045,1.0000000091983545)--
(0.05,1.0000000476980802)--
(0.055,1.0000001850818052)--
(0.06,1.0000005774192455)--
(0.065,1.0000015224257561)--
(0.07,1.0000035156445122)--
(0.075,1.0000072995774074)--
(0.08,1.0000138990278353)--
(0.085,1.0000246395485761)--
(0.09,1.00004114915099)--
(0.095,1.0000653456402255)--
(0.1,1.0000994130277145)--
(0.105,1.000145770683597)--
(0.11,1.0002070385534834)--
(0.115,1.00028600115726)--
(0.12,1.0003855724091784)--
(0.125,1.0005087626642717)--
(0.13,1.0006586488617304)--
(0.135,1.000838348217534)--
(0.14,1.0010509956107634)--
(0.145,1.0012997245946804)--
(0.15,1.0015876518258218)--
(0.155,1.0019178646237688)--
(0.16,1.0022934113351023)--
(0.165,1.0027172941646423)--
(0.17,1.0031924641455774)--
(0.17500000000000002,1.003721817940272)--
(0.18,1.004308196190177)--
(0.185,1.0049543831628027)--
(0.19,1.00566310747371)--
(0.195,1.0064370436904317)--
(0.2,1.007278814652266)--
(0.20500000000000002,1.0081909943644811)--
(0.21,1.0091761113475035)--
(0.215,1.0102366523410784)--
(0.22,1.011375066280374)--
(0.225,1.012593768475681)--
(0.23,1.0138951449399822)--
(0.23500000000000001,1.0152815568194622)--
(0.24,1.016755344891186)--
(0.245,1.0183188340999547)--
(0.25,1.0199743381128827)--
(0.255,1.021724163875755)--
(0.26,1.0235706161598421)--
(0.265,1.025516002091707)--
(0.27,1.027562635661772)--
(0.275,1.0297128422101014)--
(0.28,1.0319689628900943)--
(0.28500000000000003,1.034333359112661)--
(0.29,1.036808416975005)--
(0.295,1.039396551679445)--
(0.3,1.0421002119488114)--
(0.305,1.0449218844458672)--
(0.31,1.0478640982050131)--
(0.315,1.0509294290852134)--
(0.32,1.0541205042536848)--
(0.325,1.0574400067104421)--
(0.33,1.0608906798642803)--
(0.335,1.06447533217126)--
(0.34,1.0681968418472092)--
(0.34500000000000003,1.0720581616662104)--
(0.35000000000000003,1.0760623238575031)--
(0.355,1.0802124451137016)--
(0.36,1.084511731723737)--
(0.365,1.0889634848444465)--
(0.37,1.0935711059253177)--
(0.375,1.0983381023014929)--
(0.38,1.103268092970815)--
(0.385,1.108364814571407)--
(0.39,1.1136321275770669)--
(0.395,1.1190740227286176)--
(0.4,1.1246946277202734)--
(0.405,1.1304982141611117)--
(0.41000000000000003,1.136489204832839)--
(0.41500000000000004,1.1426721812662477)--
(0.42,1.149051891660076)--
(0.425,1.155633259167418)--
(0.43,1.1624213905763814)--
(0.435,1.1694215854133962)--
(0.44,1.1766393454994264)--
(0.445,1.1840803849913384)--
(0.45,1.1917506409428797)--
(0.455,1.1996562844221013)--
(0.46,1.2078037322246504)--
(0.465,1.2161996592251925)--
(0.47000000000000003,1.2248510114122932)--
(0.47500000000000003,1.2337650196554526)--
(0.48,1.2429492142566394)--
(0.485,1.252411440342664)--
(0.49,1.2621598741590774)--
(0.495,1.2722030403310505)--
(0.5,1.282549830161864)--
(0.505,1.293209521045333)--
(0.51,1.3041917970746908)--
(0.515,1.3155067709372643)--
(0.52,1.3271650071917156)--
(0.525,1.3391775470328038)--
(0.53,1.3515559346575727)--
(0.535,1.3643122453567251)--
(0.54,1.3774591154657645)--
(0.545,1.391009774322387)--
(0.55,1.4049780783897365)--
(0.555,1.4193785477195675)--
(0.56,1.4342264049453448)--
(0.5650000000000001,1.449537617012906)--
(0.5700000000000001,1.4653289398758156)--
(0.5750000000000001,1.4816179664041036)--
(0.58,1.498423177778994)--
(0.585,1.5157639986727556)--
(0.59,1.5336608565422614)--
(0.595,1.5521352453976272)--
(0.6,1.5712097944437777)--
(0.605,1.5909083420334806)--
(0.61,1.6112560154157942)--
(0.615,1.6322793168146628)--
(0.62,1.6540062164292146)--
(0.625,1.6764662530110317)--
(0.63,1.6996906427451826)--
(0.635,1.723712397242212)--
(0.64,1.7485664515388193)--
(0.645,1.774289803107054)--
(0.65,1.800921662987175)--
(0.655,1.828503620289802)--
(0.66,1.8570798214608453)--
(0.665,1.8866971658705762)--
(0.67,1.9174055194791082)--
(0.675,1.9492579485481105)--
(0.68,1.9823109756168835)--
(0.685,2.0166248602449834)--
(0.6900000000000001,2.05226390734911)--
(0.6950000000000001,2.0892968063358435)--
(0.7000000000000001,2.127797004662143)--
(0.705,2.1678431199518227)--
(0.71,2.209519395370123)--
(0.715,2.2529162036234003)--
(0.72,2.298130605723487)--
(0.725,2.3452669715558523)--
(0.73,2.394437670341071)--
(0.735,2.4457638403088375)--
(0.74,2.4993762483474966)--
(0.745,2.555416252091884)--
(0.75,2.6140368789197046)--
(0.755,2.6754040387046656)--
(0.76,2.739697890000721)--
(0.765,2.8071143827016662)--
(0.77,2.8778670042528325)--
(0.775,2.9521887613343076)--
(0.78,3.0303344347729464)--
(0.785,3.1125831525056604)--
(0.79,3.199241334002337)--
(0.795,3.290646070035368)--
(0.8,3.387169014529114)--
(0.805,3.48922088104553)--
(0.81,3.597256656047816)--
(0.8150000000000001,3.711781665452216)--
(0.8200000000000001,3.833358661461626)--
(0.8250000000000001,3.9626161350260585)--
(0.8300000000000001,4.100258107815061)--
(0.835,4.247075719412016)--
(0.84,4.403961004713113)--
(0.845,4.571923358878118)--
(0.85,4.7521093203535525)--
(0.855,4.94582647712193)--
(0.86,5.154572532267846)--
(0.865,5.3800708730960345)--
(0.87,5.624314403152542)--
(0.875,5.889619961474026)--
(0.88,6.178696430799341)--
(0.885,6.4947307188811)--
(0.89,6.841497323352471)--
(0.895,7.223499372641914);
\end{tikzpicture}
\end{center}
\caption{The function $1/q\mapsto \zeta(q)^{\frac{1}{q}}$.\label{fig:left-weak-q}}
\end{figure}
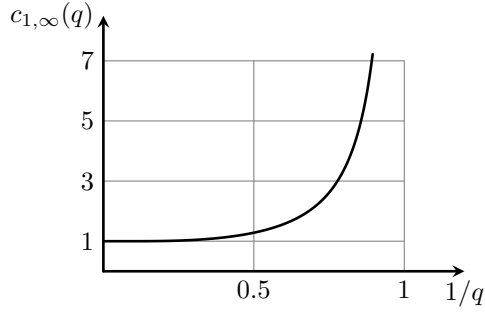

The optimal constant $C_{1,\infty}(q)$ in \eqref{eq:general-weak-l1-q} turns out to be invariant under taking Hölder conjugates.
\begin{Satz}\label{result:right-weak-q}
Let $1<q<\infty$.
The minimal constant $C_{1,\infty}(q)>0$ for which
\begin{equation*}
\sup_{n\in\NN} na_n\leq C_{1,\infty}(q)\sup_{n\in\NN}n\lr{\frac{1}{n}\sum_{k=n}^\infty a_k^q}^{\frac{1}{q}}
\end{equation*}
holds for all $(a_n)_{n\in\NN}$ with $a_1\geq a_2\geq \ldots\geq 0$ is $C_{1,\infty}(q)=\lr{\frac{1}{q}}^{-\frac{1}{q}}\lr{1-\frac{1}{q}}^{-\lr{1-\frac{1}{q}}}$.
\end{Satz}
\begin{proof}
\emph{Step 1. We prove the claim for finite sequences.}
The infimum
\begin{equation*}
\inf\setcond{\sup_{n\in\NN}n\lr{\frac{1}{n}\sum_{k=n}^\infty a_k^q}^{\frac{1}{q}}}{a_1\geq a_2\geq\ldots\geq 0,a_{N+1}=0,\sup_{n\in\NN}n a_n=1}
\end{equation*}
is attained at the sequence $(a_n)_{n\in\NN}$ given by $a_n=\frac{1}{N}$ for $n\in\setn{1,\ldots,N}$.
Its value is thus equal to
\begin{equation*}
\sup_{n=1,\ldots,N}n\lr{\frac{1}{n}\sum_{k=n}^N \frac{1}{N^q}}^{\frac{1}{q}}=\sup_{n=1,\ldots,N}\frac{1}{N}n^{1-\frac{1}{q}}(N-n+1)^{\frac{1}{q}}.
\end{equation*}
One can readily check that the derivative of the function $f:(0,N+1)\to \RR$, $f(x)=\frac{1}{N} x^{1-\frac{1}{q}} (N-x+1)^{\frac{1}{q}}$ is given by $f^\prime(x)=\frac{1}{Nq}x^{-\frac{1}{q}}(N-x+1)^{\frac{1}{q}-1}((N+1)(q-1)-qx)$.
Therefore, $f$ is maximized at $x=(N+1)\lr{1-\frac{1}{q}}$, the maximum is $\frac{N+1}{N}\lr{\frac{1}{q}}^{\frac{1}{q}}\lr{1-\frac{1}{q}}^{1-\frac{1}{q}}$, and $f$ is monotonically increasing for $x<(N+1)\lr{1-\frac{1}{q}}$ and monotonically decreasing for $x>(N+1)\lr{1-\frac{1}{q}}$.
Hence we have shown that
\begin{align*}
&\inf\setcond{\sup_{n\in\NN}n\lr{\frac{1}{n}\sum_{k=n}^\infty a_k^q}^{\frac{1}{q}}}{a_1\geq a_2\geq\ldots\geq 0,a_{N+1}=0,\sup_{n\in\NN}n a_n=1}\\
&=\sup_{n=1,\ldots,N}n^{1-\frac{1}{q}}\lr{\frac{N-n+1}{N^q}}^{\frac{1}{q}}\\
&\leq\sup_{x\in (0,N+1)}x^{1-\frac{1}{q}}\lr{\frac{N-x+1}{N^q}}^{\frac{1}{q}}\\
&=\frac{N+1}{N}\lr{\frac{1}{q}}^{\frac{1}{q}}\lr{1-\frac{1}{q}}^{1-\frac{1}{q}}.
\end{align*}
Taking the infimum over $N\in\NN$, we see that 
\begin{equation*}
\inf_{N\in\NN}\sup_{n=1,\ldots,N} n^{1-\frac{1}{q}}\lr{\frac{N-n+1}{N^q}}^{\frac{1}{q}}\leq \lr{\frac{1}{q}}^{\frac{1}{q}}\lr{1-\frac{1}{q}}^{1-\frac{1}{q}}
\end{equation*}
for finite sequences $(a_n)_{n\in\NN}$ with $a_1\geq a_2\geq\ldots\geq 0$, which also shows that $C_{1,\infty}(q)\leq\lr{\frac{1}{q}}^{-\frac{1}{q}}\lr{1-\frac{1}{q}}^{-\lr{1-\frac{1}{q}}}$.
Next we show that also
\begin{equation*}
\inf_{N\in\NN}\sup_{n=1,\ldots,N} n^{1-\frac{1}{q}}\lr{\frac{N-n+1}{N^q}}^{\frac{1}{q}}\geq \lr{\frac{1}{q}}^{\frac{1}{q}}\lr{1-\frac{1}{q}}^{1-\frac{1}{q}}
\end{equation*}
for finite sequences $(a_n)_{n\in\NN}$ with $a_1\geq a_2\geq\ldots\geq 0$ and we do it for $q\geq 2$ first.
In this case, we have $x_N\defeq 1-\frac{1}{q}-\frac{1}{N+1}\in (0,1)$ for all $N\in\NN$.
The monotonicity properties of the function $g_N:(0,1)\to\RR$, $g_N(x)=f((N+1)x)=\frac{N+1}{N}x^{1-\frac{1}{q}}(1-x)^\frac{1}{q}$ yield
\begin{align*}
\sup\setcond{g_N(x)}{x=\frac{1}{N+1},\ldots,\frac{N}{N+1}}&\geq g_N\lr{\frac{1}{N+1}\floor{(N+1)\lr{1-\frac{1}{q}}}}\\
&\geq g_N(x_N).
\end{align*}
If we can show that $(g_N(x_N))_{N\geq 3}$ is monotonically decreasing, then we know that
\begin{align*}
\inf_{N\geq 3} g_N(x_N)=\lim_{N\to\infty}g_N(x_N)=\lr{\frac{1}{q}}^{\frac{1}{q}}\lr{1-\frac{1}{q}}^{1-\frac{1}{q}}
\end{align*}
and, in turn,
\begin{align*}
&\inf_{N\in\NN}\sup_{n=1,\ldots,N} n^{1-\frac{1}{q}}\lr{\frac{N-n+1}{N^q}}^{\frac{1}{q}}\\
&=\inf\setn{1,\sup\setn{\frac{1}{2}2^{\frac{1}{q}},\frac{1}{2}2^{1-\frac{1}{q}}},\inf_{N\geq 3}\sup_{n=1,\ldots,N} n^{1-\frac{1}{q}}\lr{\frac{N-n+1}{N^q}}^{\frac{1}{q}}}\\
&\geq\inf\setn{1,\sup\setn{\frac{1}{2}2^{\frac{1}{q}},\frac{1}{2}2^{1-\frac{1}{q}}},\inf_{N\geq 3}g_N\lr{x_N}}\\
&\geq\lr{\frac{1}{q}}^{\frac{1}{q}}\lr{1-\frac{1}{q}}^{1-\frac{1}{q}}.
\end{align*}
Indeed, let $a\defeq \frac{1}{q}\in (0,\frac{1}{2}]$, and consider the function $h:(0,1-a)\to\RR$, $h(t)=\lr{\frac{1-a-t}{1-t}}^{1-a}\lr{\frac{a+t}{1-t}}^a$.
Then $g_N(x_N)=h(\frac{1}{N+1})$, and it is sufficient to show that $h$ is monotonically increasing on $(0,\frac{1}{4})$ when $a\leq \frac{1}{2}$.
The derivative
\begin{align*}
h^\prime(t)&=(1-a)\lr{\frac{1-a-t}{1-t}}^{-a}\frac{-a}{(1-t)^2}\lr{\frac{a+t}{1-t}}^a\\
&\qquad+a\lr{\frac{a+t}{1-t}}^{a-1}\frac{1+a}{(1-t)^2}\lr{\frac{1-a-t}{1-t}}^{1-a}\\
&=\frac{1}{(1-t)^2}\lr{\frac{a+t}{1-a-t}}^{a-1}\lr{(1-a)\frac{a+t}{1-a-t}(-a)+a(1+a)}
\end{align*}
is $\geq 0$ if and only if $(1-a)\frac{a+t}{1-a-t}(-a)+a(1+a)\geq 0$ if and only if $a(a+t-1)(a+2t-1)\geq 0$.
For $t\in (0,1-a)$, we have $a(a+t-1)\leq 0$, so $h^\prime(t)\geq 0$ when $a+2t-1\leq 0$.
The latter is fulfilled when $t\leq\frac{1}{4}$.

For $q\leq 2$, we consider the function $(0,1)\ni x\mapsto g_N(1-x)$ instead of $g_N$.
This is the same as $g_N$ for the Hölder conjugate $q^\prime\geq 2$, and this is covered by the first case.

\emph{Step 2. We prove the claim for all sequences.}
It remains to show that
\begin{equation*}
\frac{\sup_{n\in\NN}n^{1-\frac{1}{q}}\lr{\sum_{k=n}^\infty a_k^q}^{\frac{1}{q}}}{\sup_{n\in\NN}n a_n}\geq\lr{\frac{1}{q}}^{\frac{1}{q}}\lr{1-\frac{1}{q}}^{1-\frac{1}{q}}
\end{equation*}
for all sequences $(a_n)_{n\in\NN}$ with $a_1\geq a_2\geq \ldots\geq 0$.
Let $\eps>0$.
Then, for all sequences $(a_n)_{n\in\NN}$ with $a_1\geq a_2\geq\ldots\geq 0$ and $\sup\setcond{na_n}{n\in\NN}=1$, there exists a number $N\in\NN$ such that $a_N\geq\frac{1-\eps}{N}$.
It follows that
\begin{align*}
&\sup_{n\in\NN}n^{1-\frac{1}{q}}\lr{\sum_{k=n}^\infty a_k^q}^{\frac{1}{q}}\geq\sup_{n=1,\ldots,N}n^{1-\frac{1}{q}}\lr{\sum_{k=n}^N \frac{1}{N^q}}^{\frac{1}{q}}(1-\eps)\\
&=\sup_{n=1,\ldots,N}n^{1-\frac{1}{q}}\lr{\frac{N-n+1}{N^q}}^{\frac{1}{q}}(1-\eps)\geq\lr{\frac{1}{q}}^{\frac{1}{q}}\lr{1-\frac{1}{q}}^{1-\frac{1}{q}}(1-\eps).
\end{align*}
Taking the limit $\eps\downarrow 0$ yields the desired inequality.
\end{proof}

The result of \cref{result:right-weak-q} is depicted in \cref{fig:right-weak-q}.
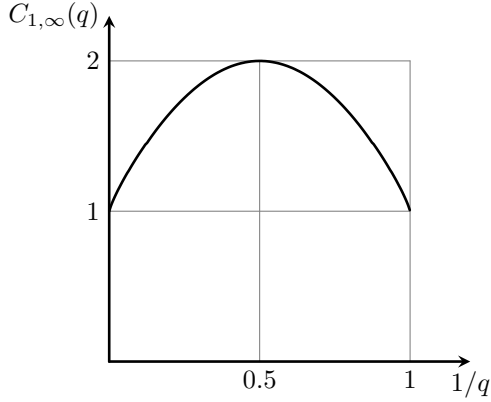
\begin{figure}[H]
\begin{center}
\begin{tikzpicture}[line cap=round,line join=round,>=stealth,x=4cm,y=2cm]
\draw[black!50] (1,0)--(1,2);
\draw[black!50] (0,1)--(1,1);
\draw[black!50] (0,2)--(1,2);
\draw[black!50] (0.5,0)--(0.5,2);
\draw[line width=1pt,->] (0,0)--(1.2,0);
\draw[line width=1pt,->] (0,0)--(0,2.3);
\draw (0,2.3) node[left]{$C_{1,\infty}(q)$};
\draw (1.2,0) node[below]{$1/q$};
\draw (0,1) node[left]{$1$};
\draw (1,0) node[below]{$1$};
\draw (0.5,0) node[below]{$0.5$};
\draw (0,2) node[left]{$2$};

\draw[line width=1pt,smooth,samples=200,domain=2.899999996260959E-8:0.9999990250810828] plot(\x,{1/((\x)^((\x))*(1-(\x))^(1-(\x)))});
\end{tikzpicture}
\end{center}
\caption{The function $1/q\mapsto \lr{\frac{1}{q}}^{-\frac{1}{q}}\lr{1-\frac{1}{q}}^{-\lr{1-\frac{1}{q}}}$.\label{fig:right-weak-q}}
\end{figure}

\section{The continuous Stechkin inequalities}
In this section, we are concerned with the optimal constants $\bar{c}_1(q)$, $\bar{C}_1(q)$, $\bar{c}_{1,\infty}(q)$, and $\bar{C}_{1,\infty}(q)>0$ in the inequalities
\begin{equation*}
\frac{1}{\bar{c}_1(q)}\int_0^\infty \lr{\frac{1}{t}\int_t^\infty f(s)^q \dd s}^{\frac{1}{q}}\dd t\leq \int_0^\infty f(t)\dd t\leq \bar{C}_1(q)\int_0^\infty \lr{\frac{1}{t}\int_t^\infty f(s)^q \dd s}^{\frac{1}{q}}\dd t
\end{equation*}
and
\begin{equation*}
\frac{1}{\bar{c}_{1,\infty}(q)}\sup_{t>0}t\lr{\frac{1}{t}\int_t^\infty f(s)^q\dd s}^{\frac{1}{q}}\leq\sup_{t>0} tf(t)\leq \bar{C}_{1,\infty}(q)\sup_{t>0}t\lr{\frac{1}{t}\int_t^\infty f(s)^q\dd s}^{\frac{1}{q}}
\end{equation*}
for monotonically decreasing functions $f:(0,\infty)\to [0,\infty)$.
In both cases, precise values for the optimal constants are available, either through the literature or shown here.
The proofs in this section are independent of their discrete counterparts, yet there is a strong resemblance in the case of the weak Stechkin inequalities.
(The arguments turn out to be less tedious in the continuous inequality, though.)

\subsection{The strong continuous Stechkin inequality}
Hardy, Littlewood, and Pól\-ya \cite[Theorem~337]{HardyLiPo1934} show that for $1<q<\infty$, the minimal constant $\bar{C}_1(q)>0$ for which
\begin{equation}
\frac{1}{\bar{c}_1(q)}\int_0^\infty \lr{\frac{1}{t}\int_t^\infty f(s)^q \dd s}^{\frac{1}{q}}\dd t\leq \int_0^\infty f(t)\dd t\leq \bar{C}_1(q)\int_0^\infty \lr{\frac{1}{t}\int_t^\infty f(s)^q \dd s}^{\frac{1}{q}}\dd t\label{eq:general-stechkin-continuous}
\end{equation}
holds true for all monotonically decreasing functions $f:(0,\infty)\to [0,\infty)$ is $\bar{C}_1(q)=(q-1)^{\frac{1}{q}}$ when $1<q<\infty$.
With the modification indicated in \cref{chap:intro}, inequality \eqref{eq:general-stechkin-continuous} holds true also for $q=\infty$.
The monotonicity assumption on $f:(0,\infty)\to [0,\infty)$ gives $\sup\setcond{f(s)}{s\geq t}=f(t)$, so $\bar{c}_1(\infty)=\bar{C}_1(\infty)=1$.
For $q=1$, the right-hand side of \eqref{eq:general-stechkin-continuous} holds in the sense that $\int_0^\infty\lr{\frac{1}{t}\int_t^\infty f(s)^q \dd s}^{\frac{1}{q}}\dd t$ diverges when $\int_0^\infty f(t)\dd t$ is finite.
Therefore $\lim_{q\to 1}\bar{C}_1(q)=0$.
For this, choose a function $f:(0,\infty)\to[0,\infty)$ with $\int_0^\infty f(t)\dd t\neq 0$.
Then $F(t)\defeq\int_t^\infty f(s)^q \dd s$ defines a monotonically decreasing function $F:(0,\infty)\to [0,\infty)$ for which there exist $\eps>0$ and $\delta>0$ such that $F(t)\geq \delta$ for all $t\in (0,\eps)$.
It follows that $\int_0^\infty \frac{1}{t}F(t)\dd t\geq \int_0^\eps\frac{1}{t}F(t)\dd t\geq\int_0^\eps\frac{1}{t}\delta\dd t=\infty$.

We complement these results by computing the minimal constant $\bar{c}_1(q)>0$ appearing in \eqref{eq:general-stechkin-continuous}.
\begin{Satz}\label{result:left-stechkin-continuous}
Let $1<q<\infty$.
The minimal constant $\bar{c}_1(q)>0$ for which
\begin{equation*}
\int_0^\infty \lr{\frac{1}{t}\int_t^\infty f(s)^q \dd s}^{\frac{1}{q}}\dd t\leq\bar{c}_1(q)\int_0^\infty f(t)\dd t
\end{equation*}
holds for all monotonically decreasing functions $f:(0,\infty)\to [0,\infty)$ is $\bar{c}_1(q)=\frac{\piup}{q\sin(\frac{\piup}{q})}$.
\end{Satz}

\begin{proof}
Lower bounds on the constant $\bar{c}_1(q)$ are given by the quotients
\begin{equation}
\frac{\int_0^t \lr{\frac{1}{t}\int_t^\infty f(s)^q \dd s}^{\frac{1}{q}}\dd t}{\int_0^\infty f(t)\dd t}\label{eq:left-stechkin-continuous-lower-bound}
\end{equation}
with $f:(0,\infty)\to [0,\infty)$ a monotonically decreasing function.
For $T\in (0,\infty)$, take $f=\frac{1}{T}\chi_{(0,T)}:(0,\infty)\to [0,\infty)$ in \eqref{eq:left-stechkin-continuous-lower-bound}, where $\chi_{(0,T)}$ denotes the function which is $1$ on $(0,T)$ and $0$ on $[T,\infty)$.
Then $\int_0^\infty f(t)\dd t=1$ and
\begin{align*}
&\int_0^\infty \lr{\frac{1}{t}\int_t^\infty f(s)^q \dd s}^{\frac{1}{q}}\dd t=\frac{1}{T}\int_0^\infty \lr{\frac{1}{t}\int_t^\infty \chi_{(0,T)}(s) \dd s}^{\frac{1}{q}}\dd t\\
&=\frac{1}{T}\int_0^T \lr{\frac{1}{t}\int_t^\infty \chi_{(0,T)}(s) \dd s}^{\frac{1}{q}}\dd t=\frac{1}{T}\int_0^T \lr{\frac{T-t}{t}}^{\frac{1}{q}}\dd t\\
&=B\lr{1-\frac{1}{q},1+\frac{1}{q}}=\frac{\piup}{q\sin(\frac{\piup}{q})}
\end{align*}
independently of $T$.
(Here $B(x,y)\defeq\int_0^1 t^{x-1}(1-t)^{y-1}\dd t$ denotes the beta function.)
This shows $\bar{c}_1(q)\geq\frac{\piup}{q\sin(\frac{\piup}{q})}$.

Now fix a function $f:[0,\infty)\to [0,\infty)$ with $\int_0^\infty f(x)\dd x=1$.
For $\eps>0$, let $N\defeq\floor{\frac{1}{\eps}\sup_{t\in [0,\infty)}f(t)}\in \NN\cup\setn{\infty}$.
For $n\in\NN$ with $n\leq N$, let 
\begin{equation*}
T_n\defeq \sup\setcond{T>0}{f(t)\geq \eps n \fall t\in (0,T)},
\end{equation*}
$\lambda_n\defeq \eps T_n$, and $g_n\defeq \frac{1}{\lambda_n}\eps\chi_{[0,T_n]}$.
Then $0\leq h_\eps(t)\defeq\sum_{n=1}^N \lambda_n g_n(t)\leq f(t)$ for all $t\in [0,\infty)$ and $\int_0^\infty g_n(t)\dd t =1$ for all $n$, so $\sum_{n=1}^N\lambda_n=\int_0^\infty \sum_{n=1}^N \lambda_n g_n(t)\dd t\leq \int_0^\infty f(t)\dd t=1$.
From the triangle inequality for integrals, it follows that
\begin{align*}
\int_0^\infty \lr{\frac{1}{t}\int_t^\infty h_\eps(s)^q\dd s}^{\frac{1}{q}}\dd t&\leq \sum_{n=1}^N \lambda_n \int_0^\infty \lr{\frac{1}{t}\int_t^\infty g_n(s)^q\dd s}^{\frac{1}{q}}\dd t\\
&=\sum_{n=1}^N \lambda_n \frac{\piup}{q\sin(\frac{\piup}{q})}\leq\frac{\piup}{q\sin(\frac{\piup}{q})}.
\end{align*}
With the abbreviations $H_\eps(t)\defeq \lr{\int_t^\infty h_\eps(s)^q\dd s}^{\frac{1}{q}}$ and $F(t)\defeq \lr{\int_t^\infty f(s)^q\dd s}^{\frac{1}{q}}$, we have $\lim_{k\to\infty} H_{2^{-k}}(t)=F(t)$ and $0\leq H_{2^{-k}}(t)\leq F(t)$ for all $t\in (0,\infty)$ and $k\in\NN$.
The dominated convergence theorem then yields
\begin{align*}
\int_0^\infty F(t)\dd t&=\int_0^\infty \lim_{k\to\infty}H_{2^{-k}}(t)\dd t=\lim_{k\to\infty}\int_0^\infty H_{2^{-k}}(t)\dd t\leq \frac{\piup}{q\sin(\frac{\piup}{q})}.
\end{align*}
This shows $\bar{c}_1(q)\leq \frac{\piup}{q\sin(\frac{\piup}{q})}$ and completes the proof.
\end{proof}

The precise values for $\bar{c}_1(q)$ and $\bar{C}_1(q)$ given in \cref{result:left-stechkin-continuous} and \cite[Theorem~337]{HardyLiPo1934} are illustrated in \cref{fig:stechkin-continuous}.
\begin{figure}[H]
\begin{center}
\begin{tikzpicture}[line cap=round,line join=round,>=stealth,x=3cm,y=0.554cm]
\draw[black!50] (0,1)--(1,1);
\draw[black!50] (0,3)--(1,3);
\draw[black!50] (0,5)--(1,5);
\draw[black!50] (0.5,0)--(0.5,5);
\draw[black!50] (1,0)--(1,5);
\draw[line width=1pt,->] (0,0)--(1.2,0);
\draw[line width=1pt,->] (0,0)--(0,6.5);
\draw (0,6.5) node[left]{$\bar{c}_1(q)$};
\draw (1.2,0) node[below]{$1/q$};
\draw (0,1) node[left]{$1$};
\draw (0,3) node[left]{$3$};
\draw (0,5) node[left]{$5$};
\draw (1,0) node[below]{$1$};
\draw (0.5,0) node[below]{$0.5$};

\draw[color=black,line width=1pt,smooth,samples=200,domain=0.001:0.8] plot(\x,{3.141592653589793*(\x)/sin((3.141592653589793*(\x))*180/pi)});
\end{tikzpicture}
\qquad
\begin{tikzpicture}[line cap=round,line join=round,>=stealth,x=3cm,y=2cm]
\draw[black!50] (1,0)--(1,1.5);
\draw[black!50] (0,1)--(1,1);
\draw[black!50] (0.5,0)--(0.5,1.5);
\draw[black!50] (0,0.5)--(1,0.5);
\draw[black!50] (0,1.5)--(1,1.5);
\draw[line width=1pt,->] (0,0)--(1.2,0);
\draw[line width=1pt,->] (0,0)--(0,1.8);
\draw (0,1.8) node[left]{$\bar{C}_1(q)$};
\draw (1.2,0) node[below]{$1/q$};
\draw (0,1) node[left]{$1$};
\draw (1,0) node[below]{$1$};
\draw (0.5,0) node[below]{$0.5$};

\draw[line width=1pt,smooth,samples=200,domain=0.001:0.9999996591791471] plot(\x,{(1/(\x)-1)^((\x))});
\end{tikzpicture}
\end{center}
\caption{The function $1/q\mapsto \frac{\piup}{q\sin(\frac{\piup}{q})}$ (left) and $1/q\mapsto (q-1)^{\frac{1}{q}}$ (right).\label{fig:stechkin-continuous}}
\end{figure}
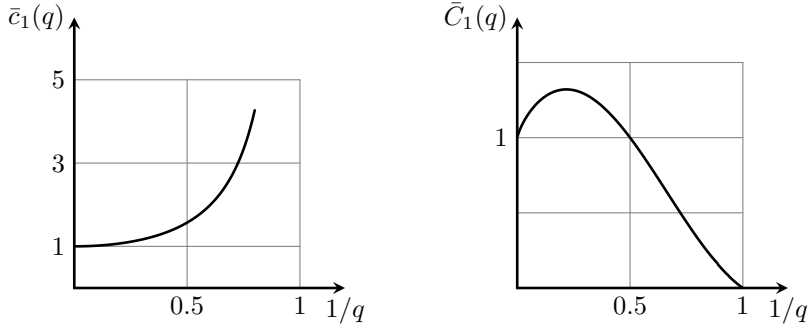

\cref{result:left-stechkin-continuous} and \cite[Theorem~337]{HardyLiPo1934} can be transformed into a result on functions $f:\Omega\to\RR$ defined on a measure space $(\Omega,\mu)$.
For such a function $f$, the assignment $f^\ast(t)\defeq\inf\setcond{s>0}{\mu(\setcond{x\in\Omega}{\abs{f(x)}>s})\leq t}$ defines a function $f^\ast:(0,\infty)\to [0,\infty)$, called the \emph{non-increasing rearrangement} of $f$.
\begin{Kor}
Let $(\Omega,\mu)$ be a measure space and $1<q<\infty$.
The minimal constants $c(q),C(q)>0$ for which
\begin{align*}
\frac{1}{c(q)}\int_0^\infty \lr{\frac{1}{t}\int_t^\infty f^\ast(s)^q \dd s}^{\frac{1}{q}}\dd t&\leq \int_0^\infty \abs{f(x)}\dd \mu(x)\\
&\leq C(q)\int_0^\infty \lr{\frac{1}{t}\int_t^\infty f^\ast(s)^q \dd s}^{\frac{1}{q}}\dd t
\end{align*}
holds for all functions $f:\Omega\to \RR$ are $c(q)=\bar{c}_1(q)$ and $C(q)=\bar{C}_1(q)$.
\end{Kor}
\begin{proof}
Note that $\int_\Omega\abs{f(x)}\dd\mu(x)=\int_0^\infty f^\ast(t)\dd t$ and apply \cref{result:left-stechkin-continuous} and \cite[Theorem~337]{HardyLiPo1934} to $f^\ast$.
\end{proof}

\subsection{The weak continuous Stechkin inequality}
Here we compute the optimal constants $\bar{c}_{1,\infty}(q)$ and $\bar{C}_{1,\infty}(q)>0$, for which the inequality
\begin{equation}
\frac{1}{\bar{c}_{1,\infty}(q)}\sup_{t>0}t\lr{\frac{1}{t}\int_t^\infty f(s)^q\dd s}^{\frac{1}{q}}\leq\sup_{t>0} tf(t)\leq \bar{C}_{1,\infty}(q)\sup_{t>0}t\lr{\frac{1}{t}\int_t^\infty f(s)^q\dd s}^{\frac{1}{q}}\label{eq:general-weak-l1-continuous}
\end{equation}
holds true for all monotonically decreasing functions $f:(0,\infty)\to [0,\infty)$, when $1<q<\infty$.
With the modification indicated in \cref{chap:intro}, inequality \eqref{eq:general-weak-l1-continuous} holds true also for $q=\infty$.
The monotonicity assumption on $f$ gives $\sup\setcond{f(s)}{s\geq t}=f(t)$, so $\bar{c}_{1,\infty}(\infty)=\bar{C}_{1,\infty}(\infty)=1$.
For $q=1$, we have $\bar{C}_{1,\infty}(1)=1$ because
\begin{equation*}
\sup_{t>0} tf(t)=\sup_{t>0}\int_0^t f(t)\dd s\leq\int_0^\infty f(s)\dd s=\sup_{t>0}\int_t^\infty f(s)\dd s.
\end{equation*}

\begin{Satz}\label{result:left-weak-continuous}
Let $1<q<\infty$.
The minimal constant $\bar{c}_{1,\infty}(q)>0$ for which
\begin{equation*}
\sup_{t>0}t\lr{\frac{1}{t}\int_t^\infty f(s)^q\dd s}^{\frac{1}{q}}\leq \bar{c}_{1,\infty}(q)\sup_{t>0} tf(t)
\end{equation*}
holds for all monotonically decreasing functions $f:(0,\infty)\to [0,\infty)$ is $\bar{c}_{1,\infty}(q)=(q-1)^{-\frac{1}{q}}$.
\end{Satz}
\begin{proof}
The supremum of the expression
\begin{equation}
\sup_{t>0}t^{1-\frac{1}{q}}\lr{\int_t^\infty f(s)^q\dd s}^{\frac{1}{q}}\label{eq:left-weak-continuous}
\end{equation}
over the monotonically decreasing functions $f:(0,\infty)\to [0,\infty)$ with $\sup_{t>0}tf(t)=1$ is attained at the function $f:(0,\infty)\to [0,\infty)$ defined by $f(t)=\frac{1}{t}$.
Indeed, for any monotonically decreasing function $f:(0,\infty)\to [0,\infty)$ with $\sup_{t>0}tf(t)=1$, we have $0\leq f(t)\leq\frac{1}{t}$ for all $t>0$.
Therefore, $\lr{\int_t^\infty f(s)^q\dd s}^{\frac{1}{q}}\leq \lr{\int_t^\infty\frac{1}{s^q}\dd s}^{\frac{1}{q}}$ for all $t>0$ with equality if and only if $f(t)=\frac{1}{t}$ for all $t>0$.
Also, we have $\sup_{t>0} t\frac{1}{t}=1$, and \eqref{eq:left-weak-continuous} evaluates to $\sup_{t>0}t^{1-\frac{1}{q}}\lr{\int_t^\infty\frac{1}{s^q}\dd s}^{\frac{1}{q}}=(q-1)^{-\frac{1}{q}}$.
\end{proof}

The optimal constant $\bar{C}_{1,\infty}(q)$ in \eqref{eq:general-weak-l1-continuous} coincides with its discrete counterpart.
\begin{Satz}\label{result:right-weak-continuous}
Let $1<q<\infty$.
The minimal constant $\bar{C}_{1,\infty}(q)>0$ for which
\begin{equation*}
\sup_{t>0} tf(t)\leq \bar{C}_{1,\infty}(q)\sup_{t>0}t\lr{\frac{1}{t}\int_t^\infty f(s)^q\dd s}^{\frac{1}{q}}
\end{equation*}
holds for all monotonically decreasing functions $f:(0,\infty)\to [0,\infty)$ is $\bar{C}_{1,\infty}(q)=\lr{\frac{1}{q}}^{-\frac{1}{q}}\lr{1-\frac{1}{q}}^{-\lr{1-\frac{1}{q}}}$.
\end{Satz}
\begin{proof}
\emph{Step 1. We prove the claim for functions supported by an interval $(0,T)$ with $T>0$.}
The infimum of the expression
\begin{equation*}
\sup_{t>0}t\lr{\frac{1}{t}\int_t^\infty f(s)^q\dd s}^{\frac{1}{q}}
\end{equation*}
over the monotonically decreasing functions $f:(0,\infty)\to [0,\infty)$ with $f(T)=0$ and $\sup_{t>0}tf(t)=1$ is attained at the function $f=\frac{1}{T}\chi_{(0,T)}:(0,\infty)\to [0,\infty)$.
Its value is thus equal to
\begin{equation*}
\sup_{t\in (0,T)}t\lr{\frac{1}{t}\int_t^T\frac{1}{T^q}\dd s}^{\frac{1}{q}}=\sup_{t\in (0,T)}\frac{1}{T}t^{1-\frac{1}{q}}(T-t)^{\frac{1}{q}}.
\end{equation*}
One can readily check that the derivative of the function $F:(0,T)\to \RR$, $F(t)=\frac{1}{T} t^{1-\frac{1}{q}} (T-t)^{\frac{1}{q}}$ is given by $F^\prime(t)=\frac{q(T-t)+T}{q(T-t)T}\lr{\frac{T-t}{t}}^{\frac{1}{q}}$.
Therefore, $F$ is maximized at $t=T\lr{1-\frac{1}{q}}$, and the maximum is $\lr{\frac{1}{q}}^{\frac{1}{q}}\lr{1-\frac{1}{q}}^{1-\frac{1}{q}}$ independently of $T$.

\emph{Step 2. We prove the claim for all functions.}
It remains to show that
\begin{equation*}
\frac{\sup_{t>0}t^{1-\frac{1}{q}}\lr{\int_t^\infty f(s)^q\dd s}^{\frac{1}{q}}}{\sup_{t>0}tf(t)}\geq\lr{\frac{1}{q}}^{\frac{1}{q}}\lr{1-\frac{1}{q}}^{1-\frac{1}{q}}
\end{equation*}
for all monotonically decreasing functions $f:(0,\infty)\to [0,\infty)$.
Let $\eps>0$.
Then, for all monotonically decreasing functions $f:(0,\infty)\to [0,\infty)$ with $\sup\setcond{tf(t)}{t>0}=1$, there exists a number $T\in\NN$ such that $f(T)\geq\frac{1-\eps}{T}$.
It follows that
\begin{align*}
&\sup_{t>0}t^{1-\frac{1}{q}}\lr{\int_t^\infty f(s)^q\dd s}^{\frac{1}{q}}\geq \sup_{t\in (0,T)} t^{1-\frac{1}{q}}\lr{\int_t^T \frac{1}{T^q}\dd s}^{\frac{1}{q}}(1-\eps)\\
&=\sup_{t\in (0,T)}t^{1-\frac{1}{q}}\lr{\frac{T-t}{T^q}}^{\frac{1}{q}}(1-\eps)\geq\lr{\frac{1}{q}}^{\frac{1}{q}}\lr{1-\frac{1}{q}}^{1-\frac{1}{q}}(1-\eps).
\end{align*}
Taking the limit $\eps\downarrow 0$ yields the desired inequality.
\end{proof}

The precise values for $\bar{c}_{1,\infty}(q)$ and $\bar{C}_{1,\infty}(q)$ are illustrated in \cref{fig:weak-l1-continuous}.
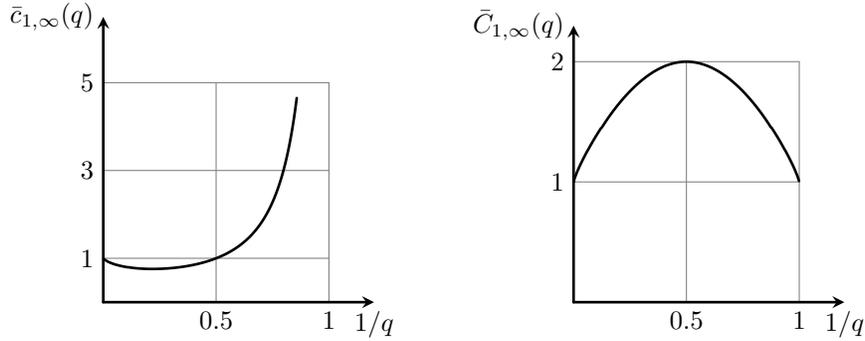
\begin{figure}[H]
\begin{center}
\begin{tikzpicture}[line cap=round,line join=round,>=stealth,x=3cm,y=0.584cm]
\draw[black!50] (0,1)--(1,1);
\draw[black!50] (0,3)--(1,3);
\draw[black!50] (0,5)--(1,5);
\draw[black!50] (0.5,0)--(0.5,5);
\draw[black!50] (1,0)--(1,5);
\draw[line width=1pt,->] (0,0)--(1.2,0);
\draw[line width=1pt,->] (0,0)--(0,6.5);
\draw (0,6.5) node[left]{$\bar{c}_{1,\infty}(q)$};
\draw (1.2,0) node[below]{$1/q$};
\draw (0,1) node[left]{$1$};
\draw (0,3) node[left]{$3$};
\draw (0,5) node[left]{$5$};
\draw (1,0) node[below]{$1$};
\draw (0.5,0) node[below]{$0.5$};

\draw[line width=1pt,smooth,samples=200,domain=0.001:0.86] plot(\x,{((\x)/(1-(\x)))^((\x))});
\end{tikzpicture}
\qquad
\begin{tikzpicture}[line cap=round,line join=round,>=stealth,x=3cm,y=1.6cm]
\draw[black!50] (1,0)--(1,2);
\draw[black!50] (0,1)--(1,1);
\draw[black!50] (0,2)--(1,2);
\draw[black!50] (0.5,0)--(0.5,2);
\draw[line width=1pt,->] (0,0)--(1.2,0);
\draw[line width=1pt,->] (0,0)--(0,2.3);
\draw (0,2.3) node[left]{$\bar{C}_{1,\infty}(q)$};
\draw (1.2,0) node[below]{$1/q$};
\draw (0,1) node[left]{$1$};
\draw (1,0) node[below]{$1$};
\draw (0.5,0) node[below]{$0.5$};
\draw (0,2) node[left]{$2$};

\draw[line width=1pt,smooth,samples=200,domain=2.899999996260959E-8:0.9999990250810828] plot(\x,{1/((\x)^((\x))*(1-(\x))^(1-(\x)))});
\end{tikzpicture}
\end{center}
\caption{The function $1/q\mapsto (q-1)^{-\frac{1}{q}}$ (left) and $1/q\mapsto \lr{\frac{1}{q}}^{-\frac{1}{q}}\lr{1-\frac{1}{q}}^{-\lr{1-\frac{1}{q}}}$ (right).\label{fig:weak-l1-continuous}}
\end{figure}

\section{Applications to sparse approximation}\label{chap:devore}
As mentioned in \cref{chap:intro}, the inequalities \eqref{eq:general-stechkin-q-intro} and \eqref{eq:general-weak-l1-q-intro} play an important role in nonlinear approximation.
More precisely, we will outline the connection of \cite[Theorem~4]{DeVore1998} and our results.
Let $\mathcal{H}$ be an infinite-dimensional separable real Hilbert space with inner product $\skpr{\cdot}{\cdot}_{\mathcal{H}}$ and norm $\mnorm{\cdot}_{\mathcal{H}}$.
The choice of an orthonormal basis $(e_k)_{k\in\NN}$ and Parseval's identity give an isometric isomorphism $\mathcal{H}\to\ell_2$, $f\mapsto (\skpr{f}{e_k}_{\mathcal{H}})_{k\in\NN}$.
Sparse approximation in $\mathcal{H}$ is implemented by defining the approximation error for $f\in\mathcal{H}$ as
\begin{equation*}
E_n(f)_{\mathcal{H}}=\setcond{\mnorm{f-g}_{\mathcal{H}}}{g\in\Sigma_{n-1}(\mathcal{H})}
\end{equation*}
where $\Sigma_{n-1}(\mathcal{H})\defeq\setcond{\sum_{k\in\Lambda} \lambda_k e_k}{\lambda_k\in\RR,\Lambda\subset\NN,\# \Lambda < n}$.
Then, for $\alpha>0$, the approximation space $\mathcal{A}_r^\alpha(\mathcal{H})$ is defined as the set of elements $f\in\mathcal{H}$ for which the quantity
\begin{equation*}
\mnorm{f}_{\mathcal{A}_r^\alpha(\mathcal{H})}\defeq\begin{cases}
\lr{\sum_{n=1}^\infty(n^\alpha E_n(f)_{\mathcal{H}})^r \frac{1}{n}}^{\frac{1}{r}},&0<r<\infty,\\
\sup_{n\in\NN}n^\alpha E_n(f)_{\mathcal{H}},&r=\infty
\end{cases}
\end{equation*}
is finite.
Approximation spaces are then subject to characterizations in terms of Lorentz sequence spaces $\ell_{p,r}$, i.e., the sets of bounded sequences for which the quantity
\begin{equation*}
\mnorm{(f_k)_{k\in\NN}}_{\ell_{p,r}}\defeq\begin{cases}
\lr{\sum_{n\in\NN}(n^{\frac{1}{p}}f_n^\ast)^r\frac{1}{n}}^{\frac{1}{r}},&0<r<\infty,\\
\sup_{n\in\NN} n^{\frac{1}{p}} f_n^\ast,&r=\infty
\end{cases}
\end{equation*}
is finite.
Here $0<p\leq\infty$, and $(f_k^\ast)_{k\in\NN}$ is the non-increasing rearrangement of the sequence $(\abs{f_k})_{k\in\NN}$.
Now, given $f\in\mathcal{H}$ and $f_k\defeq\skpr{f}{e_k}_{\mathcal{H}}$ for $k\in\NN$, DeVore \cite[Theorem~4]{DeVore1998} shows that
\begin{equation}
\mnorm{f}_{\mathcal{A}_r^\alpha(\mathcal{H})}\asymp\mnorm{(f_k)_{k\in\NN}}_{\ell_{\tau,r}},\label{eq:devore}
\end{equation}
meaning that there exist constants $c,C>0$ such that
\begin{equation*}
\frac{1}{c}\mnorm{f}_{\mathcal{A}_r^\alpha(\mathcal{H})}\leq\mnorm{(f_k)_{k\in\NN}}_{\ell_{\tau,r}}\leq C \mnorm{f}_{\mathcal{A}_r^\alpha(\mathcal{H})}.
\end{equation*}
Here $\tau=(\alpha+\frac{1}{2})^{-1}$.
Notable special cases of \eqref{eq:devore} are re-parameterized by the results from \cref{chap:stechkin,chap:weak}.
\begin{Satz}\label{result:devore}
With the definitions above, the following statements are true for the minimal constants $c,C>0$ in the inequality
\begin{equation}
\frac{1}{c}\mnorm{f}_{\mathcal{A}_r^\alpha(\mathcal{H})}\leq\mnorm{(f_k)_{k\in\NN}}_{\ell_{\tau,r}}\leq C\mnorm{f}_{\mathcal{A}_r^\alpha(\mathcal{H})},\label{eq:quasi-norm-equivalence}
\end{equation}
where $\tau=(\alpha+\frac{1}{2})^{-1}$.
\begin{enumerate}[label={(\roman*)},leftmargin=*,align=left,noitemsep]
\item{If $r=\tau$, then $c=c_1(2\alpha+1)^{\alpha+\frac{1}{2}}$ and $C=C_1(2\alpha+1)^{\alpha+\frac{1}{2}}$.\label{devore-strong-q}}
\item{If $r=\infty$, then $c=c_{1,\infty}(2\alpha+1)^{\alpha+\frac{1}{2}}$ and $C=C_{1,\infty}(2\alpha+1)^{\alpha+\frac{1}{2}}$.\label{devore-weak-q}}
\end{enumerate}
\end{Satz}

\begin{proof}
For $r=\tau$, inequality \eqref{eq:devore} becomes
\begin{align*}
\frac{1}{C_1}\lr{\sum_{n=1}^\infty\lr{n^\alpha\lr{\sum_{k=n}^\infty (f_k^\ast)^2}^{\frac{1}{2}}}^\tau \frac{1}{n}}^{\frac{1}{\tau}}&\leq\lr{\sum_{k=1}^\infty (f_k^\ast)^\tau}^{\frac{1}{\tau}}\\
&\leq C_2\lr{\sum_{n=1}^\infty\lr{n^\alpha\lr{\sum_{k=n}^\infty (f_k^\ast)^2}^{\frac{1}{2}}}^\tau \frac{1}{n}}^{\frac{1}{\tau}}.
\end{align*}
Setting $a_k=(f_k^\ast)^\tau$ and $q=\frac{2}{\tau}=2\alpha+1$ gives
\begin{equation*}
\frac{1}{C_1}\lr{\sum_{n=1}^\infty n^{-\frac{1}{q}}\lr{\sum_{k=n}^\infty a_k^q}^{\frac{1}{q}}}^{\frac{1}{\tau}}\leq\lr{\sum_{k=1}^\infty a_k}^{\frac{1}{\tau}}\leq C_2\lr{\sum_{n=1}^\infty n^{-\frac{1}{q}}\lr{\sum_{k=n}^\infty a_k^q}^{\frac{1}{q}}}^{\frac{1}{\tau}}.
\end{equation*}
Raising everything to the $\tau$th power shows \ref{devore-strong-q}.
Similarly, for $r=\infty$, inequality \eqref{eq:devore} becomes
\begin{equation*}
\frac{1}{C_1}\sup_{n\in\NN}n^\alpha\lr{\sum_{k=n}^\infty (f_k^\ast)^2}^{\frac{1}{2}}\leq\sup_{k\in\NN} k^{\frac{1}{\tau}} f_k^\ast\leq C_2\sup_{n\in\NN}n^\alpha\lr{\sum_{k=n}^\infty (f_k^\ast)^2}^{\frac{1}{2}}
\end{equation*}
Raising everything to the $\tau$th power gives
\begin{equation*}
\frac{1}{C_1}\sup_{n\in\NN}n^{\alpha\tau}\lr{\sum_{k=n}^\infty (f_k^\ast)^2}^{\frac{\tau}{2}}\leq\sup_{k\in\NN} k (f_k^\ast)^\tau\leq C_2\sup_{n\in\NN}n^{\alpha\tau}\lr{\sum_{k=n}^\infty (f_k^\ast)^2}^{\frac{\tau}{2}}.
\end{equation*}
Setting $a_k=(f_k^\ast)^\tau$ and $q=\frac{2}{\tau}=2\alpha+1$ shows \ref{devore-weak-q}.
\end{proof}

Stechkin \cite{Stechkin1951} considers the Hilbert space $\mathcal{H}=L_2(\mathbb{T}^d)$ of square-integrable functions on $\mathbb{T}^d = [0,2\pi]^d$.
An orthonormal basis in $L_2(\mathbb{T}^d)$ is given by $e_k(x)\defeq\frac{1}{\sqrt{2\piup}}\exp(\ii kx)$ for $k\in\ZZ$ and $x\in\mathbb{T}^d$. 
Then \eqref{eq:devore} for $r=\tau=1$ shows that the approximation space $\mathcal{A}_1^{1/2}(L_2(\mathbb{T}^d))$ coincides with the Wiener algebra $\mathcal{A}(\mathbb{T}^d)\defeq\setcond{f\in C(\mathbb{T})}{\sum_{k\in\ZZ^d}\abss{\hat{f}(k)}<\infty}$ and, moreover, their quasi-norms are equivalent in the sense of \eqref{eq:quasi-norm-equivalence}.
\cref{result:devore} yields 
\begin{equation*}
\frac{2}{\piup}\mnorm{f}_{\mathcal{A}_1^{1/2}(L_2(\mathbb{T}^d))}\leq\mnorm{f}_{\mathcal{A}(\mathbb{T}^d)}\leq 1.1064957714\mnorm{f}_{\mathcal{A}_1^{1/2}(L_2(\mathbb{T}^d))}
\end{equation*}
for $f\in L_2(\mathbb{T}^d)$, the constants being independent of $d$.

Following DeVore \cite[Remark~7.4]{DeVore1998}, if the Hilbert space $\mathcal{H}$ is chosen to be $L_2(\RR)$ with a wavelet orthonormal basis $(\psi_I)_{I\in\mathcal{D}}$, then the approximation space can be characterized in terms of Besov smoothness.
A multivariate version of this result using a tensorized wavelet basis is derived by Sickel and Ullrich in \cite[Theorem~2.7]{SickelUl2009}.
For further results in this direction, see the papers of Hansen and Sickel \cite{HansenSi2010,HansenSi2012}.

\textbf{Acknowledgements.} The authors would like to acknowledge support by the DFG Ul-403/2-1.
They would further like to thank Kai Lüttgen, Winfried Sickel, Vladimir Temlyakov, and Gerd Wachsmuth for insightful discussions.

\providecommand{\bysame}{\leavevmode\hbox to3em{\hrulefill}\thinspace}
\providecommand{\href}[2]{#2}

\end{document}

%% file: q3.tex
\draw (0,0.557506665975558)--(0.001,0.557506665377656)--(0.002,0.557506661192345)--(0.003,0.557506649832214)--(0.004,0.557506627709854)--(0.005,0.557506591237854)--(0.006,0.557506536828802)--(0.007,0.557506460895287)--(0.008,0.557506359849895)--(0.009,0.557506230105208)--(0.01,0.557506068073808)--(0.011,0.55750587016827)--(0.012,0.557505632801167)--(0.013,0.557505352385065)--(0.014,0.557505025332522)--(0.015,0.55750464805609)--(0.016,0.557504216968311)--(0.017,0.557503728481716)--(0.018,0.557503179008822)--(0.019,0.557502564962136)--(0.02,0.557501882754148)--(0.021,0.557501128797328)--(0.022,0.557500299504131)--(0.023,0.557499391286989)--(0.024,0.557498400558311)--(0.025,0.557497323730482)--(0.026,0.557496157215858)--(0.027,0.557494897426766)--(0.028,0.557493540775502)--(0.029,0.557492083674324)--(0.03,0.557490522535457)--(0.031,0.557488853771083)--(0.032,0.557487073793343)--(0.033,0.557485179014331)--(0.034,0.557483165846093)--(0.035,0.557481030700624)--(0.036,0.557478769989864)--(0.037,0.557476380125695)--(0.038,0.557473857519936)--(0.039,0.557471198584343)--(0.04,0.557468399730604)--(0.041,0.557465457370335)--(0.042,0.557462367915074)--(0.043,0.557459127776282)--(0.044,0.557455733365338)--(0.045,0.557452181093531)--(0.046,0.55744846737206)--(0.047,0.55744458861203)--(0.048,0.557440541224444)--(0.049,0.557436321620203)--(0.05,0.557431926210101)--(0.051,0.557427351404817)--(0.052,0.557422593614914)--(0.053,0.557417649250833)--(0.054,0.557412514722891)--(0.055,0.557407186441271)--(0.056,0.55740166081602)--(0.057,0.557395934257048)--(0.058,0.557390003174115)--(0.059,0.557383863976832)--(0.06,0.557377513074655)--(0.061,0.557370946876877)--(0.062,0.557364161792625)--(0.063,0.557357154230857)--(0.064,0.557349920600349)--(0.065,0.5573424573097)--(0.066,0.557334760767316)--(0.067,0.557326827381413)--(0.068,0.557318653560006)--(0.069,0.557310235710905)--(0.07,0.55730157024171)--(0.071,0.557292653559806)--(0.072,0.557283482072351)--(0.073,0.55727405218628)--(0.074,0.557264360308291)--(0.075,0.557254402844841)--(0.076,0.557244176202144)--(0.077,0.557233676786159)--(0.078,0.557222901002587)--(0.079,0.557211845256866)--(0.08,0.55720050595416)--(0.081,0.557188879499358)--(0.082,0.557176962297067)--(0.083,0.557164750751602)--(0.084,0.557152241266983)--(0.085,0.557139430246927)--(0.086,0.557126314094843)--(0.087,0.557112889213826)--(0.088,0.557099152006647)--(0.089,0.557085098875752)--(0.09,0.557070726223249)--(0.091,0.557056030450909)--(0.092,0.557041007960154)--(0.093,0.557025655152052)--(0.094,0.557009968427313)--(0.0946220976344154,0.557)--(0.095,0.556993944126515)--(0.096,0.556977578600636)--(0.097,0.556960868343909)--(0.098,0.556943809754842)--(0.099,0.556926399231529)--(0.1,0.55690863317164)--(0.101,0.556890507972419)--(0.102,0.556872020030676)--(0.103,0.556853165742777)--(0.104,0.556833941504641)--(0.105,0.556814343711733)--(0.106,0.556794368759054)--(0.107,0.556774013041139)--(0.108,0.556753272952048)--(0.109,0.556732144885359)--(0.11,0.556710625234163)--(0.111,0.556688710391058)--(0.112,0.556666396748139)--(0.113,0.556643680696997)--(0.114,0.556620558628711)--(0.115,0.556597026933838)--(0.116,0.556573082002414)--(0.117,0.556548720223941)--(0.118,0.556523937987388)--(0.119,0.556498731681179)--(0.12,0.556473097693192)--(0.121,0.556447032410749)--(0.122,0.556420532220617)--(0.123,0.556393593508996)--(0.124,0.556366212661517)--(0.125,0.556338386063238)--(0.126,0.556310110098636)--(0.127,0.556281381151605)--(0.128,0.55625219560545)--(0.129,0.556222549842882)--(0.13,0.556192440246014)--(0.131,0.556161863196359)--(0.132,0.556130815074821)--(0.133,0.556099292261696)--(0.134,0.556067291136666)--(0.135,0.556034808078794)--(0.136,0.556001839466525)--(0.136054984869221,0.556)--(0.137,0.55596838074714)--(0.138,0.555934429117996)--(0.139,0.555899981006692)--(0.14,0.555865032787176)--(0.141,0.555829580832716)--(0.142,0.555793621515902)--(0.143,0.555757151208644)--(0.144,0.555720166282165)--(0.145,0.555682663107004)--(0.146,0.555644638053011)--(0.147,0.555606087489348)--(0.148,0.555567007784487)--(0.149,0.555527395306205)--(0.15,0.555487246421592)--(0.151,0.555446557497041)--(0.152,0.555405324898255)--(0.153,0.555363544990243)--(0.154,0.555321214137322)--(0.155,0.555278328703118)--(0.156,0.555234885050564)--(0.157,0.555190879541906)--(0.158,0.555146308538702)--(0.159,0.555101168401823)--(0.16,0.555055455491456)--(0.161,0.555009166167109)--(0.161195594213009,0.555)--(0.162,0.554962294962814)--(0.163,0.554914839515303)--(0.164,0.554866796614681)--(0.165,0.554818162614614)--(0.166,0.554768933868052)--(0.167,0.554719106727233)--(0.168,0.554668677543688)--(0.169,0.554617642668246)--(0.17,0.554565998451039)--(0.171,0.55451374124151)--(0.172,0.554460867388414)--(0.173,0.554407373239833)--(0.174,0.554353255143174)--(0.175,0.554298509445185)--(0.176,0.554243132491957)--(0.177,0.554187120628934)--(0.178,0.554130470200924)--(0.179,0.554073177552105)--(0.18,0.554015239026036)--(0.18026014387849,0.554)--(0.181,0.553956648055186)--(0.182,0.55389740271568)--(0.183,0.553837500352144)--(0.184,0.553776937301544)--(0.185,0.553715709900227)--(0.186,0.553653814483929)--(0.187,0.553591247387788)--(0.188,0.553528004946359)--(0.189,0.553464083493628)--(0.19,0.553399479363024)--(0.191,0.55333418888744)--(0.192,0.553268208399239)--(0.193,0.553201534230281)--(0.194,0.553134162711931)--(0.195,0.553066090175081)--(0.195960940030764,0.553)--(0.196,0.552997312722184)--(0.197,0.552927821154817)--(0.198,0.552857617324104)--(0.199,0.55278669755405)--(0.2,0.552715058168216)--(0.201,0.552642695489739)--(0.202,0.552569605841358)--(0.203,0.55249578554543)--(0.204,0.552421230923958)--(0.205,0.552345938298609)--(0.206,0.552269903990741)--(0.207,0.552193124321428)--(0.208,0.552115595611481)--(0.209,0.55203731418148)--(0.209472172842599,0.552)--(0.21,0.55195827202291)--(0.211,0.551878465666059)--(0.212,0.551797895244945)--(0.213,0.551716557073357)--(0.214,0.551634447464926)--(0.215,0.551551562733147)--(0.216,0.551467899191418)--(0.217,0.551383453153064)--(0.218,0.551298220931373)--(0.219,0.551212198839629)--(0.22,0.551125383191141)--(0.221,0.551037770299281)--(0.221427272593469,0.551)--(0.222,0.550949350303636)--(0.223,0.550860120776344)--(0.224,0.550770082571745)--(0.225,0.550679231996852)--(0.226,0.550587565358856)--(0.227,0.550495078965165)--(0.228,0.55040176912344)--(0.229,0.550307632141641)--(0.23,0.550212664328063)--(0.231,0.550116861991378)--(0.232,0.550020221440682)--(0.232207490774917,0.55)--(0.233,0.54992272815749)--(0.234,0.549824386038806)--(0.235,0.549725194188551)--(0.236,0.549625148910203)--(0.237,0.549524246507817)--(0.238,0.549422483286079)--(0.239,0.549319855550352)--(0.24,0.549216359606722)--(0.241,0.549111991762054)--(0.242,0.549006748324034)--(0.242063610956427,0.549)--(0.243,0.548900609890916)--(0.244,0.548793586914135)--(0.245,0.548685676750876)--(0.246,0.548576875704209)--(0.247,0.548467180078235)--(0.248,0.548356586178144)--(0.249,0.548245090310265)--(0.25,0.548132688782132)--(0.251,0.548019377902536)--(0.251169701797899,0.548)--(0.252,0.547905137346661)--(0.253,0.547789976112716)--(0.254,0.547673893867595)--(0.255,0.547556886916659)--(0.256,0.547438951566779)--(0.257,0.547320084126394)--(0.258,0.547200280905578)--(0.259,0.547079538216108)--(0.259653726534003,0.547)--(0.26,0.546957844290794)--(0.261,0.546835187775659)--(0.262,0.546711580074534)--(0.263,0.546587017497065)--(0.264,0.546461496354887)--(0.265,0.546335012961694)--(0.266,0.546207563633312)--(0.267,0.54607914468777)--(0.267611765366767,0.546)--(0.268,0.545949741951968)--(0.269,0.545819345145158)--(0.27,0.545687966951369)--(0.271,0.545555603688081)--(0.272,0.545422251675307)--(0.273,0.54528790723567)--(0.274,0.545152566694487)--(0.275,0.545016226379844)--(0.275118190684647,0.545)--(0.276,0.544878855166079)--(0.277,0.544740472392633)--(0.278,0.544601078030373)--(0.279,0.544460668408708)--(0.28,0.544319239860187)--(0.281,0.544176788720588)--(0.282,0.544033311329006)--(0.282230599061338,0.544)--(0.283,0.54388877699499)--(0.284,0.543743200192086)--(0.285,0.543596585285418)--(0.286,0.543448928617249)--(0.287,0.543300226533571)--(0.288,0.543150475384187)--(0.288997822872593,0.543)--(0.289,0.54299967143811)--(0.29,0.542847771725703)--(0.291,0.542694811066033)--(0.292,0.542540785813823)--(0.293,0.54238569232805)--(0.294,0.542229526972048)--(0.295,0.542072286113603)--(0.295456706828514,0.542)--(0.296,0.54191394243992)--(0.297,0.541754495254771)--(0.298,0.541593960664286)--(0.299,0.541432335043474)--(0.3,0.541269614772273)--(0.301,0.541105796235655)--(0.301641618011593,0.541)--(0.302,0.540940858623657)--(0.303,0.540774783907598)--(0.304,0.540607599008045)--(0.305,0.540439300319333)--(0.306,0.540269884241345)--(0.307,0.540099347179613)--(0.307578868975551,0.54)--(0.308,0.539927663351707)--(0.309,0.539754819949245)--(0.31,0.539580843637918)--(0.311,0.539405730834949)--(0.312,0.539229477963755)--(0.313,0.539052081454059)--(0.313291815562819,0.539)--(0.314,0.538873496862574)--(0.315,0.538693743562689)--(0.316,0.538512834701278)--(0.317,0.538330766722262)--(0.318,0.53814753607644)--(0.318800191446556,0.538)--(0.319,0.537963126781779)--(0.32,0.537777497011291)--(0.321,0.537590692653798)--(0.322,0.537402710176506)--(0.323,0.537213546054093)--(0.324,0.537023196768822)--(0.324121170030242,0.537)--(0.325,0.536831599143774)--(0.326,0.536638799799369)--(0.327,0.536444803393179)--(0.328,0.536249606426714)--(0.329,0.536053205409695)--(0.329269363996688,0.536)--(0.33,0.53585554355092)--(0.331,0.535656649731949)--(0.332,0.535456539983362)--(0.333,0.535255210836609)--(0.334,0.535052658831987)--(0.334258529437575,0.535)--(0.335,0.534848822472224)--(0.336,0.534643734755973)--(0.337,0.534437412331295)--(0.338,0.53422985176282)--(0.339,0.534021049624685)--(0.339100270150105,0.534)--(0.34,0.533810927044081)--(0.341,0.53359954613994)--(0.342,0.533386911853326)--(0.343,0.533173020785428)--(0.343804282256824,0.533)--(0.344,0.532957852190986)--(0.345,0.53274134757858)--(0.346,0.532523574400535)--(0.347,0.532304529287349)--(0.348,0.532084208880267)--(0.348380158536375,0.532)--(0.349,0.531862551106171)--(0.35,0.531639573873029)--(0.351,0.531415309617416)--(0.352,0.531189755012797)--(0.352836574545542,0.531)--(0.353,0.530962890407339)--(0.354,0.530734643995812)--(0.355,0.530505095550182)--(0.356,0.530274241778553)--(0.357,0.530042079401065)--(0.357180330873891,0.53)--(0.358,0.529808517844993)--(0.359,0.529573620244555)--(0.36,0.529337402428881)--(0.361,0.529099861155827)--(0.361418226344842,0.529)--(0.362,0.528860927952358)--(0.363,0.528620616350433)--(0.364,0.528378969741813)--(0.365,0.52813598492462)--(0.36555674238998,0.528)--(0.366,0.527891606419879)--(0.367,0.527645816125236)--(0.368,0.527398676140601)--(0.369,0.527150183306969)--(0.36960126550848,0.527)--(0.37,0.526900285044133)--(0.371,0.526648951560303)--(0.372,0.526396253822077)--(0.373,0.526142188715978)--(0.373556828192799,0.526)--(0.374,0.525886695450585)--(0.375,0.525629754517093)--(0.376,0.52537143489475)--(0.377,0.5251117335183)--(0.37742813368615,0.525)--(0.378,0.524850569237613)--(0.379,0.524587956873063)--(0.38,0.524323951526511)--(0.381,0.524058550183652)--(0.381219578851725,0.524)--(0.382,0.52379163813012)--(0.383,0.523523290677057)--(0.384,0.523253536101769)--(0.384935034515547,0.523)--(0.385,0.522982361794411)--(0.386,0.522709634142536)--(0.387,0.522435488313983)--(0.388,0.522159921387789)--(0.388577535540061,0.522)--(0.389,0.521882864789466)--(0.39,0.521604289750931)--(0.391,0.521324282678521)--(0.392,0.521042840709953)--(0.392151542585358,0.521)--(0.393,0.520759823085676)--(0.394,0.520475338073957)--(0.395,0.520189407357539)--(0.395659257306245,0.52)--(0.396,0.519901970769919)--(0.397,0.519612970126926)--(0.398,0.519322513062248)--(0.399,0.519030596820859)--(0.399104363837985,0.519)--(0.4,0.518737061224235)--(0.401,0.518442040334479)--(0.402,0.518145549695898)--(0.402488677721465,0.518)--(0.403,0.517847493628322)--(0.404,0.517547871577599)--(0.405,0.517246769310936)--(0.405815655453317,0.517)--(0.406,0.5169441495213)--(0.407,0.516639889132512)--(0.408,0.516334138171667)--(0.409,0.51602689405148)--(0.409087172469936,0.516)--(0.41,0.515717975324558)--(0.411,0.515407538786671)--(0.412,0.51509559890364)--(0.412305163693891,0.515)--(0.413,0.514782012512678)--(0.414,0.514466853717071)--(0.415,0.514150181515471)--(0.415472191431259,0.514)--(0.416,0.513831883140802)--(0.417,0.513511965628362)--(0.418,0.513190524780229)--(0.418590118209611,0.513)--(0.419,0.512867469791407)--(0.42,0.512542757344397)--(0.421,0.512216511769378)--(0.42166072186147,0.512)--(0.422,0.51188865524116)--(0.423,0.511559111903417)--(0.424,0.511228025789009)--(0.424685699592629,0.511)--(0.425,0.510895322518595)--(0.426,0.510560912616088)--(0.427,0.51022495043822)--(0.427666671842293,0.51)--(0.428,0.509887354963716)--(0.429,0.509548043125438)--(0.43,0.509207169669372)--(0.430605185946937,0.509)--(0.431,0.508864636289469)--(0.432,0.508520387468613)--(0.433,0.508174567850125)--(0.433502719618998,0.508)--(0.434,0.507827050644964)--(0.435,0.507477830140361)--(0.436,0.507127029827173)--(0.436360684250774,0.507)--(0.437,0.506774482680365)--(0.438,0.506420256158146)--(0.439,0.506064440991561)--(0.439180428053286,0.506)--(0.44,0.505706817613336)--(0.441,0.50534755112877)--(0.441963139329091,0.505)--(0.442,0.504986677347624)--(0.443,0.504623941296936)--(0.444,0.504259601316418)--(0.444709575762368,0.504)--(0.445,0.503893574589785)--(0.446,0.503525740288844)--(0.447,0.503156293711992)--(0.447421418900992,0.503)--(0.448,0.50278506731641)--(0.449,0.502412101921804)--(0.45,0.502037516103615)--(0.450099799764095,0.502)--(0.451,0.501661043235325)--(0.452,0.501282914375156)--(0.452745153152322,0.501)--(0.453,0.500903080615648)--(0.454,0.500521390951193)--(0.455,0.500138066747338)--(0.45535885343589,0.5)--(0.456,0.499752908814185)--(0.457,0.499366000038888)--(0.457942009648712,0.499)--(0.458,0.498977430898952)--(0.459,0.498586920048532)--(0.46,0.498194761117619)--(0.460494778610793,0.498)--(0.461,0.497800789770319)--(0.462,0.497405005502515)--(0.463,0.497007565925652)--(0.4630189741269,0.497)--(0.464,0.496608145362191)--(0.465,0.496207057532251)--(0.465514331077486,0.496)--(0.466,0.495804143770752)--(0.467,0.495399390685997)--(0.467982717167556,0.495)--(0.468,0.494992963809374)--(0.469,0.494584528588194)--(0.47,0.494174420055443)--(0.470423790973519,0.494)--(0.471,0.493762434465012)--(0.472,0.493348621691241)--(0.472839177432594,0.493)--(0.473,0.492933071579433)--(0.474,0.49251553804732)--(0.475,0.492096319575251)--(0.475228998131131,0.492)--(0.476,0.491675132577101)--(0.477,0.491252171893942)--(0.477594049555575,0.491)--(0.478,0.49082736880193)--(0.479,0.490400649790396)--(0.479935245792272,0.49)--(0.48,0.489972210323923)--(0.481,0.489541717023228)--(0.482,0.489109523060077)--(0.482252588805322,0.489)--(0.483,0.488675337451307)--(0.484,0.48823934872575)--(0.484547063854753,0.488)--(0.485,0.487801475047122)--(0.486,0.487361676049195)--(0.486819314970693,0.487)--(0.487,0.486920093910383)--(0.488,0.486476469308109)--(0.489,0.486031129762666)--(0.489069691870444,0.486)--(0.49,0.485583692927609)--(0.491,0.485134508381147)--(0.49129849870358,0.485)--(0.492,0.48468331149425)--(0.493,0.484230267115929)--(0.49350656749846,0.484)--(0.494,0.483775289770223)--(0.495,0.483318370928511)--(0.4956943469115,0.483)--(0.496,0.482859592707723)--(0.497,0.482398784976723)--(0.497862273234878,0.482)--(0.498,0.481936185463478)--(0.499,0.481471474629228)--(0.5,0.481005028625072)--(0.500010750484843,0.481)--(0.501,0.480536405480155)--(0.502,0.480066038845103)--(0.50213999171901,0.48)--(0.503,0.479593543363864)--(0.504,0.479119242463443)--(0.504250658737773,0.479)--(0.505,0.4786428543698)--(0.506,0.478164605802104)--(0.506343139672872,0.478)--(0.507,0.477684304857453)--(0.508,0.477202095459238)--(0.50841781207703,0.477)--(0.509,0.476717861471387)--(0.51,0.476231678324117)--(0.510475043231295,0.476)--(0.511,0.47574349115633)--(0.512,0.475253321592154)--(0.512515190443402,0.475)--(0.513,0.47476116117231)--(0.514,0.474266992779976)--(0.514538601337404,0.474)--(0.515,0.473770839109821)--(0.516,0.473272659740514)--(0.516545614134867,0.473)--(0.517,0.472772492905004)--(0.518,0.472270290678109)--(0.518536557927867,0.472)--(0.519,0.471766090854827)--(0.52,0.471259854163601)--(0.520511752944036,0.471)--(0.521,0.47075160163225)--(0.522,0.470241319149407)--(0.52247151080392,0.47)--(0.523,0.469728994301357)--(0.524,0.469214654984557)--(0.524416134770846,0.469)--(0.525,0.468698238332442)--(0.526,0.468179831429677)--(0.526345919993558,0.468)--(0.527,0.467659303617033)--(0.528,0.467136818671909)--(0.528261153741803,0.467)--(0.529,0.46661216048284)--(0.53,0.466085587339744)--(0.530162115635107,0.466)--(0.531,0.46555677970861)--(0.532,0.465026108517765)--(0.532049077864924,0.465)--(0.533,0.464493132538874)--(0.533922190048572,0.464)--(0.534,0.463958309878812)--(0.535,0.46342119069858)--(0.535781737290216,0.463)--(0.536,0.462882170480276)--(0.537,0.462340926407575)--(0.537628045818634,0.462)--(0.538,0.461797690100919)--(0.539,0.461252312394956)--(0.539461361281435,0.461)--(0.54,0.460704841664847)--(0.541,0.46015532191324)--(0.541281922868655,0.46)--(0.542,0.459603598629256)--(0.543,0.459049928752371)--(0.54308996350074,0.459)--(0.544,0.458493934998387)--(0.544885553696576,0.458)--(0.545,0.45793603809851)--(0.546,0.457375825337275)--(0.54666893774619,0.457)--(0.547,0.456813629969331)--(0.548,0.456249244785297)--(0.548440456554166,0.456)--(0.549,0.45568273355084)--(0.55,0.45511416906948)--(0.550200320552146,0.455)--(0.551,0.454543324818988)--(0.551948668283782,0.454)--(0.552,0.4539705421642)--(0.553,0.45339538035979)--(0.553685497699558,0.453)--(0.554,0.452818237679598)--(0.555,0.452238877382056)--(0.555411267575367,0.452)--(0.556,0.451657358135207)--(0.557,0.451073793729844)--(0.557126168165493,0.451)--(0.558,0.450487881657832)--(0.558830176677678,0.45)--(0.559,0.449899996487818)--(0.56,0.449309787030699)--(0.560523522919877,0.449)--(0.561,0.448717483217607)--(0.562,0.448123053705198)--(0.562206540727419,0.448)--(0.563,0.447526315687379)--(0.563879260035668,0.447)--(0.564,0.446927580352341)--(0.565,0.446326474343478)--(0.565541743635676,0.446)--(0.566,0.445723279932096)--(0.567,0.445117940331508)--(0.567194402885057,0.445)--(0.568,0.444510272041628)--(0.568837213532651,0.444)--(0.569,0.443900582679008)--(0.57,0.443288538874587)--(0.570470293802214,0.443)--(0.571,0.442674351811196)--(0.572,0.44205806335295)--(0.572094019750746,0.442)--(0.573,0.441439364606886)--(0.573708223010994,0.441)--(0.574,0.440818621603481)--(0.575,0.440195605086154)--(0.575313264852731,0.44)--(0.576,0.439570327727687)--(0.576909294786543,0.439)--(0.577,0.438942992467626)--(0.578,0.438313233207554)--(0.578496213836515,0.438)--(0.579,0.437681341626198)--(0.58,0.437047323963133)--(0.580074494504762,0.437)--(0.581,0.436410863498406)--(0.581643906628366,0.436)--(0.582,0.435772323034602)--(0.583,0.435131545177411)--(0.583204869003529,0.435)--(0.584,0.434488415170801)--(0.5847573439809,0.434)--(0.585,0.433843192177792)--(0.586,0.433195643407854)--(0.586301494195982,0.433)--(0.587,0.432545811014665)--(0.587837489108187,0.432)--(0.588,0.431893872919394)--(0.589,0.431239544056145)--(0.589365313291746,0.431)--(0.59,0.43058297754086)--(0.590885269253689,0.43)--(0.591,0.429924292903723)--(0.592,0.429263176324938)--(0.592397234406233,0.429)--(0.593,0.428599845119992)--(0.593901577193694,0.428)--(0.594,0.427934383693313)--(0.595,0.42726647335342)--(0.595398132026666,0.427)--(0.596,0.426596348116619)--(0.596887272678762,0.426)--(0.597,0.425924080901684)--(0.598,0.425249372348097)--(0.598368848417624,0.425)--(0.599,0.424572425018506)--(0.599843183815639,0.424)--(0.6,0.423893324321023)--(0.601,0.423211814708452)--(0.601310194968501,0.423)--(0.602,0.422528018560644)--(0.602770108391951,0.422)--(0.603,0.421842058044496)--(0.604,0.421153746147211)--(0.604222953485147,0.421)--(0.605,0.420463075843733)--(0.605668815145953,0.42)--(0.606,0.419770230582899)--(0.607,0.419075116804936)--(0.607107877427874,0.419)--(0.608,0.418377548446877)--(0.608540044983464,0.418)--(0.609,0.4176777949754)--(0.609965665853571,0.417)--(0.61,0.416975852532983)--(0.611,0.416271392534235)--(0.611384512144084,0.416)--(0.612,0.415564708894129)--(0.612796943949608,0.415)--(0.613,0.414855826702637)--(0.614,0.414144568955418)--(0.614202905318675,0.414)--(0.615,0.41343093474241)--(0.615602477022351,0.413)--(0.616,0.412715092880033)--(0.61699588566136,0.412)--(0.617,0.411997039790033)--(0.618,0.411276439746441)--(0.618382918337553,0.411)--(0.619,0.410553619907767)--(0.619763932351084,0.41)--(0.62,0.409828580386985)--(0.621,0.409101196040251)--(0.621138898137981,0.409)--(0.622,0.40837138157661)--(0.622507820419016,0.408)--(0.623,0.407639339504346)--(0.623870934787291,0.407)--(0.624,0.406905066571946)--(0.625,0.406168356701126)--(0.625228148115163,0.406)--(0.626,0.405429297668599)--(0.626579600219791,0.405)--(0.627,0.404688000555229)--(0.627925443038342,0.404)--(0.628,0.40394446230968)--(0.629,0.403198440533065)--(0.629265561994592,0.403)--(0.63,0.40245010714525)--(0.630600152248341,0.402)--(0.631,0.401699526144318)--(0.631929323234429,0.401)--(0.632,0.400946694693732)--(0.633,0.400191378974411)--(0.633252972620992,0.4)--(0.634,0.399433745508995)--(0.634571282026384,0.399)--(0.635,0.39867385588454)--(0.635884354007369,0.398)--(0.636,0.39791170749322)--(0.637,0.397147120137136)--(0.637192127299263,0.397)--(0.638,0.39638016507654)--(0.638494710967608,0.396)--(0.639,0.39561094634073)--(0.639792231290663,0.395)--(0.64,0.394839461563734)--(0.641,0.394065629207379)--(0.641084692480619,0.394)--(0.642,0.393289335367272)--(0.642372080987977,0.393)--(0.643,0.392510771405746)--(0.643654572855034,0.392)--(0.644,0.391729935211007)--(0.644932219911494,0.391)--(0.645,0.390946824728182)--(0.646,0.390161243444156)--(0.646204958861353,0.39)--(0.647,0.389373322628662)--(0.647472922592969,0.389)--(0.648,0.388583124506273)--(0.648736199122357,0.388)--(0.649,0.387790647291251)--(0.649994837566036,0.387)--(0.65,0.386995889259312)--(0.651,0.38619860949417)--(0.651248754565205,0.386)--(0.652,0.385399043552513)--(0.652498132650594,0.385)--(0.653,0.384597194874937)--(0.653743021743279,0.384)--(0.654,0.383793062019828)--(0.654983468413112,0.383)--(0.655,0.382986643611189)--(0.656,0.382177724701123)--(0.656219408320833,0.382)--(0.657,0.381366504559767)--(0.657450993790172,0.381)--(0.658,0.380552998078751)--(0.658678277957772,0.38)--(0.659,0.37973720417679)--(0.659901305002886,0.379)--(0.66,0.378919121841997)--(0.661,0.378098631907119)--(0.661120061316441,0.378)--(0.662,0.377275757815637)--(0.662334604108611,0.377)--(0.663,0.376450595665664)--(0.663545023435493,0.376)--(0.664,0.375623144750512)--(0.66475136123197,0.375)--(0.665,0.374793404436265)--(0.665953658923618,0.374)--(0.666,0.37396137416257)--(0.667,0.373126901975112)--(0.667151889625192,0.373)--(0.668,0.372290095865386)--(0.66834614453243,0.372)--(0.669,0.371451001593603)--(0.669536484404657,0.371)--(0.67,0.370609618915668)--(0.670722948590177,0.37)--(0.671,0.369765947663916)--(0.671905575964751,0.369)--(0.672,0.368919987747793)--(0.673,0.368071654079194)--(0.673084369698769,0.368)--(0.674,0.367220939867429)--(0.67425936645759,0.367)--(0.675,0.366367940238148)--(0.675430643549114,0.366)--(0.676,0.365512655425735)--(0.676598237926081,0.365)--(0.677,0.364655085744035)--(0.677762186102687,0.364)--(0.678,0.363795231586916)--(0.678922524160634,0.363)--(0.679,0.362933093428831)--(0.68,0.362068590986925)--(0.680079257190554,0.362)--(0.681,0.361201729934772)--(0.68123242360476,0.361)--(0.682,0.360332589845038)--(0.682382088378099,0.36)--(0.683,0.359461171524651)--(0.683528285836968,0.359)--(0.684,0.358587475862841)--(0.684671049906267,0.358)--(0.685,0.357711503831573)--(0.685810414114918,0.357)--(0.686,0.356833256485979)--(0.686946411601317,0.356)--(0.687,0.355952734964749)--(0.688,0.355069858989456)--(0.688079047358609,0.355)--(0.689,0.354184659237707)--(0.689208364792217,0.354)--(0.69,0.353297192253808)--(0.690334414881377,0.353)--(0.691,0.352407459513845)--(0.691457229184418,0.352)--(0.692,0.351515462578459)--(0.692576838896971,0.351)--(0.693,0.350621203093141)--(0.693693274856944,0.35)--(0.694,0.349724682788499)--(0.694806567549419,0.349)--(0.695,0.348825903480505)--(0.695916747111489,0.348)--(0.696,0.347924867070727)--(0.697,0.347021550744499)--(0.697023835823294,0.347)--(0.698,0.346115897821634)--(0.698127845900406,0.346)--(0.699,0.345207996966484)--(0.699228832981919,0.345)--(0.7,0.344297850418691)--(0.700326825801866,0.344)--(0.701,0.343385460503681)--(0.701421852770898,0.343)--(0.702,0.34247082963279)--(0.702513941980692,0.342)--(0.703,0.341553960303362)--(0.703603121208298,0.341)--(0.704,0.340634855098825)--(0.70468941792042,0.34)--(0.705,0.339713516688753)--(0.705772859277641,0.339)--(0.706,0.338789947828909)--(0.706853472138586,0.338)--(0.707,0.337864151361264)--(0.707931283064023,0.337)--(0.708,0.336936130213999)--(0.709,0.336005880779833)--(0.70900631660663,0.336)--(0.71,0.335073343609392)--(0.710078582841831,0.335)--(0.711,0.334138593856904)--(0.711148126309533,0.334)--(0.712,0.333201634782027)--(0.712214972375231,0.333)--(0.713,0.332262469729801)--(0.713279146127363,0.332)--(0.714,0.33132110213055)--(0.714340672381053,0.331)--(0.715,0.330377535499767)--(0.715399575681793,0.33)--(0.716,0.32943177343797)--(0.716455880309078,0.329)--(0.717,0.328483819630555)--(0.71750961027999,0.328)--(0.718,0.327533677847617)--(0.718560789352724,0.327)--(0.719,0.326581351943762)--(0.719609441030069,0.326)--(0.72,0.325626845857903)--(0.720655588562836,0.325)--(0.721,0.32467016361303)--(0.721699254953239,0.324)--(0.722,0.323711309315972)--(0.722740462958224,0.323)--(0.723,0.322750287157136)--(0.723779235092753,0.322)--(0.724,0.321787101410233)--(0.724815593633041,0.321)--(0.725,0.320821756431985)--(0.725849560619741,0.32)--(0.726,0.31985425666182)--(0.72688115786109,0.319)--(0.727,0.318884606621541)--(0.727910406936002,0.318)--(0.728,0.317912810914994)--(0.728937329197126,0.317)--(0.729,0.316938874227705)--(0.729961945773848,0.316)--(0.73,0.315962801326512)--(0.730984277575261,0.315)--(0.731,0.314984597059176)--(0.732,0.314004261808101)--(0.732004344424605,0.314)--(0.733,0.313021791039522)--(0.733022165035377,0.313)--(0.734,0.31203720627552)--(0.734037762835404,0.312)--(0.735,0.311050512660354)--(0.735051157857682,0.311)--(0.736,0.310061715416708)--(0.736062369929797,0.31)--(0.737,0.309070819845219)--(0.737071418676622,0.309)--(0.738,0.308077831324004)--(0.738078323522979,0.308)--(0.739,0.307082755308166)--(0.739083103696257,0.307)--(0.74,0.30608559732929)--(0.740085778228999,0.306)--(0.741,0.30508636299492)--(0.741086365961443,0.305)--(0.742,0.304085057988036)--(0.742084885544031,0.304)--(0.743,0.303081688066501)--(0.743081355439881,0.303)--(0.744,0.302076259062511)--(0.744075793927224,0.302)--(0.745,0.301068776882026)--(0.745068219101801,0.301)--(0.746,0.300059247504187)--(0.746058648879229,0.3)--(0.747,0.299047676980729)--(0.747047100997332,0.299)--(0.748,0.298034071435372)--(0.748033593018437,0.298)--(0.749,0.297018437063214)--(0.749018142331636,0.297)--(0.75,0.296000780130101)--(0.750000766155015,0.296)--(0.750981478719344,0.295)--(0.751,0.294981087942888)--(0.75196029940946,0.294)--(0.752,0.293959383261449)--(0.752937245075181,0.293)--(0.753,0.292935673376168)--(0.753912332285139,0.292)--(0.754,0.291909964857416)--(0.754885577447183,0.291)--(0.755,0.290882264343205)--(0.755856996810439,0.29)--(0.756,0.289852578538493)--(0.756826606467329,0.289)--(0.757,0.288820914214482)--(0.757794422355565,0.288)--(0.758,0.287787278207909)--(0.758760460260107,0.287)--(0.759,0.286751677420324)--(0.759724735815104,0.286)--(0.76,0.285714118817364)--(0.760687264505795,0.285)--(0.761,0.284674609428015)--(0.761648061670388,0.284)--(0.762,0.283633156343873)--(0.762607142501911,0.283)--(0.763,0.282589766718389)--(0.763564522050035,0.282)--(0.764,0.281544447766111)--(0.76452021522287,0.281)--(0.765,0.280497206761925)--(0.765474236788734,0.28)--(0.766,0.279448051040277)--(0.7664266013779,0.279)--(0.767,0.278396987994397)--(0.767377323484312,0.278)--(0.768,0.277344025075518)--(0.768326417467277,0.277)--(0.769,0.276289169792083)--(0.76927389755314,0.276)--(0.77,0.275232429708951)--(0.770219777836918,0.275)--(0.771,0.274173812446594)--(0.771164072283931,0.274)--(0.772,0.273113325680297)--(0.772106794731388,0.273)--(0.773,0.272050977139338)--(0.773047958889968,0.272)--(0.773987577024294,0.271)--(0.774,0.270986762380279)--(0.774925658777009,0.27)--(0.775,0.269920652926436)--(0.775862222674125,0.269)--(0.776,0.268852704007246)--(0.776797281940205,0.268)--(0.777,0.267782923589893)--(0.777730849680578,0.267)--(0.778,0.266711319691421)--(0.778662938882778,0.266)--(0.779,0.265637900377894)--(0.77959356241797,0.265)--(0.78,0.264562673763548)--(0.780522733042356,0.264)--(0.781,0.263485648009946)--(0.781450463398562,0.263)--(0.782,0.262406831325126)--(0.782376766016992,0.262)--(0.783,0.261326231962752)--(0.783301653317182,0.261)--(0.784,0.260243858221253)--(0.78422513760912,0.26)--(0.785,0.259159718442973)--(0.785147231094546,0.259)--(0.786,0.258073821013307)--(0.786067945868246,0.258)--(0.786987292832649,0.257)--(0.787,0.256986162285421)--(0.787905279162884,0.256)--(0.788,0.255896697209631)--(0.788821922548338,0.255)--(0.789,0.25480549907943)--(0.789737234670321,0.254)--(0.79,0.253712576473454)--(0.790651227109188,0.253)--(0.791,0.252617938008974)--(0.791563911345518,0.252)--(0.792,0.251521592341032)--(0.792475298761286,0.251)--(0.793,0.250423548161571)--(0.793385400641004,0.25)--(0.794,0.24932381419856)--(0.794294228172857,0.249)--(0.795,0.248222399215135)--(0.795201792449815,0.248)--(0.796,0.24711931200872)--(0.796108104470735,0.247)--(0.797,0.246014561410167)--(0.797013175141441,0.246)--(0.797917009425988,0.245)--(0.798,0.244908080207312)--(0.798819623094618,0.244)--(0.799,0.243799940734576)--(0.799721027725152,0.243)--(0.8,0.242690164065252)--(0.800621233861987,0.242)--(0.801,0.241578759183662)--(0.801520251961576,0.241)--(0.802,0.240465735103063)--(0.802418092393421,0.24)--(0.803,0.239351100864791)--(0.803314765441057,0.239)--(0.804,0.238234865537392)--(0.804210281303018,0.238)--(0.805,0.237117038215769)--(0.805104650093792,0.237)--(0.805997881716048,0.236)--(0.806,0.235997626136242)--(0.806889979934525,0.235)--(0.807,0.234876546605988)--(0.80778096108644,0.234)--(0.808,0.233753902252373)--(0.808670834958606,0.233)--(0.809,0.232629702294373)--(0.809559611258285,0.232)--(0.81,0.231503955972861)--(0.810447299614075,0.231)--(0.811,0.23037667254976)--(0.811333909576766,0.23)--(0.812,0.229247861307206)--(0.812219450620197,0.229)--(0.813,0.228117531546705)--(0.813103932142096,0.228)--(0.813987362805835,0.227)--(0.814,0.2269856817165)--(0.814869747163335,0.226)--(0.815,0.225852242171709)--(0.815751099905249,0.225)--(0.816,0.224717312074554)--(0.816631430131592,0.224)--(0.817,0.223580900821397)--(0.81751074687025,0.223)--(0.818,0.222443017823612)--(0.818389059077758,0.222)--(0.819,0.221303672506765)--(0.819266375640055,0.221)--(0.82,0.220162874309807)--(0.820142705373242,0.22)--(0.821,0.219020632684265)--(0.821018057024324,0.219)--(0.821892434479565,0.218)--(0.822,0.217876867869623)--(0.822765850484567,0.217)--(0.823,0.216731663653168)--(0.823638314400876,0.216)--(0.824,0.215585044588589)--(0.824509834704408,0.215)--(0.825,0.21443702019367)--(0.825380419805508,0.214)--(0.826,0.213287599994565)--(0.826250078049638,0.213)--(0.827,0.21213679352502)--(0.827118817718045,0.212)--(0.827986646503362,0.211)--(0.828,0.210984599578308)--(0.828853568485005,0.21)--(0.829,0.20983094264841)--(0.82971959651385,0.209)--(0.83,0.208675928392806)--(0.830584738618666,0.208)--(0.831,0.207519566390463)--(0.831449002767224,0.207)--(0.832,0.206361866223874)--(0.832312396866923,0.206)--(0.833,0.205202837478318)--(0.833174928765401,0.205)--(0.834,0.204042489741122)--(0.834036606251136,0.204)--(0.834897433574077,0.203)--(0.835,0.202880753064943)--(0.835757420773293,0.202)--(0.836,0.201717688489337)--(0.83661657672718,0.201)--(0.837,0.2005533341553)--(0.837474908991085,0.2)--(0.838,0.199387699671856)--(0.838332425064092,0.199)--(0.839,0.198220794646214)--(0.839189132389579,0.198)--(0.84,0.197052628683069)--(0.84004503835577,0.197)--(0.840900147327989,0.196)--(0.841,0.195883136574413)--(0.841754468335343,0.195)--(0.842,0.194712369379826)--(0.842608009966957,0.194)--(0.843,0.193540370635797)--(0.843460779393296,0.193)--(0.844,0.192367149949301)--(0.84431278373231,0.192)--(0.845,0.191192716921131)--(0.845164030049949,0.191)--(0.846,0.190017081145242)--(0.846014525360655,0.19)--(0.846864273106226,0.189)--(0.847,0.188840154168383)--(0.847713283475583,0.188)--(0.848,0.187662033755952)--(0.848561563713684,0.187)--(0.849,0.186482740047012)--(0.849409120631296,0.186)--(0.85,0.185302282619303)--(0.850255960990099,0.185)--(0.851,0.184120671040262)--(0.851102091503147,0.184)--(0.851947517626531,0.183)--(0.852,0.182937878176834)--(0.852792244917847,0.182)--(0.853,0.181753879265153)--(0.853636282345423,0.181)--(0.854,0.180568755648986)--(0.854479636429765,0.18)--(0.855,0.179382516864732)--(0.855322313644992,0.179)--(0.856,0.178195172435215)--(0.856164320419266,0.178)--(0.857,0.177006731869105)--(0.857005663135214,0.177)--(0.857846345065348,0.176)--(0.858,0.175817101510521)--(0.858686375575052,0.175)--(0.859,0.174626391064135)--(0.859525761025701,0.174)--(0.86,0.17343461383315)--(0.860364507619371,0.173)--(0.861,0.172241779278602)--(0.861202621514653,0.172)--(0.862,0.171047896844328)--(0.862040108827043,0.171)--(0.862876973466501,0.17)--(0.863,0.169852896309873)--(0.863713223020934,0.169)--(0.864,0.168656841578037)--(0.864548864197217,0.168)--(0.865,0.167459768187529)--(0.865383902939801,0.167)--(0.866,0.166261685521399)--(0.866218345151944,0.166)--(0.867,0.165062602942643)--(0.867052196696077,0.165)--(0.867885461626038,0.164)--(0.868,0.163862458408256)--(0.868718146775203,0.163)--(0.869,0.162661300932094)--(0.869550258707469,0.162)--(0.87,0.161459172580117)--(0.870381803123066,0.161)--(0.871,0.160256082642179)--(0.871212785683165,0.16)--(0.872,0.159052040385436)--(0.872043212010212,0.159)--(0.872873085974758,0.158)--(0.873,0.157846979043885)--(0.873702414338813,0.157)--(0.874,0.156640958913465)--(0.874531203205483,0.156)--(0.875,0.155434015266524)--(0.875359458043257,0.155)--(0.876,0.154226157285541)--(0.876187184283545,0.154)--(0.877,0.153017394127845)--(0.87701438732099,0.153)--(0.877841070645084,0.152)--(0.878,0.1518076436312)--(0.878667241365041,0.151)--(0.879,0.150596999063158)--(0.879492904938056,0.15)--(0.88,0.149385477838996)--(0.880318066612499,0.149)--(0.881,0.148173089020661)--(0.881142731601507,0.148)--(0.881966904737434,0.147)--(0.882,0.146959823303075)--(0.882790590066874,0.146)--(0.883,0.145745629567179)--(0.883613794225682,0.145)--(0.884,0.144530596508172)--(0.884436522286737,0.144)--(0.885,0.143314733084451)--(0.88525877928913,0.143)--(0.886,0.142098048225316)--(0.886080570238433,0.142)--(0.886901899222009,0.141)--(0.887,0.140880498885434)--(0.88772277139682,0.14)--(0.888,0.139662104165617)--(0.888543192413074,0.139)--(0.889,0.138442915922126)--(0.889363167142864,0.138)--(0.89,0.137222942971506)--(0.890182700426115,0.137)--(0.891,0.136002194099313)--(0.891001797070834,0.136)--(0.891820460462676,0.135)--(0.892,0.134780587420588)--(0.892638696793747,0.134)--(0.893,0.133558222722726)--(0.893456510790109,0.133)--(0.894,0.132335109617503)--(0.894273907132629,0.132)--(0.895,0.131111256769344)--(0.895090890471511,0.131)--(0.895907464797598,0.13)--(0.896,0.129886627980138)--(0.896723634759953,0.129)--(0.897,0.128661233631384)--(0.897539405551228,0.128)--(0.898,0.12743512671613)--(0.898354781699468,0.127)--(0.899,0.126208315771355)--(0.899169767703246,0.126)--(0.899984367939064,0.125)--(0.9,0.124980802045323)--(0.900798585963283,0.124)--(0.901,0.123752523154459)--(0.901612427222559,0.123)--(0.902,0.122523567050834)--(0.902425896098015,0.122)--(0.903,0.121293942138199)--(0.90323899694243,0.121)--(0.904,0.120063656785219)--(0.904051734080443,0.12)--(0.904864111127569,0.119)--(0.905,0.118832659600117)--(0.90567613282034,0.118)--(0.906,0.117600997085162)--(0.906487803665648,0.117)--(0.907,0.11636870046327)--(0.907299127876007,0.116)--(0.908,0.115135777963582)--(0.908110109636845,0.115)--(0.908920752763452,0.114)--(0.909,0.113902204554505)--(0.909731061287851,0.113)--(0.91,0.112667976447751)--(0.910541039805545,0.112)--(0.911,0.111433148394287)--(0.911350692394028,0.111)--(0.912,0.110197728479073)--(0.912160023104708,0.11)--(0.912969035847872,0.109)--(0.913,0.108961712393423)--(0.913777734167914,0.108)--(0.914,0.107725057452731)--(0.914586122652292,0.107)--(0.915,0.106487836162382)--(0.915394205248461,0.106)--(0.916,0.105250056458983)--(0.916201985878734,0.105)--(0.917,0.104011726241193)--(0.917009468440462,0.104)--(0.917816656242797,0.103)--(0.918,0.10277278465527)--(0.918623553705308,0.102)--(0.919,0.101533306146194)--(0.919430164680215,0.101)--(0.92,0.100293302086604)--(0.920236492965414,0.1)--(0.921,0.0990527802224434)--(0.921042542334725,0.099)--(0.921848316143418,0.098)--(0.922,0.0978116944418539)--(0.92265381843195,0.097)--(0.923,0.0965700924705285)--(0.923459053016232,0.096)--(0.924,0.095327997208213)--(0.924264023574187,0.095)--(0.925,0.0940854162449923)--(0.925068733760505,0.094)--(0.925873186929627,0.093)--(0.926,0.092842314654747)--(0.926677386842068,0.092)--(0.927,0.0915987206390078)--(0.927481337238383,0.091)--(0.928,0.0903546649738509)--(0.928285041680748,0.09)--(0.929,0.0891101550907971)--(0.92908850370891,0.089)--(0.929891726643236,0.088)--(0.93,0.0878651642499753)--(0.930694714035645,0.087)--(0.931,0.0866197072000247)--(0.931497469525786,0.086)--(0.932,0.0853738195111873)--(0.932299996564114,0.085)--(0.933,0.0841275084541349)--(0.933102298579418,0.084)--(0.933904378835375,0.083)--(0.934,0.0828807529880557)--(0.934706240725078,0.082)--(0.935,0.0816335594452104)--(0.935507887772081,0.081)--(0.936,0.0803859656316607)--(0.936309323318951,0.08)--(0.937,0.0791379786553461)--(0.937110550687312,0.079)--(0.937911573069638,0.078)--(0.938,0.0778895811566163)--(0.938712393733549,0.077)--(0.939,0.0766407751395389)--(0.939513016079083,0.076)--(0.94,0.0753915985679981)--(0.940313443344567,0.075)--(0.941,0.0741420583857022)--(0.941113678748083,0.074)--(0.941913725402392,0.073)--(0.942,0.0728921393239663)--(0.942713586470839,0.072)--(0.943,0.0716418422821141)--(0.943513265228151,0.071)--(0.944,0.0703912037421507)--(0.944312764811902,0.07)--(0.945,0.0691402304824185)--(0.94511208834008,0.069)--(0.945911238841548,0.068)--(0.946,0.0678889081226004)--(0.94671021938786,0.067)--(0.947,0.0666372388966201)--(0.94750903313054,0.066)--(0.948,0.065385256562613)--(0.948307683109737,0.065)--(0.949,0.0641329677327712)--(0.949106172346643,0.064)--(0.949904503785006,0.063)--(0.95,0.0628803580615894)--(0.950702680411104,0.062)--(0.951,0.0616274328503629)--(0.951500705257906,0.061)--(0.952,0.0603742222500474)--(0.952298581271308,0.06)--(0.953,0.0591207327062281)--(0.95309631137886,0.059)--(0.9538938984399,0.058)--(0.954,0.0578669493667316)--(0.95469134535813,0.057)--(0.955,0.0566128817007883)--(0.955488655054066,0.056)--(0.956,0.0553585556900902)--(0.956285830382393,0.055)--(0.957,0.054103977613472)--(0.957082874180024,0.054)--(0.957879789223721,0.053)--(0.958,0.0528491318473331)--(0.95867657833513,0.052)--(0.959,0.0515940325683263)--(0.959473244329063,0.051)--(0.96,0.0503387013121879)--(0.9602697899719,0.05)--(0.961,0.0490831441911514)--(0.9610662180128,0.049)--(0.961862531148309,0.048)--(0.962,0.0478273447884407)--(0.962658732117567,0.047)--(0.963,0.0465713220343891)--(0.963454823636485,0.046)--(0.964,0.0453150929932866)--(0.96425080838592,0.045)--(0.965,0.0440586636111892)--(0.965046689030035,0.044)--(0.965842468187643,0.043)--(0.966,0.0428020168673374)--(0.966638148514805,0.042)--(0.967,0.0415451760633346)--(0.96743373263497,0.041)--(0.968,0.0402881539851866)--(0.968229223146136,0.04)--(0.969,0.0390309564134557)--(0.969024622630119,0.039)--(0.969819933630326,0.038)--(0.97,0.0377735660931064)--(0.970615158719599,0.037)--(0.971,0.0365160099471928)--(0.971410300434737,0.036)--(0.972,0.0352582968643632)--(0.972205361293436,0.035)--(0.973,0.0340004324606082)--(0.973000343797698,0.034)--(0.973795250417567,0.033)--(0.974,0.0327423997680555)--(0.974590083643253,0.032)--(0.975,0.0314842282719612)--(0.975384845929952,0.031)--(0.976,0.0302259235030559)--(0.976179539717512,0.03)--(0.976974167429169,0.029)--(0.977,0.0289674884278979)--(0.977768731467452,0.028)--(0.978,0.0277089144698078)--(0.978563234237216,0.027)--(0.979,0.0264502249042661)--(0.97935767811768,0.026)--(0.98,0.0251914250604349)--(0.980152065473198,0.025)--(0.980946398651744,0.024)--(0.981,0.0239325159914484)--(0.981740679986212,0.023)--(0.982,0.0226734960528637)--(0.982534911801819,0.022)--(0.983,0.0214143829972106)--(0.983329096404119,0.021)--(0.984,0.0201551819926636)--(0.984123236084234,0.02)--(0.984917333117734,0.019)--(0.985,0.0188958930344999)--(0.985711389767178,0.018)--(0.986,0.0176365196684616)--(0.986505408282841,0.017)--(0.987,0.0163770750142353)--(0.987299390898798,0.016)--(0.988,0.0151175640806919)--(0.988093339835118,0.015)--(0.988887257297286,0.014)--(0.989,0.0138579867184332)--(0.989681145478064,0.013)--(0.99,0.0125983498015773)--(0.99047500655651,0.012)--(0.991,0.0113386627698385)--(0.991268842697336,0.011)--(0.992,0.0100789304746356)--(0.992062656051643,0.01)--(0.992856448756787,0.009)--(0.993,0.00881915357133286)--(0.993650222937161,0.008)--(0.994,0.0075593403239254)--(0.994443980703561,0.007)--(0.995,0.00629949748603319)--(0.995237724153363,0.006)--(0.996,0.00503962975362057)--(0.996031455370719,0.005)--(0.996825176426614,0.004)--(0.997,0.00377973955263223)--(0.997618889379028,0.003)--(0.998,0.00251983456185885)--(0.998412596272877,0.002)--(0.999,0.00125991986343297)--(0.999206299140153,0.001)--(1,0);

%% file: q6.tex
\draw (0,0.528849054899926)--(0.001,0.528849054899926)--(0.002,0.528849054899926)--(0.003,0.528849054899924)--(0.004,0.528849054899917)--(0.005,0.528849054899893)--(0.006,0.528849054899827)--(0.007,0.528849054899675)--(0.008,0.528849054899367)--(0.009,0.528849054898794)--(0.01,0.528849054897795)--(0.011,0.528849054896151)--(0.012,0.528849054893564)--(0.013,0.528849054889641)--(0.014,0.528849054883883)--(0.015,0.528849054875656)--(0.016,0.528849054864179)--(0.017,0.528849054848496)--(0.018,0.528849054827456)--(0.019,0.528849054799685)--(0.02,0.52884905476356)--(0.021,0.528849054717183)--(0.022,0.528849054658346)--(0.023,0.528849054584504)--(0.024,0.52884905449274)--(0.025,0.528849054379733)--(0.026,0.528849054241715)--(0.027,0.528849054074444)--(0.028,0.528849053873156)--(0.029,0.528849053632529)--(0.03,0.528849053346637)--(0.031,0.528849053008911)--(0.032,0.528849052612091)--(0.033,0.528849052148179)--(0.034,0.528849051608392)--(0.035,0.528849050983111)--(0.036,0.528849050261829)--(0.037,0.528849049433101)--(0.038,0.528849048484487)--(0.039,0.528849047402495)--(0.04,0.528849046172528)--(0.041,0.52884904477882)--(0.042,0.528849043204378)--(0.043,0.52884904143092)--(0.044,0.528849039438809)--(0.045,0.52884903720699)--(0.046,0.528849034712925)--(0.047,0.528849031932519)--(0.048,0.528849028840056)--(0.049,0.528849025408125)--(0.05,0.528849021607546)--(0.051,0.528849017407298)--(0.052,0.528849012774444)--(0.053,0.528849007674048)--(0.054,0.528849002069102)--(0.055,0.528848995920443)--(0.056,0.528848989186669)--(0.057,0.528848981824061)--(0.058,0.528848973786491)--(0.059,0.52884896502534)--(0.06,0.528848955489407)--(0.061,0.528848945124824)--(0.062,0.528848933874957)--(0.063,0.528848921680321)--(0.064,0.52884890847848)--(0.065,0.528848894203953)--(0.066,0.528848878788121)--(0.067,0.528848862159118)--(0.068,0.52884884424174)--(0.069,0.528848824957339)--(0.07,0.528848804223719)--(0.071,0.528848781955034)--(0.072,0.528848758061675)--(0.073,0.528848732450171)--(0.074,0.528848705023069)--(0.075,0.528848675678833)--(0.076,0.528848644311724)--(0.077,0.528848610811687)--(0.078,0.528848575064239)--(0.079,0.528848536950345)--(0.08,0.528848496346307)--(0.081,0.528848453123637)--(0.082,0.528848407148939)--(0.083,0.528848358283782)--(0.084,0.528848306384578)--(0.085,0.528848251302455)--(0.086,0.528848192883129)--(0.087,0.528848130966772)--(0.088,0.528848065387883)--(0.089,0.528847995975154)--(0.09,0.528847922551339)--(0.091,0.528847844933114)--(0.092,0.528847762930941)--(0.093,0.528847676348931)--(0.094,0.528847584984701)--(0.095,0.528847488629234)--(0.096,0.528847387066734)--(0.097,0.528847280074484)--(0.098,0.528847167422696)--(0.099,0.528847048874363)--(0.1,0.528846924185112)--(0.101,0.528846793103051)--(0.102,0.528846655368619)--(0.103,0.528846510714427)--(0.104,0.528846358865106)--(0.105,0.528846199537148)--(0.106,0.528846032438749)--(0.107,0.528845857269647)--(0.108,0.52884567372096)--(0.109,0.528845481475025)--(0.11,0.528845280205229)--(0.111,0.528845069575846)--(0.112,0.528844849241867)--(0.113,0.528844618848832)--(0.114,0.528844378032658)--(0.115,0.528844126419466)--(0.116,0.528843863625407)--(0.117,0.528843589256487)--(0.118,0.52884330290839)--(0.119,0.528843004166296)--(0.12,0.528842692604705)--(0.121,0.528842367787254)--(0.122,0.528842029266532)--(0.123,0.528841676583894)--(0.124,0.52884130926928)--(0.125,0.52884092684102)--(0.126,0.52884052880565)--(0.127,0.528840114657718)--(0.128,0.528839683879593)--(0.129,0.528839235941268)--(0.13,0.528838770300168)--(0.131,0.528838286400949)--(0.132,0.528837783675303)--(0.133,0.528837261541754)--(0.134,0.52883671940546)--(0.135,0.528836156658005)--(0.136,0.528835572677198)--(0.137,0.528834966826865)--(0.138,0.528834338456639)--(0.139,0.528833686901754)--(0.14,0.528833011482831)--(0.141,0.528832311505667)--(0.142,0.52883158626102)--(0.143,0.528830835024393)--(0.144,0.528830057055817)--(0.145,0.528829251599632)--(0.146,0.528828417884267)--(0.147,0.528827555122017)--(0.148,0.528826662508821)--(0.149,0.528825739224035)--(0.15,0.528824784430208)--(0.151,0.528823797272849)--(0.152,0.528822776880205)--(0.153,0.528821722363021)--(0.154,0.528820632814313)--(0.155,0.528819507309133)--(0.156,0.52881834490433)--(0.157,0.528817144638317)--(0.158,0.528815905530825)--(0.159,0.528814626582673)--(0.16,0.528813306775514)--(0.161,0.5288119450716)--(0.162,0.528810540413533)--(0.163,0.528809091724019)--(0.164,0.528807597905617)--(0.165,0.528806057840496)--(0.166,0.528804470390174)--(0.167,0.528802834395273)--(0.168,0.52880114867526)--(0.169,0.528799412028192)--(0.17,0.528797623230457)--(0.171,0.528795781036517)--(0.172,0.528793884178647)--(0.173,0.528791931366667)--(0.174,0.528789921287687)--(0.175,0.528787852605834)--(0.176,0.528785723961987)--(0.177,0.52878353397351)--(0.178,0.528781281233979)--(0.179,0.528778964312912)--(0.18,0.528776581755495)--(0.181,0.528774132082306)--(0.182,0.528771613789038)--(0.183,0.528769025346224)--(0.184,0.528766365198954)--(0.185,0.528763631766593)--(0.186,0.5287608234425)--(0.187,0.528757938593745)--(0.188,0.528754975560818)--(0.189,0.528751932657345)--(0.19,0.528748808169797)--(0.191,0.5287456003572)--(0.192,0.528742307450842)--(0.193,0.528738927653978)--(0.194,0.528735459141537)--(0.195,0.528731900059821)--(0.196,0.528728248526208)--(0.197,0.528724502628852)--(0.198,0.52872066042638)--(0.199,0.528716719947588)--(0.2,0.528712679191137)--(0.201,0.528708536125244)--(0.202,0.528704288687375)--(0.203,0.528699934783936)--(0.204,0.528695472289956)--(0.205,0.528690899048781)--(0.206,0.528686212871754)--(0.207,0.528681411537899)--(0.208,0.528676492793604)--(0.209,0.528671454352301)--(0.21,0.528666293894143)--(0.211,0.528661009065685)--(0.212,0.528655597479552)--(0.213,0.528650056714122)--(0.214,0.528644384313189)--(0.215,0.528638577785642)--(0.216,0.528632634605126)--(0.217,0.528626552209715)--(0.218,0.528620328001574)--(0.219,0.528613959346625)--(0.22,0.52860744357421)--(0.221,0.528600777976749)--(0.222,0.528593959809401)--(0.223,0.528586986289722)--(0.224,0.52857985459732)--(0.225,0.528572561873509)--(0.226,0.528565105220962)--(0.227,0.528557481703362)--(0.228,0.528549688345052)--(0.229,0.528541722130683)--(0.23,0.528533580004858)--(0.231,0.52852525887178)--(0.232,0.52851675559489)--(0.233,0.528508066996516)--(0.234,0.528499189857506)--(0.235,0.528490120916867)--(0.236,0.528480856871405)--(0.237,0.528471394375358)--(0.238,0.528461730040029)--(0.239,0.528451860433419)--(0.24,0.528441782079856)--(0.241,0.528431491459623)--(0.242,0.528420985008588)--(0.243,0.528410259117828)--(0.244,0.52839931013325)--(0.245,0.52838813435522)--(0.246,0.528376728038178)--(0.247,0.528365087390258)--(0.248,0.528353208572911)--(0.249,0.528341087700513)--(0.25,0.528328720839987)--(0.251,0.528316104010414)--(0.252,0.528303233182642)--(0.253,0.528290104278901)--(0.254,0.528276713172408)--(0.255,0.528263055686976)--(0.256,0.528249127596619)--(0.257,0.528234924625161)--(0.258,0.528220442445831)--(0.259,0.528205676680871)--(0.26,0.528190622901135)--(0.261,0.528175276625687)--(0.262,0.528159633321399)--(0.263,0.528143688402549)--(0.264,0.528127437230411)--(0.265,0.528110875112856)--(0.266,0.528093997303939)--(0.267,0.528076799003492)--(0.268,0.528059275356713)--(0.269,0.528041421453758)--(0.27,0.528023232329326)--(0.271,0.528004702962245)--(0.271249200344376,0.528)--(0.272,0.52798582698231)--(0.273,0.527966600021089)--(0.274,0.527947017302628)--(0.275,0.527927073573424)--(0.276,0.527906763521862)--(0.277,0.527886081777798)--(0.278,0.527865022912128)--(0.279,0.527843581436367)--(0.28,0.527821751802222)--(0.281,0.527799528401162)--(0.282,0.527776905563992)--(0.283,0.527753877560428)--(0.284,0.527730438598658)--(0.285,0.527706582824923)--(0.286,0.527682304323079)--(0.287,0.527657597114167)--(0.288,0.527632455155985)--(0.289,0.527606872342651)--(0.29,0.527580842504175)--(0.291,0.527554359406024)--(0.292,0.527527416748689)--(0.293,0.527500008167253)--(0.294,0.527472127230957)--(0.295,0.52744376744277)--(0.296,0.527414922238952)--(0.297,0.527385584988623)--(0.298,0.527355748993329)--(0.299,0.527325407486615)--(0.3,0.527294553633588)--(0.301,0.527263180530487)--(0.302,0.527231281204253)--(0.303,0.527198848612101)--(0.304,0.527165875641088)--(0.305,0.527132355107688)--(0.306,0.527098279757361)--(0.307,0.527063642264133)--(0.308,0.527028435230166)--(0.308794687301504,0.527)--(0.309,0.526992649767933)--(0.31,0.526956273998918)--(0.311,0.526919305678419)--(0.312,0.52688173710335)--(0.313,0.526843560496296)--(0.314,0.526804768005096)--(0.315,0.526765351702424)--(0.316,0.526725303585373)--(0.317,0.526684615575038)--(0.318,0.526643279516111)--(0.319,0.526601287176466)--(0.32,0.526558630246758)--(0.321,0.526515300340021)--(0.322,0.526471288991267)--(0.323,0.526426587657092)--(0.324,0.526381187715287)--(0.325,0.526335080464445)--(0.326,0.526288257123585)--(0.327,0.526240708831768)--(0.328,0.526192426647723)--(0.329,0.526143401549483)--(0.33,0.526093624434013)--(0.331,0.526043086116857)--(0.331839802966598,0.526)--(0.332,0.525991774932453)--(0.333,0.525939670838641)--(0.334,0.525886776746539)--(0.335,0.525833083118669)--(0.336,0.525778580333877)--(0.337,0.525723258687001)--(0.338,0.525667108388552)--(0.339,0.525610119564394)--(0.34,0.525552282255443)--(0.341,0.525493586417362)--(0.342,0.525434021920272)--(0.343,0.525373578548471)--(0.344,0.525312246000157)--(0.345,0.525250013887166)--(0.346,0.525186871734717)--(0.347,0.525122808981165)--(0.348,0.525057814977769)--(0.348876897085433,0.525)--(0.349,0.524991875844209)--(0.35,0.524924960701821)--(0.351,0.524857080650011)--(0.352,0.524788224655148)--(0.353,0.524718381594824)--(0.354,0.524647540257675)--(0.355,0.524575689343224)--(0.356,0.524502817461739)--(0.357,0.524428913134095)--(0.358,0.524353964791663)--(0.359,0.524277960776204)--(0.36,0.524200889339779)--(0.361,0.524122738644681)--(0.362,0.524043496763372)--(0.36254154503539,0.524)--(0.363,0.523963133931099)--(0.364,0.52388163348786)--(0.365,0.523799003845465)--(0.366,0.52371523267201)--(0.367,0.523630307544982)--(0.368,0.523544215951303)--(0.369,0.523456945287391)--(0.37,0.523368482859247)--(0.371,0.523278815882556)--(0.372,0.523187931482816)--(0.373,0.523095816695482)--(0.374,0.523002458466132)--(0.374026004027719,0.523)--(0.375,0.522907790689178)--(0.376,0.522811849507549)--(0.377,0.522714622938633)--(0.378,0.52261609751592)--(0.379,0.522516259683432)--(0.38,0.522415095796038)--(0.381,0.522312592119784)--(0.382,0.522208734832259)--(0.383,0.522103510022981)--(0.383970980476892,0.522)--(0.384,0.521996901671725)--(0.385,0.521888828288002)--(0.386,0.521779342387952)--(0.387,0.52166842965073)--(0.388,0.521556075669035)--(0.389,0.521442265949675)--(0.39,0.52132698591416)--(0.391,0.521210220899332)--(0.392,0.521091956158022)--(0.392767885875158,0.521)--(0.393,0.520972156346949)--(0.394,0.520850756831517)--(0.395,0.520727809538562)--(0.396,0.520603299322043)--(0.397,0.520477210955018)--(0.398,0.520349529130522)--(0.399,0.520220238462464)--(0.4,0.520089323486573)--(0.400674091523599,0.52)--(0.401,0.519956733317677)--(0.402,0.51982241166491)--(0.403,0.519686414965604)--(0.404,0.519548727376633)--(0.405,0.519409332981399)--(0.406,0.519268215791016)--(0.407,0.519125359745522)--(0.407867005171201,0.519)--(0.408,0.518980731540103)--(0.409,0.518834217148093)--(0.41,0.518685910869089)--(0.411,0.518535796290817)--(0.412,0.518383856936239)--(0.413,0.518230076265039)--(0.414,0.518074437675145)--(0.414472879722747,0.518)--(0.415,0.517916843967522)--(0.416,0.517757282577552)--(0.417,0.517595808076008)--(0.418,0.517432403545226)--(0.419,0.517267052013854)--(0.42,0.517099736458656)--(0.420589428744156,0.517)--(0.421,0.516930367539981)--(0.422,0.516758892297621)--(0.423,0.516585396096511)--(0.424,0.516409861637974)--(0.425,0.516232271581143)--(0.426,0.516052608545065)--(0.426289723828056,0.516)--(0.427,0.515870711795363)--(0.428,0.515686643298283)--(0.429,0.515500443325725)--(0.43,0.51531209427054)--(0.431,0.515121578496878)--(0.431631243272265,0.515)--(0.432,0.514928795038733)--(0.433,0.514733661817356)--(0.434,0.514536302200361)--(0.435,0.514336698374011)--(0.436,0.51413483250873)--(0.436660794706143,0.514)--(0.437,0.513930601276244)--(0.438,0.513723900366919)--(0.439,0.513514876746716)--(0.44,0.513303512463574)--(0.441,0.513089789563923)--(0.44141586381074,0.513)--(0.442,0.512873526457227)--(0.443,0.512654745900067)--(0.444,0.512433545278128)--(0.445,0.512209906577891)--(0.445928502906854,0.512)--(0.446,0.51198378998651)--(0.447,0.511754910681248)--(0.448,0.51152353134304)--(0.449,0.511289633954153)--(0.45,0.511053200524919)--(0.4502228827572,0.511)--(0.451,0.510813951278875)--(0.452,0.510572047259878)--(0.453,0.510327544947289)--(0.454,0.510080426422439)--(0.454322379701288,0.51)--(0.455,0.50983042652375)--(0.456,0.509577649255749)--(0.457,0.509322193444125)--(0.458,0.509064041307091)--(0.458245803572712,0.509)--(0.459,0.508802877727448)--(0.46,0.508538877135777)--(0.461,0.508272118029355)--(0.462,0.508002582833323)--(0.462009500258433,0.508)--(0.463,0.507729833066648)--(0.464,0.507454258779338)--(0.465,0.507175846592447)--(0.465625599415096,0.507)--(0.466,0.506894410521194)--(0.467,0.506609812511661)--(0.468,0.506322315176249)--(0.469,0.506031901469803)--(0.469108921153049,0.506)--(0.47,0.505738123626566)--(0.471,0.505441333287162)--(0.472,0.505141565988647)--(0.472468029398117,0.505)--(0.473,0.504838533395623)--(0.474,0.504532243014697)--(0.475,0.50422291584212)--(0.475713975080109,0.504)--(0.476,0.503910381520362)--(0.477,0.503594385425385)--(0.478,0.503275293634514)--(0.478854638146682,0.503)--(0.479,0.502953007819889)--(0.48,0.50262710227945)--(0.481,0.502298043245511)--(0.481897276282025,0.502)--(0.482,0.501965754041705)--(0.483,0.501629737867954)--(0.484,0.50129051169263)--(0.484848572186807,0.501)--(0.485,0.500947965788066)--(0.486,0.500601640962266)--(0.487,0.500252051102707)--(0.487714677458035,0.5)--(0.488,0.499898994548362)--(0.489,0.49954216686589)--(0.49,0.499182020781953)--(0.490501253316498,0.499)--(0.491,0.49881819982614)--(0.492,0.498450679556859)--(0.493,0.498079789372413)--(0.493213508409156,0.498)--(0.494,0.497704951347697)--(0.495,0.497326553907223)--(0.495855503037997,0.497)--(0.496,0.49694462986055)--(0.497,0.496558631337501)--(0.498,0.496169177964575)--(0.498431043962189,0.496)--(0.499,0.495775817929803)--(0.5,0.495378636844254)--(0.50094508363487,0.495)--(0.501,0.494977911486584)--(0.502,0.494572913380342)--(0.503,0.494164382100131)--(0.503399384617472,0.494)--(0.504,0.493751809365258)--(0.505,0.493335342336916)--(0.505798672519167,0.493)--(0.506,0.492915129148404)--(0.507,0.492490642641172)--(0.508,0.492062552194099)--(0.508145141268039,0.492)--(0.509,0.491630092843649)--(0.51,0.49119387354606)--(0.510441250271081,0.491)--(0.511,0.490753506397714)--(0.512,0.490309081146448)--(0.512690196208477,0.49)--(0.513,0.489860700828724)--(0.514,0.489407995445866)--(0.514894138202515,0.489)--(0.515,0.488951498187084)--(0.516,0.488490441600009)--(0.517,0.4880256715171)--(0.517054899061847,0.488)--(0.518,0.487556249920709)--(0.519,0.487083039696213)--(0.519174390532905,0.487)--(0.52,0.486605256325351)--(0.521,0.48612354335012)--(0.521254856556032,0.486)--(0.522,0.485637302782382)--(0.523,0.485147028121237)--(0.523298010491001,0.485)--(0.524,0.484652237750594)--(0.525,0.484153346299578)--(0.525305474134919,0.484)--(0.526,0.483649916609895)--(0.527,0.483142357246355)--(0.527278782703752,0.483)--(0.528,0.482630202081306)--(0.529,0.482113927806015)--(0.5292193895626,0.482)--(0.53,0.481592964633925)--(0.531,0.481067932704527)--(0.531128670716201,0.481)--(0.532,0.480538082876721)--(0.533,0.480004254931798)--(0.533007929070617,0.48)--(0.534,0.479465443933047)--(0.534857947841042,0.479)--(0.535,0.478922619122864)--(0.536,0.478374943795889)--(0.536680260145762,0.478)--(0.537,0.477823044760755)--(0.538,0.47726648766197)--(0.538476008452536,0.477)--(0.539,0.476705450991559)--(0.54,0.476139990242948)--(0.540246250374975,0.476)--(0.541,0.475569757496885)--(0.541991964827666,0.475)--(0.542,0.474995366050862)--(0.543,0.474415893959144)--(0.54371336975998,0.474)--(0.544,0.473832218014193)--(0.545,0.473243800307561)--(0.545412101027781,0.473)--(0.546,0.472650794575468)--(0.547,0.472053426936)--(0.547089018023145,0.472)--(0.548,0.471451051770044)--(0.548744272454169,0.471)--(0.549,0.470844401219208)--(0.55,0.470232956429732)--(0.550379061409105,0.47)--(0.551,0.469616876928949)--(0.551994346591825,0.469)--(0.552,0.468996478808097)--(0.553,0.468370952951574)--(0.553589884313423,0.468)--(0.554,0.467741079719404)--(0.555,0.467106624205428)--(0.555167309439965,0.467)--(0.556,0.466467260220467)--(0.556726629920549,0.466)--(0.557,0.465823524087213)--(0.558,0.465175032680468)--(0.558268699280389,0.465)--(0.559,0.464521783130823)--(0.559794038003236,0.464)--(0.56,0.463864136081053)--(0.561,0.463201658031266)--(0.561303063618038,0.463)--(0.562,0.462534489526487)--(0.562796478936876,0.462)--(0.563,0.461862905445002)--(0.564,0.46118651012951)--(0.564274584722011,0.461)--(0.565,0.460505411465505)--(0.565737984946646,0.46)--(0.566,0.459819886693452)--(0.567,0.459129663371693)--(0.567187123323458,0.459)--(0.568,0.458434645963362)--(0.568622278788056,0.458)--(0.569,0.457735199659571)--(0.57,0.457031257092564)--(0.57004424588062,0.457)--(0.571,0.456322355211366)--(0.571452800049404,0.456)--(0.572,0.455609029517408)--(0.572848985668539,0.455)--(0.573,0.454891278580098)--(0.574,0.454168766275526)--(0.574232729253208,0.454)--(0.575,0.453441626286342)--(0.575604528866483,0.453)--(0.576,0.452710076725975)--(0.576964952961138,0.452)--(0.577,0.451974119690019)--(0.578,0.451233303834603)--(0.578313797075349,0.451)--(0.579,0.450488051397818)--(0.579651831720295,0.45)--(0.58,0.449738417762639)--(0.580979390748998,0.449)--(0.581,0.448984408592483)--(0.582,0.448225601061796)--(0.582296304433474,0.448)--(0.583,0.447462416577726)--(0.583603275865491,0.447)--(0.584,0.446694893279287)--(0.58490058821241,0.446)--(0.585,0.445923040373708)--(0.586,0.44514659580856)--(0.586188227767574,0.445)--(0.587,0.444365713043991)--(0.58746658102914,0.444)--(0.588,0.443580547509094)--(0.588736017075834,0.443)--(0.589,0.442791111869305)--(0.58999675917751,0.442)--(0.59,0.44199741965395)--(0.591,0.441199128226541)--(0.591248714398283,0.441)--(0.592,0.440396618515577)--(0.592492423258716,0.44)--(0.593,0.43958991115794)--(0.593728093698149,0.439)--(0.594,0.438779022965396)--(0.594955923895559,0.438)--(0.595,0.437963971571557)--(0.596,0.437144525936179)--(0.596175907980018,0.437)--(0.597,0.436320904647469)--(0.597388400324721,0.436)--(0.598,0.435493190219125)--(0.598593632964875,0.435)--(0.599,0.434661403313288)--(0.599791782178957,0.434)--(0.6,0.433825565351342)--(0.600983019455428,0.433)--(0.601,0.432985698500967)--(0.602,0.432141593228476)--(0.602167344551673,0.432)--(0.603,0.43129349325691)--(0.603345084240584,0.431)--(0.604,0.430441446380016)--(0.604516412985412,0.43)--(0.605,0.429585477412014)--(0.60568148372224,0.429)--(0.606,0.428725611831407)--(0.606840445309449,0.428)--(0.607,0.427861875764229)--(0.607993442648816,0.427)--(0.608,0.426994295966726)--(0.609,0.426122711010887)--(0.60914049578537,0.426)--(0.61,0.42524733866691)--(0.610281867865483,0.425)--(0.611,0.424368216114753)--(0.611417697772855,0.424)--(0.612,0.423485372233871)--(0.612548115137152,0.423)--(0.613,0.42259883643731)--(0.613673246210885,0.422)--(0.614,0.421708638652206)--(0.614793213968256,0.421)--(0.615,0.420814809299996)--(0.615908138200928,0.42)--(0.616,0.419917379276365)--(0.617,0.419016356735594)--(0.617018122488668,0.419)--(0.618,0.418111686460971)--(0.61812323274691,0.418)--(0.619,0.417203518753973)--(0.619223647380359,0.417)--(0.62,0.416291885984693)--(0.620319473940647,0.416)--(0.621,0.415376820894571)--(0.621410817253093,0.415)--(0.622,0.41445835657588)--(0.622497779494915,0.414)--(0.623,0.413536526451196)--(0.623580460270985,0.413)--(0.624,0.412611364252883)--(0.624658956687225,0.412)--(0.625,0.41168290400261)--(0.625733363421711,0.411)--(0.626,0.410751179990955)--(0.626803772793558,0.41)--(0.627,0.409816226757109)--(0.627870274829656,0.409)--(0.628,0.408878079068704)--(0.628932957329328,0.408)--(0.629,0.407936771901803)--(0.629991905926973,0.407)--(0.63,0.406992340421075)--(0.631,0.406044766044947)--(0.6310471793774,0.406)--(0.632,0.40509413408333)--(0.632098882811395,0.405)--(0.633,0.404140489288243)--(0.633147099860357,0.404)--(0.634,0.403183867193519)--(0.634191907881742,0.403)--(0.635,0.40222430342887)--(0.635233382364432,0.402)--(0.636,0.401261833701755)--(0.636271596980058,0.401)--(0.637,0.400296493779548)--(0.637306623632749,0.4)--(0.638,0.399328319472042)--(0.638338532507352,0.399)--(0.639,0.398357346614287)--(0.639367392116166,0.398)--(0.64,0.39738361104979)--(0.640393269344253,0.397)--(0.641,0.39640714861408)--(0.64141622949336,0.396)--(0.642,0.395427995118648)--(0.642436336324502,0.395)--(0.643,0.39444618633528)--(0.643453652099245,0.394)--(0.644,0.393461757980779)--(0.644468237619741,0.393)--(0.645,0.3924747457021)--(0.645480152267547,0.392)--(0.646,0.391485185061887)--(0.646489454041261,0.391)--(0.647,0.39049311152443)--(0.647496199593041,0.39)--(0.648,0.389498560442039)--(0.648500444264001,0.389)--(0.649,0.388501567041849)--(0.649502242118568,0.388)--(0.65,0.387502166413044)--(0.650501645977789,0.387)--(0.651,0.386500393494513)--(0.651498707451658,0.386)--(0.652,0.385496283062938)--(0.652493476970471,0.385)--(0.653,0.384489869721307)--(0.653486003815255,0.384)--(0.654,0.383481187887863)--(0.654476336147293,0.383)--(0.655,0.382470271785477)--(0.655464521036774,0.382)--(0.656,0.381457155431447)--(0.65645060449061,0.381)--(0.657,0.380441872627727)--(0.657434631479421,0.38)--(0.658,0.379424456951573)--(0.658416645963745,0.379)--(0.659,0.378404941746607)--(0.659396690919476,0.378)--(0.66,0.377383360114296)--(0.66037480836256,0.377)--(0.661,0.376359744905843)--(0.661351039372987,0.376)--(0.662,0.375334128714472)--(0.662325424118076,0.375)--(0.663,0.374306543868125)--(0.663298001875099,0.374)--(0.664,0.373277022422544)--(0.664268811053257,0.373)--(0.665,0.372245596154737)--(0.665237889215023,0.372)--(0.666,0.371212296556836)--(0.666205273096882,0.371)--(0.667,0.370177154830312)--(0.667170998629477,0.37)--(0.668,0.369140201880569)--(0.668135100957199,0.369)--(0.669,0.368101468311892)--(0.669097614457207,0.368)--(0.67,0.367060984422745)--(0.670058572757933,0.367)--(0.671,0.366018780201416)--(0.671018008757056,0.366)--(0.671975950159715,0.365)--(0.672,0.364974868750679)--(0.672932429632846,0.364)--(0.673,0.363929283262042)--(0.673887481439866,0.363)--(0.674,0.36288206541802)--(0.674841135735476,0.362)--(0.675,0.36183324398024)--(0.675793422025177,0.361)--(0.676,0.360782847380661)--(0.676744369180553,0.36)--(0.677,0.359730903719186)--(0.677694005454134,0.359)--(0.678,0.358677440761548)--(0.678642358493866,0.358)--(0.679,0.357622485937477)--(0.679589455357175,0.357)--(0.68,0.356566066339138)--(0.680535322524667,0.356)--(0.681,0.35550820871982)--(0.681479985913451,0.355)--(0.682,0.354448939492894)--(0.682423470890112,0.354)--(0.683,0.353388284730991)--(0.683365802283337,0.353)--(0.684,0.352326270165438)--(0.684307004396208,0.352)--(0.685,0.3512629211859)--(0.685247101018176,0.351)--(0.686,0.350198262840258)--(0.686186115436712,0.35)--(0.687,0.34913231983468)--(0.68712407044866,0.349)--(0.688,0.34806511653391)--(0.688060988371298,0.348)--(0.688996890690842,0.347)--(0.689,0.346996675329396)--(0.689931792149514,0.346)--(0.69,0.345926989570083)--(0.690865720986784,0.345)--(0.691,0.344856115160148)--(0.691798697702598,0.344)--(0.692,0.34378407512053)--(0.692730742368255,0.343)--(0.693,0.342710892133443)--(0.693661874635895,0.342)--(0.694,0.34163658854376)--(0.694592113747757,0.341)--(0.695,0.34056118636055)--(0.695521478545199,0.34)--(0.696,0.339484707258754)--(0.69644998747749,0.339)--(0.697,0.338407172580992)--(0.697377658610392,0.338)--(0.698,0.337328603339502)--(0.698304509634515,0.337)--(0.699,0.336249020218194)--(0.699230557873482,0.336)--(0.7,0.335168443574825)--(0.700155820291877,0.335)--(0.701,0.33408689344327)--(0.701080313503015,0.334)--(0.702,0.333004389535908)--(0.70200405377651,0.333)--(0.702927051352888,0.332)--(0.703,0.331920921260544)--(0.703849327478397,0.331)--(0.704,0.330836536778869)--(0.704770897753366,0.33)--(0.705,0.329751256548671)--(0.705691777137693,0.329)--(0.706,0.328665099020464)--(0.706611980282828,0.328)--(0.707,0.327578082335727)--(0.707531521538325,0.327)--(0.708,0.326490224329821)--(0.708450414958244,0.326)--(0.709,0.325401542534958)--(0.709368674307401,0.325)--(0.71,0.324312054183239)--(0.710286313067465,0.324)--(0.711,0.323221776209731)--(0.711203344442923,0.323)--(0.712,0.322130725255604)--(0.712119781366897,0.322)--(0.713,0.321038917671309)--(0.713035636506832,0.321)--(0.713950919507069,0.32)--(0.714,0.319946353089814)--(0.714865643457077,0.319)--(0.715,0.318853052408263)--(0.715779822316616,0.318)--(0.716,0.31775904326277)--(0.716693467737484,0.317)--(0.717,0.31666434086448)--(0.71760659113246,0.316)--(0.718,0.315568960149862)--(0.718519203680253,0.315)--(0.719,0.314472915784097)--(0.719431316330344,0.314)--(0.72,0.313376222164488)--(0.720342939807716,0.313)--(0.721,0.312278893423875)--(0.721254084617477,0.312)--(0.722,0.311180943434072)--(0.722164761049383,0.311)--(0.723,0.310082385809307)--(0.723074979182255,0.31)--(0.723984748255534,0.309)--(0.724,0.308983229710522)--(0.724894075568099,0.308)--(0.725,0.307883472222104)--(0.725802973786893,0.307)--(0.726,0.30678314722699)--(0.726711452182709,0.306)--(0.727,0.305682267330992)--(0.727619519836245,0.305)--(0.728,0.304580844899444)--(0.728527185641975,0.304)--(0.729,0.303478892060662)--(0.729434458311934,0.303)--(0.73,0.302376420709382)--(0.730341346379417,0.302)--(0.731,0.301273442510203)--(0.731247858202603,0.301)--(0.732,0.300169968901004)--(0.732154001968103,0.3)--(0.733,0.299066011096355)--(0.733059785694422,0.299)--(0.733965216178957,0.298)--(0.734,0.297961572270139)--(0.734870300459071,0.297)--(0.735,0.296856658060176)--(0.735775047933437,0.296)--(0.736,0.295751292721873)--(0.736679465975761,0.295)--(0.737,0.294645486596296)--(0.737583561808069,0.294)--(0.738,0.293539249817997)--(0.738487342503752,0.293)--(0.739,0.292432592318315)--(0.739390814990554,0.292)--(0.74,0.291325523828652)--(0.740293986053486,0.291)--(0.741,0.290218053883718)--(0.741196862337695,0.29)--(0.742,0.289110191824754)--(0.74209945035126,0.289)--(0.743,0.288001946802727)--(0.743001756467934,0.288)--(0.743903784806272,0.287)--(0.744,0.286893310236847)--(0.744805543686116,0.286)--(0.745,0.28578430878351)--(0.745707039129364,0.285)--(0.746,0.284674951354309)--(0.746608276994193,0.284)--(0.747,0.283565246356426)--(0.74750926301748,0.283)--(0.748,0.282455202022422)--(0.748410002817224,0.282)--(0.749,0.281344826413252)--(0.749310501894911,0.281)--(0.75,0.28023412742125)--(0.750210765637837,0.28)--(0.751,0.279123112773079)--(0.751110799321384,0.279)--(0.752,0.278011790032648)--(0.752010608111244,0.278)--(0.75291019559157,0.277)--(0.753,0.276900153140405)--(0.753809568106727,0.276)--(0.754,0.275788221764904)--(0.754708730681955,0.275)--(0.755,0.274676004608696)--(0.755607688061825,0.274)--(0.756,0.273563508599744)--(0.756506444891821,0.273)--(0.757,0.272450740517505)--(0.757405005720307,0.272)--(0.758,0.271337706995638)--(0.758303375000462,0.271)--(0.759,0.270224414524676)--(0.759201557092167,0.27)--(0.76,0.269110869454676)--(0.760099556263863,0.269)--(0.760997376661544,0.268)--(0.761,0.267997077669475)--(0.761895021202151,0.267)--(0.762,0.266883033365643)--(0.762792495172646,0.266)--(0.763,0.265768755199743)--(0.763689802490395,0.265)--(0.764,0.264654248973007)--(0.764586946990152,0.264)--(0.765,0.263539520359048)--(0.765483932425709,0.263)--(0.766,0.262424574906284)--(0.766380762471497,0.262)--(0.767,0.261309418040329)--(0.76727744072417,0.261)--(0.768,0.260194055066351)--(0.768173970704149,0.26)--(0.769,0.259078491171388)--(0.769070355857138,0.259)--(0.769966599249135,0.258)--(0.77,0.257962728027581)--(0.77086270388057,0.257)--(0.771,0.256846767117121)--(0.771758673645994,0.256)--(0.772,0.255730620590961)--(0.77265451170436,0.255)--(0.773,0.254614293174785)--(0.773550221147139,0.254)--(0.774,0.253497789487427)--(0.774445804999664,0.253)--(0.775,0.252381114042988)--(0.775341266222446,0.252)--(0.776,0.251264271252905)--(0.776236607712472,0.251)--(0.777,0.250147265428003)--(0.777131832304469,0.25)--(0.778,0.249030100780501)--(0.778026942772148,0.249)--(0.778921941308364,0.248)--(0.779,0.247912775006882)--(0.779816830948831,0.247)--(0.78,0.246795296654722)--(0.780711614475354,0.246)--(0.781,0.245677671907376)--(0.781606294427625,0.245)--(0.782,0.244559904592057)--(0.782500873290255,0.244)--(0.783,0.243441998447131)--(0.783395353493877,0.243)--(0.784,0.242323957123935)--(0.784289737416229,0.242)--(0.785,0.241205784188557)--(0.785184027383216,0.241)--(0.786,0.240087483123592)--(0.786078225669951,0.24)--(0.786972334364013,0.239)--(0.787,0.238969055461086)--(0.787866355412029,0.238)--(0.788,0.237850501304962)--(0.788760291344177,0.237)--(0.789,0.236731829294523)--(0.78965414424089,0.236)--(0.79,0.235613042586393)--(0.790547916136781,0.235)--(0.791,0.234494144262202)--(0.791441609021566,0.234)--(0.792,0.23337513733016)--(0.792335224840971,0.233)--(0.793,0.232256024726602)--(0.793228765497623,0.232)--(0.794,0.231136809317503)--(0.794122232851924,0.231)--(0.795,0.230017493899974)--(0.795015628722904,0.23)--(0.795908954565017,0.229)--(0.796,0.228898076297114)--(0.796802212409467,0.228)--(0.797,0.227778563482045)--(0.79769540401254,0.227)--(0.798,0.226658958908413)--(0.798588531034859,0.226)--(0.799,0.225539265104532)--(0.799481595099761,0.225)--(0.8,0.224419484537064)--(0.800374597794059,0.224)--(0.801,0.223299619612348)--(0.801267540668788,0.223)--(0.802,0.222179672677704)--(0.802160425239936,0.222)--(0.803,0.221059646022712)--(0.803053252989161,0.221)--(0.80394602522366,0.22)--(0.804,0.219939539520303)--(0.804838743375163,0.219)--(0.805,0.218819355546814)--(0.805731408969603,0.218)--(0.806,0.217699098604913)--(0.806624023357323,0.217)--(0.807,0.216578770758861)--(0.807516587857757,0.216)--(0.808,0.215458374021376)--(0.808409103760077,0.215)--(0.809,0.214337910354771)--(0.80930157232381,0.214)--(0.81,0.21321738167207)--(0.810193994779456,0.213)--(0.811,0.212096789838102)--(0.811086372329093,0.212)--(0.811978706106603,0.211)--(0.812,0.210976135919957)--(0.812870997144464,0.21)--(0.813,0.209855419506228)--(0.81376324673209,0.209)--(0.814,0.208734645439198)--(0.814655455962938,0.208)--(0.815,0.207613815395653)--(0.815547625904908,0.207)--(0.816,0.206492931009473)--(0.816439757600879,0.206)--(0.817,0.205371993872606)--(0.81733185206923,0.205)--(0.818,0.204251005536015)--(0.81822391030436,0.204)--(0.819,0.203129967510609)--(0.819115933277186,0.203)--(0.82,0.202008881268151)--(0.820007921935643,0.202)--(0.82089987707389,0.201)--(0.821,0.200887745489074)--(0.821791799726479,0.2)--(0.822,0.199766564250556)--(0.822683690785514,0.199)--(0.823,0.198645339130797)--(0.82357555111123,0.198)--(0.824,0.197524071448904)--(0.824467381543233,0.197)--(0.825,0.196402762489252)--(0.825359182900938,0.196)--(0.826,0.195281413502283)--(0.826250955984003,0.195)--(0.827,0.194160025705297)--(0.827142701572748,0.194)--(0.828,0.193038600283224)--(0.828034420428572,0.193)--(0.828926113225585,0.192)--(0.829,0.19191713677689)--(0.829817780731873,0.191)--(0.83,0.190795637274847)--(0.830709423687388,0.19)--(0.831,0.189674103636699)--(0.831601042781424,0.189)--(0.832,0.188552536921168)--(0.832492638686323,0.188)--(0.833,0.187430938158335)--(0.833384212057839,0.187)--(0.834,0.186309308350316)--(0.834275763535504,0.186)--(0.835,0.185187648471931)--(0.835167293742978,0.185)--(0.836,0.184065959471347)--(0.836058803288393,0.184)--(0.836950292732266,0.183)--(0.837,0.182944241417013)--(0.837841762650912,0.182)--(0.838,0.181822495123752)--(0.838733213647562,0.181)--(0.839,0.180700722498295)--(0.839624646271276,0.18)--(0.84,0.179578924384955)--(0.840516061057233,0.179)--(0.841,0.178457101604541)--(0.841407458527042,0.178)--(0.842,0.177335254954927)--(0.842298839189036,0.177)--(0.843,0.176213385211605)--(0.843190203538569,0.176)--(0.844,0.175091493128236)--(0.844081552058311,0.175)--(0.84497288520629,0.174)--(0.845,0.173969579074107)--(0.84586420341865,0.173)--(0.846,0.172847643083524)--(0.846755507179823,0.172)--(0.847,0.171725686968715)--(0.847646796924274,0.171)--(0.848,0.170603711398303)--(0.84853807307515,0.17)--(0.849,0.169481717021714)--(0.849429336044537,0.169)--(0.85,0.168359704469656)--(0.850320586233712,0.168)--(0.851,0.167237674354588)--(0.851211824033382,0.167)--(0.852,0.166115627271171)--(0.852103049823934,0.166)--(0.852994263973903,0.165)--(0.853,0.164993563737456)--(0.85388546681538,0.164)--(0.854,0.163871483343398)--(0.854776658732081,0.163)--(0.855,0.162749387727562)--(0.855667840065241,0.162)--(0.856,0.161627277415345)--(0.856559011146909,0.161)--(0.857,0.160505152916558)--(0.857450172300164,0.16)--(0.858,0.159383014725816)--(0.858341323839325,0.159)--(0.859,0.158260863322934)--(0.859232466070148,0.158)--(0.86,0.157138699173305)--(0.860123599290034,0.157)--(0.861,0.156016522728271)--(0.861014723788222,0.156)--(0.861905839827516,0.155)--(0.862,0.154894333706424)--(0.862796947698801,0.154)--(0.863,0.153772133186981)--(0.863688047670868,0.153)--(0.864,0.152649921695285)--(0.864579140002072,0.152)--(0.865,0.15152769962805)--(0.865470224943532,0.151)--(0.866,0.150405467369738)--(0.866361302739304,0.15)--(0.867,0.149283225292878)--(0.867252373626553,0.149)--(0.868,0.148160973758382)--(0.868143437835713,0.148)--(0.869,0.147038713115854)--(0.869034495590656,0.147)--(0.869925547099287,0.146)--(0.87,0.14591644327645)--(0.870816592579082,0.145)--(0.871,0.144794164831825)--(0.871707632239718,0.144)--(0.872,0.143671878302452)--(0.87259866628028,0.143)--(0.873,0.142549583994018)--(0.873489694894052,0.142)--(0.874,0.141427282202388)--(0.874380718268663,0.141)--(0.875,0.140304973213878)--(0.875271736586226,0.14)--(0.876,0.139182657305508)--(0.876162750023474,0.139)--(0.877,0.138060334745262)--(0.877053758751895,0.138)--(0.877944762933323,0.137)--(0.878,0.136938005557442)--(0.878835762729949,0.136)--(0.879,0.135815670022975)--(0.879726758302892,0.135)--(0.88,0.134693328619524)--(0.880617749803884,0.134)--(0.881,0.133570981580021)--(0.881508737380043,0.133)--(0.882,0.132448629129587)--(0.88239972117399,0.132)--(0.883,0.131326271485754)--(0.883290701323966,0.131)--(0.884,0.130203908858681)--(0.884181677963941,0.13)--(0.885,0.129081541451357)--(0.88507265122373,0.129)--(0.885963621227226,0.128)--(0.886,0.127959169347148)--(0.886854588094785,0.127)--(0.887,0.126836792639337)--(0.887745551948273,0.126)--(0.888,0.125714411742935)--(0.888636512901884,0.125)--(0.889,0.124592026833187)--(0.889527471066168,0.124)--(0.89,0.123469638079179)--(0.890418426548133,0.123)--(0.891,0.122347245644022)--(0.891309379451336,0.122)--(0.892,0.121224849685022)--(0.892200329875975,0.121)--(0.893,0.120102450353859)--(0.893091277918983,0.12)--(0.893982223673557,0.119)--(0.894,0.118980047757423)--(0.894873167228295,0.118)--(0.895,0.117857641884854)--(0.895764108673824,0.117)--(0.896,0.116735233080986)--(0.896655048094839,0.116)--(0.897,0.115612821475737)--(0.897545985573189,0.115)--(0.898,0.11449040719422)--(0.898436921187958,0.114)--(0.899,0.113367990356889)--(0.89932785501554,0.113)--(0.9,0.112245571079679)--(0.900218787129713,0.112)--(0.901,0.111123149474148)--(0.901109717601715,0.111)--(0.902,0.11000072564761)--(0.902000646500315,0.11)--(0.902891573890025,0.109)--(0.903,0.108878299542679)--(0.903782499836956,0.108)--(0.904,0.107755871431581)--(0.904673424402876,0.107)--(0.905,0.106633441409997)--(0.90556434764734,0.106)--(0.906,0.105511009568899)--(0.906455269627773,0.105)--(0.907,0.104388575995671)--(0.907346190399525,0.104)--(0.908,0.103266140774218)--(0.908237110015935,0.103)--(0.909,0.102143703985085)--(0.909128028528391,0.102)--(0.91,0.101021265705562)--(0.910018945986386,0.101)--(0.91090986243671,0.1)--(0.911,0.0998988259192909)--(0.911800777926045,0.099)--(0.912,0.0987763847777827)--(0.912691692498718,0.098)--(0.913,0.0976539423685624)--(0.913582606197185,0.097)--(0.914,0.0965314987564291)--(0.91447351906228,0.096)--(0.915,0.0954090540034533)--(0.915364431133261,0.095)--(0.916,0.0942866081690692)--(0.91625534244786,0.094)--(0.917,0.0931641613101651)--(0.917146253042333,0.093)--(0.918,0.092041713481171)--(0.918037162951503,0.092)--(0.918928072208439,0.091)--(0.919,0.0909192646867369)--(0.919818980845481,0.09)--(0.92,0.0897968150053186)--(0.920709888893622,0.089)--(0.921,0.0886743645097891)--(0.921600796382395,0.088)--(0.922,0.0875519132451777)--(0.922491703340114,0.087)--(0.923,0.0864294612544754)--(0.923382609793922,0.086)--(0.924,0.0853070085787074)--(0.92427351576982,0.085)--(0.925,0.0841845552570049)--(0.925164421292712,0.084)--(0.926,0.083062101326674)--(0.92605532638644,0.083)--(0.926946231073679,0.082)--(0.927,0.0819396468008893)--(0.927837135376281,0.081)--(0.928,0.080817191716452)--(0.92872803931526,0.08)--(0.929,0.0796947361295779)--(0.929618942910555,0.079)--(0.93,0.078572280070633)--(0.930509846181208,0.078)--(0.931,0.0774498235684889)--(0.931400749145404,0.077)--(0.932,0.0763273666505831)--(0.932291651820497,0.076)--(0.933,0.0752049093429722)--(0.933182554223039,0.075)--(0.934,0.0740824516703894)--(0.934073456368814,0.074)--(0.934964358272817,0.073)--(0.935,0.0729599936473697)--(0.93585525994934,0.072)--(0.936,0.0718375352888118)--(0.936746161412111,0.071)--(0.937,0.0707150766350797)--(0.937637062674098,0.07)--(0.938,0.0695926177059004)--(0.93852796374763,0.069)--(0.939,0.068470158519933)--(0.939418864644423,0.068)--(0.94,0.0673476990948161)--(0.940309765375602,0.067)--(0.941,0.0662252394472102)--(0.941200665951726,0.066)--(0.942,0.0651027795928394)--(0.942091566382809,0.065)--(0.942982466678332,0.064)--(0.943,0.063980319544047)--(0.943873366847253,0.063)--(0.944,0.0628578593054949)--(0.94476426689813,0.062)--(0.945,0.0617353989040599)--(0.945655166839008,0.061)--(0.946,0.060612938351961)--(0.946546066677489,0.06)--(0.947,0.0594904776606795)--(0.947436966420752,0.059)--(0.948,0.0583680168409942)--(0.948327866075571,0.058)--(0.949,0.0572455559030139)--(0.94921876564833,0.057)--(0.95,0.0561230948562086)--(0.950109665145045,0.056)--(0.951,0.0550006337094415)--(0.951000564571376,0.055)--(0.951891463932624,0.054)--(0.952,0.0538781724635601)--(0.952782363233811,0.053)--(0.953,0.052755711134846)--(0.953673262479627,0.052)--(0.954,0.05163324973045)--(0.954564161674466,0.051)--(0.955,0.0505107882569929)--(0.955455060822445,0.05)--(0.956,0.049388326720632)--(0.956345959927413,0.049)--(0.957,0.0482658651270865)--(0.957236858992965,0.048)--(0.958,0.0471434034816607)--(0.958127758022458,0.047)--(0.959,0.0460209417892663)--(0.959018657019017,0.046)--(0.959909555985547,0.045)--(0.96,0.0448984800515781)--(0.960800454924759,0.044)--(0.961,0.0437760182756373)--(0.961691353839162,0.043)--(0.962,0.0426535564658875)--(0.962582252731082,0.042)--(0.963,0.0415310946258189)--(0.96347315160267,0.041)--(0.964,0.0404086327586339)--(0.964364050455912,0.04)--(0.965,0.0392861708672665)--(0.96525494929264,0.039)--(0.966,0.0381637089543977)--(0.966145848114542,0.038)--(0.967,0.0370412470224731)--(0.967036746923164,0.037)--(0.967927645719926,0.036)--(0.968,0.0359187850728156)--(0.968818544506128,0.035)--(0.969,0.0347963231081385)--(0.969709443282956,0.034)--(0.97,0.0336738611307419)--(0.970600342051489,0.033)--(0.971,0.0325513991422374)--(0.971491240812707,0.032)--(0.972,0.0314289371440742)--(0.972382139567498,0.031)--(0.973,0.0303064751375516)--(0.973273038316665,0.03)--(0.974,0.02918401312383)--(0.97416393706093,0.029)--(0.975,0.028061551103942)--(0.975054835800942,0.028)--(0.975945734537281,0.027)--(0.976,0.0269390890785971)--(0.976836633270464,0.026)--(0.977,0.0258166270486905)--(0.977727532000951,0.025)--(0.978,0.0246941650151512)--(0.978618430729146,0.024)--(0.979,0.0235717029785805)--(0.979509329455408,0.023)--(0.98,0.0224492409395001)--(0.98040022818005,0.022)--(0.981,0.02132677889836)--(0.981291126903343,0.021)--(0.982,0.0202043168555465)--(0.982182025625523,0.02)--(0.983,0.0190818548113875)--(0.983072924346794,0.019)--(0.983963823067328,0.018)--(0.984,0.017959392766134)--(0.984854721787272,0.017)--(0.985,0.0168369307200127)--(0.98574562050675,0.016)--(0.986,0.015714468673272)--(0.986636519225864,0.015)--(0.987,0.0145920066260614)--(0.987527417944699,0.014)--(0.988,0.0134695445785011)--(0.988418316663323,0.013)--(0.989,0.0123470825306865)--(0.989309215381791,0.012)--(0.99,0.0112246204826915)--(0.990200114100146,0.011)--(0.991,0.0101021584345731)--(0.991091012818422,0.01)--(0.991981911536645,0.009)--(0.992,0.00897969638637242)--(0.992872810254832,0.008)--(0.993,0.00785723433811877)--(0.993763708972997,0.007)--(0.994,0.00673477228983717)--(0.99465460769115,0.006)--(0.995,0.00561231024154015)--(0.995545506409296,0.005)--(0.996,0.00448984819323573)--(0.996436405127438,0.004)--(0.997,0.0033673861449278)--(0.997327303845579,0.003)--(0.998,0.00224492409661867)--(0.998218202563719,0.002)--(0.999,0.00112246204830941)--(0.99910910128186,0.001)--(1,0);

%% file: q1-4.tex
\draw (0,0.621308330701774)--(0.001,0.621274457349217)--(0.002,0.621218937766363)--(0.003,0.621150629618249)--(0.004,0.621072417002724)--(0.0048370662409421,0.621)--(0.005,0.620985903342466)--(0.006,0.620892141528815)--(0.007,0.620791888431626)--(0.008,0.620685720369428)--(0.009,0.62057409421275)--(0.01,0.620457383077931)--(0.011,0.620335898747245)--(0.012,0.620209906542026)--(0.013,0.620079635621655)--(0.0135927715511943,0.62)--(0.014,0.619945286258341)--(0.015,0.619807035403924)--(0.016,0.619665041025125)--(0.017,0.619519444901709)--(0.018,0.619370375259115)--(0.019,0.619217948753745)--(0.02,0.619062272076969)--(0.0203920946419757,0.619)--(0.021,0.618903442930194)--(0.022,0.618741551830753)--(0.023,0.61857668278147)--(0.024,0.618408913679133)--(0.025,0.618238317163423)--(0.026,0.618064961169073)--(0.0263690152013847,0.618)--(0.027,0.617888908850537)--(0.028,0.617710220235263)--(0.029,0.617528952050025)--(0.03,0.617345157568108)--(0.031,0.617158887189901)--(0.0318420321672304,0.617)--(0.032,0.616970188498476)--(0.033,0.61677910596349)--(0.034,0.616585683670652)--(0.035,0.6163899626983)--(0.036,0.616191982226885)--(0.0369589449859207,0.616)--(0.037,0.61599177961377)--(0.038,0.615789389152048)--(0.039,0.615584846525539)--(0.04,0.615378184573855)--(0.041,0.615169434809388)--(0.0418037637008003,0.615)--(0.042,0.614958627121701)--(0.043,0.614745789343452)--(0.044,0.614530950989949)--(0.045,0.614314139040222)--(0.046,0.614095379504291)--(0.0464322376363159,0.614)--(0.047,0.613874696130431)--(0.048,0.613652113344813)--(0.049,0.613427655558403)--(0.05,0.613201345413504)--(0.0508825646676461,0.613)--(0.051,0.612973204497416)--(0.052,0.612743251816535)--(0.053,0.612511510289878)--(0.054,0.612277999845204)--(0.055,0.612042739814077)--(0.0551803658322272,0.612)--(0.056,0.611805746291194)--(0.057,0.611567039413865)--(0.058,0.611326637458766)--(0.059,0.611084557636417)--(0.0593469507830635,0.611)--(0.06,0.610840814277174)--(0.061,0.610595424637176)--(0.062,0.610348405822751)--(0.063,0.610099773276199)--(0.0633987620211425,0.61)--(0.064,0.609849539551988)--(0.065,0.609597719991069)--(0.066,0.609344330512059)--(0.067,0.609089385065734)--(0.067348532894062,0.609)--(0.068,0.608832894254425)--(0.069,0.608574872593113)--(0.07,0.608315334640371)--(0.071,0.608054293075804)--(0.0712068318913296,0.608)--(0.072,0.607791756225306)--(0.073,0.607527739009791)--(0.074,0.607262254314109)--(0.0749824473784131,0.607)--(0.075,0.606995313627636)--(0.076,0.606726922829919)--(0.077,0.606457098345543)--(0.078,0.60618585103717)--(0.0786816601052311,0.606)--(0.079,0.605913189612705)--(0.08,0.605639122126188)--(0.081,0.605363662851604)--(0.082,0.605086821795774)--(0.082312106647384,0.605)--(0.083,0.604808604210544)--(0.084,0.604529021997299)--(0.085,0.604248086625489)--(0.0858788881015951,0.604)--(0.086,0.603965806494387)--(0.087,0.603682185175286)--(0.088,0.603397237700298)--(0.089,0.603110972819605)--(0.0893859357369034,0.603)--(0.09,0.602823394414895)--(0.091,0.602534512461734)--(0.092,0.602244338165792)--(0.0928383536401176,0.602)--(0.093,0.601952878354348)--(0.094,0.601660135352857)--(0.095,0.601366123757776)--(0.096,0.601070851301354)--(0.0962389714149764,0.601)--(0.097,0.600774318904156)--(0.098,0.600476538385665)--(0.099,0.600177519169845)--(0.0995912812432469,0.6)--(0.1,0.599877264697739)--(0.101,0.599575780188729)--(0.102,0.599273078064635)--(0.102898557758996,0.599)--(0.103,0.59896916422674)--(0.104,0.598664037483119)--(0.105,0.598357713207212)--(0.106,0.598050197979596)--(0.10616263682716,0.598)--(0.107,0.597741489503131)--(0.108,0.597431600967221)--(0.109,0.597120540345462)--(0.109386112293,0.597)--(0.11,0.59680830706723)--(0.111,0.596494909505586)--(0.112,0.596180357885867)--(0.11257134289385,0.596)--(0.113,0.595864653178336)--(0.114,0.595547799354733)--(0.115,0.59522980872104)--(0.115720167892867,0.595)--(0.116,0.594910683566454)--(0.117,0.594590423943566)--(0.118,0.594269044031686)--(0.118834285811283,0.594)--(0.119,0.593946547180541)--(0.12,0.593622930073775)--(0.121,0.593298208520333)--(0.121915268650738,0.593)--(0.122,0.592972386635538)--(0.123,0.59264545835307)--(0.124,0.592317440831333)--(0.12496457435001,0.592)--(0.125,0.591988338619658)--(0.126,0.591658143590236)--(0.127,0.591326873934131)--(0.127983557733665,0.591)--(0.128,0.590994534266352)--(0.129,0.590661115156183)--(0.13,0.590326635473816)--(0.13097348016735,0.59)--(0.131,0.58999109949477)--(0.132,0.589654497314669)--(0.133,0.589316848092452)--(0.133935518099508,0.589)--(0.134,0.588978155322059)--(0.135,0.588638409525862)--(0.136,0.588297629724224)--(0.136870770640798,0.588)--(0.137,0.587955818150412)--(0.138,0.587612966725552)--(0.139,0.587269093867077)--(0.139780266308986,0.587)--(0.14,0.586924200031413)--(0.141,0.586578279582441)--(0.142,0.586231349833188)--(0.142664969047703,0.586)--(0.143,0.585883408909946)--(0.144,0.585534454735675)--(0.145,0.58518450298034)--(0.145525783611315,0.585)--(0.146,0.584833548849623)--(0.147,0.584481595014515)--(0.148,0.584128654926018)--(0.148363560394773,0.584)--(0.149,0.583774720241508)--(0.15,0.583419799641835)--(0.151,0.583063903745847)--(0.151179099776048,0.583)--(0.152,0.582707019997671)--(0.153,0.582349164422934)--(0.153973095617772,0.582)--(0.154,0.58199034361174)--(0.155,0.581630541730969)--(0.156,0.581269781920998)--(0.156745883389615,0.581)--(0.157,0.580908062353836)--(0.158,0.580545375922274)--(0.159,0.580181741620389)--(0.159498551495515,0.58)--(0.16,0.579817151956693)--(0.161,0.579451610076253)--(0.162,0.579085130078831)--(0.162231738808836,0.579)--(0.163,0.578717698009794)--(0.164,0.578349328866683)--(0.164945938436929,0.578)--(0.165,0.577980029825692)--(0.166,0.577609783363231)--(0.167,0.577238614272182)--(0.167641333915938,0.577)--(0.168,0.57686651719139)--(0.169,0.576493488257495)--(0.17,0.576119545703152)--(0.170318961265426,0.576)--(0.171,0.575744675902038)--(0.172,0.575368890444972)--(0.172979298134689,0.575)--(0.173,0.574992199601528)--(0.174,0.57461458210256)--(0.175,0.574236065303794)--(0.175622237504887,0.574)--(0.176,0.573856642379678)--(0.177,0.573476309635601)--(0.178,0.573095085944662)--(0.178248884946399,0.573)--(0.179,0.572712954157847)--(0.18,0.572329930143965)--(0.180859428624179,0.572)--(0.181,0.571946019484524)--(0.182,0.571561205153712)--(0.183,0.571175513166909)--(0.183454086215761,0.571)--(0.184,0.570788930971894)--(0.185,0.570401463542434)--(0.186,0.570013126230744)--(0.186033733146335,0.57)--(0.187,0.569623893845224)--(0.188,0.569233795543702)--(0.188598057230526,0.569)--(0.189,0.568842823100176)--(0.19,0.568450973054036)--(0.191,0.568058265503789)--(0.191148075057949,0.568)--(0.192,0.567664677235369)--(0.193,0.567270231726563)--(0.193683668409159,0.567)--(0.194,0.566874926203603)--(0.195,0.566478751546672)--(0.196,0.56608173124424)--(0.196205462713738,0.566)--(0.197,0.565683842564413)--(0.198,0.565285106277988)--(0.198713570796065,0.565)--(0.199,0.564885522044972)--(0.2,0.564485078178757)--(0.201,0.564083800041634)--(0.201208438788213,0.564)--(0.202,0.563681663728623)--(0.203,0.563278690594767)--(0.203690218060147,0.563)--(0.204,0.562874879395188)--(0.205,0.56247021928861)--(0.206,0.562064735834073)--(0.206159359426221,0.562)--(0.207,0.561658402137326)--(0.208,0.561251243749359)--(0.208615885785312,0.561)--(0.209,0.560843254859462)--(0.21,0.560434429205702)--(0.211,0.560024790704856)--(0.211060412208314,0.56)--(0.212,0.559614307495377)--(0.213,0.559203013258413)--(0.213492688444993,0.559)--(0.214,0.55879089363232)--(0.215,0.558377950992327)--(0.21591350544335,0.558)--(0.216,0.557964202358617)--(0.217,0.557549618519181)--(0.218,0.557134235786319)--(0.218322594388016,0.557)--(0.219,0.556718030189084)--(0.22,0.556301017123158)--(0.22072051839043,0.556)--(0.221,0.555883200095632)--(0.222,0.555464563593291)--(0.223,0.55504513779065)--(0.223107439110234,0.555)--(0.224,0.554624888920497)--(0.225,0.554203850109191)--(0.22548331965467,0.554)--(0.226,0.553782006586164)--(0.227,0.553359361373757)--(0.227848731341095,0.553)--(0.228,0.552935929834618)--(0.229,0.552511684716344)--(0.23,0.552086663071754)--(0.230203570416438,0.552)--(0.231,0.551660833039583)--(0.232,0.551234221343624)--(0.232548095538251,0.551)--(0.233,0.550806819021663)--(0.234,0.550378623503189)--(0.234882672692377,0.55)--(0.235,0.549949655121085)--(0.236,0.54951988190522)--(0.237,0.549089344325313)--(0.23720719058493,0.549)--(0.238,0.548658008691709)--(0.239,0.548225902740726)--(0.239521937826408,0.548)--(0.24,0.54779301579628)--(0.241,0.547359347353323)--(0.241827214827981,0.547)--(0.242,0.546924914948462)--(0.243,0.546489689796622)--(0.244,0.546053711880726)--(0.244123012817143,0.546)--(0.245,0.545616941506261)--(0.246,0.545179415720071)--(0.246409427669621,0.545)--(0.247,0.54474111372396)--(0.248,0.544302045622851)--(0.248686811570144,0.544)--(0.249,0.543862217501371)--(0.25,0.54342161255167)--(0.250955297030264,0.543)--(0.251,0.542980263703824)--(0.252,0.542538127284665)--(0.253,0.542095252191669)--(0.253214760587166,0.542)--(0.254,0.541651600419089)--(0.255,0.541207202119906)--(0.255465544972716,0.541)--(0.256,0.540762042374789)--(0.257,0.540316126041777)--(0.257707825854842,0.54)--(0.258,0.539869463397586)--(0.259,0.539422034122024)--(0.259941722850367,0.539)--(0.26,0.538973873562564)--(0.261,0.538524936356287)--(0.262,0.538075273791562)--(0.262167165719538,0.538)--(0.263,0.537624842574257)--(0.264,0.537173679819686)--(0.264384404349578,0.537)--(0.265,0.536721762442742)--(0.266,0.536269104321807)--(0.266593615514949,0.536)--(0.267,0.535815705468643)--(0.268,0.535361556730775)--(0.268794907583868,0.535)--(0.269,0.53490668100186)--(0.27,0.534451046323884)--(0.270988386464528,0.534)--(0.271,0.533994698238108)--(0.272,0.533537582225658)--(0.273,0.53307975612855)--(0.273173972451429,0.533)--(0.274,0.532621173410561)--(0.275,0.5321618731432)--(0.275351949792715,0.532)--(0.276,0.531701828705635)--(0.277,0.531241058704945)--(0.27752242934929,0.531)--(0.278,0.53077955679307)--(0.279,0.530317321429431)--(0.279685507463501,0.53)--(0.28,0.529854366212707)--(0.281,0.529390669791208)--(0.281841278374955,0.529)--(0.282,0.528926265364478)--(0.283,0.528461112126136)--(0.283989834280871,0.528)--(0.284,0.527995262510779)--(0.285,0.527528656633732)--(0.286,0.527061357483128)--(0.286131136622467,0.527)--(0.287,0.526593311379453)--(0.288,0.526124566148989)--(0.288265402155865,0.526)--(0.289,0.525655084296924)--(0.29,0.525184897015855)--(0.290392727011004,0.525)--(0.291,0.524713983190107)--(0.292,0.524242357828742)--(0.292513195428865,0.524)--(0.293,0.523770015735424)--(0.294,0.523296956206185)--(0.294626889891469,0.523)--(0.295,0.522823189483816)--(0.296,0.522348699642275)--(0.296733891170077,0.522)--(0.297,0.52187351186276)--(0.298,0.521397595508651)--(0.298834278371758,0.521)--(0.299,0.52092099017823)--(0.3,0.520443651056437)--(0.300928128984362,0.52)--(0.301,0.519965631616617)--(0.302,0.519486873418137)--(0.303,0.519007442154085)--(0.303015504996514,0.519)--(0.304,0.518527269609481)--(0.305,0.518046425140194)--(0.305096436853719,0.518)--(0.306,0.517564846531224)--(0.307,0.517082592420301)--(0.307171062327869,0.517)--(0.308,0.516599610970912)--(0.309,0.516115950731254)--(0.309239452587842,0.516)--(0.31,0.515631569604595)--(0.311,0.51514650669927)--(0.311301677395597,0.515)--(0.312,0.514660728998503)--(0.313,0.514174266841594)--(0.313357805142591,0.514)--(0.314,0.513687095610686)--(0.315,0.513199237568113)--(0.315407902885015,0.513)--(0.316,0.512710675792612)--(0.317,0.51222142518294)--(0.317452036377912,0.512)--(0.318,0.511731475790728)--(0.319,0.511240835885955)--(0.319490270108198,0.511)--(0.32,0.510749501747985)--(0.321,0.510257475774313)--(0.321522667326645,0.51)--(0.322,0.509764759705329)--(0.323,0.509271350843915)--(0.323549290078857,0.509)--(0.324,0.508777255603157)--(0.325,0.508282466990848)--(0.325570199235278,0.508)--(0.326,0.507786995282737)--(0.327,0.507290830012792)--(0.327585454520274,0.507)--(0.328,0.5067939844876)--(0.329,0.506296445610392)--(0.329595114540308,0.506)--(0.33,0.505798228864899)--(0.331,0.5052993193886)--(0.331599236811267,0.505)--(0.332,0.504799733966734)--(0.333,0.504299456857992)--(0.333597877784942,0.504)--(0.334,0.503798505251452)--(0.335,0.503296863436044)--(0.335591092874719,0.503)--(0.336,0.502794548084916)--(0.337,0.502291544448393)--(0.337578936480487,0.502)--(0.338,0.501787867741742)--(0.339,0.501283505130058)--(0.339561462012816,0.501)--(0.34,0.500778469406513)--(0.341,0.500272750626642)--(0.341538721916398,0.5)--(0.342,0.499766358174968)--(0.343,0.499259285995503)--(0.343510767692815,0.499)--(0.344,0.498751539055157)--(0.345,0.4982431162069)--(0.345477649922626,0.498)--(0.346,0.497734016968575)--(0.347,0.497224246145117)--(0.347439418286817,0.497)--(0.348,0.496713796751275)--(0.349,0.496202680609555)--(0.349396121587623,0.496)--(0.35,0.495690883154949)--(0.351,0.495178424315808)--(0.35134780776876,0.495)--(0.352,0.494665280847987)--(0.353,0.494151481896711)--(0.353294523935063,0.494)--(0.354,0.493636994416522)--(0.355,0.493121857903368)--(0.355236316371578,0.493)--(0.356,0.492606028365438)--(0.357,0.492089556806152)--(0.357173230562107,0.492)--(0.358,0.491572387119365)--(0.359,0.491054582995693)--(0.359105311207233,0.491)--(0.36,0.490536075023653)--(0.361,0.490016940783832)--(0.361032602241836,0.49)--(0.362,0.489497096345321)--(0.362955117048831,0.489)--(0.363,0.488976629949833)--(0.364,0.488455455273987)--(0.36487290391207,0.488)--(0.365,0.487933655268213)--(0.366,0.487411155922777)--(0.366786026543898,0.487)--(0.367,0.486888024012595)--(0.368,0.48636420232922)--(0.368694526078461,0.486)--(0.369,0.485839740177167)--(0.37,0.485314598456113)--(0.370598442961716,0.485)--(0.371,0.484788807681801)--(0.372,0.484262348192381)--(0.37249781696635,0.484)--(0.373,0.483735230372893)--(0.374,0.483207455353905)--(0.374392687206312,0.483)--(0.375,0.482679012024193)--(0.376,0.482149923684343)--(0.37628309215094,0.482)--(0.377,0.481620156337608)--(0.378,0.481089756855929)--(0.378169069638721,0.481)--(0.379,0.480558666943994)--(0.38,0.480026958470253)--(0.380050656890683,0.48)--(0.381,0.479494547403932)--(0.381927846825946,0.479)--(0.382,0.478961524153128)--(0.383,0.478427801208485)--(0.383800687000497,0.478)--(0.384,0.477893459028673)--(0.385,0.477358431779937)--(0.38566924392375,0.477)--(0.386,0.476822771977957)--(0.387,0.476286442472521)--(0.387533552542927,0.476)--(0.388,0.475749466315395)--(0.389,0.475211836573127)--(0.389393647247431,0.475)--(0.39,0.474673545288393)--(0.391,0.474134617301997)--(0.391249561880345,0.474)--(0.392,0.473595012078033)--(0.393,0.473054787813408)--(0.393101329749641,0.473)--(0.394,0.472513869799742)--(0.394948954596055,0.472)--(0.395,0.471972345224157)--(0.396,0.471430121503957)--(0.396792438835563,0.471)--(0.397,0.470887285957521)--(0.398,0.470343770176765)--(0.398631872526459,0.47)--(0.399,0.469799624710061)--(0.4,0.469254818740534)--(0.400467286996969,0.469)--(0.401,0.468709364366323)--(0.402,0.468163270054529)--(0.402298713090336,0.468)--(0.403,0.467616507748035)--(0.404,0.46706912691552)--(0.404126181174487,0.467)--(0.405,0.466521057614704)--(0.405949693993898,0.466)--(0.406,0.46597238584493)--(0.407,0.465423016664201)--(0.407769239055424,0.465)--(0.408,0.464873039388154)--(0.409,0.464322387533331)--(0.409584914201709,0.464)--(0.41,0.463771105593688)--(0.411,0.463219172798392)--(0.411396748035218,0.463)--(0.412,0.462666587001468)--(0.413,0.462113374975727)--(0.413204768726524,0.462)--(0.414,0.461559486091716)--(0.415,0.461004996522249)--(0.415009004022694,0.461)--(0.416,0.460449805285476)--(0.416809383586868,0.46)--(0.417,0.459894015016061)--(0.418,0.459337546945127)--(0.418606027359516,0.459)--(0.419,0.458780455490921)--(0.42,0.45822271337489)--(0.420398966560965,0.458)--(0.421,0.457664321378109)--(0.422,0.457105306821323)--(0.422188227359099,0.457)--(0.423,0.45654561488909)--(0.42397382256305,0.456)--(0.424,0.455985325948481)--(0.425,0.455424338178367)--(0.425755696544466,0.455)--(0.426,0.454862750270037)--(0.427,0.454300493343945)--(0.42753396856795,0.454)--(0.428,0.453737606984676)--(0.429,0.453174082427795)--(0.429308663331248,0.453)--(0.43,0.452609898099983)--(0.431,0.452045107416297)--(0.431079805173421,0.452)--(0.432,0.451479625568149)--(0.432847345578738,0.451)--(0.433,0.450913548801942)--(0.434,0.450346791286403)--(0.434611342807044,0.45)--(0.435,0.449779417716031)--(0.436,0.449211397097437)--(0.436371858460496,0.449)--(0.437,0.448642727082806)--(0.438,0.448073444789813)--(0.438128915532768,0.448)--(0.439,0.447503478657266)--(0.43988248269495,0.447)--(0.44,0.446932919048293)--(0.441,0.446361674140857)--(0.441632577010137,0.446)--(0.442,0.445789818936828)--(0.443,0.445217315181864)--(0.443379280337073,0.445)--(0.444,0.444644164622208)--(0.445,0.444070403375786)--(0.445122614413754,0.444)--(0.446,0.443495957666931)--(0.446862539569559,0.443)--(0.447,0.442920919672112)--(0.448,0.442345199581238)--(0.448599083666993,0.442)--(0.449,0.441768866769218)--(0.45,0.441191891823468)--(0.45033232259433,0.441)--(0.451,0.440614264314993)--(0.452,0.440036035800406)--(0.452062276916369,0.44)--(0.453,0.439457113683726)--(0.453788876113038,0.439)--(0.454,0.438877600222001)--(0.455,0.438297416198629)--(0.455512204711851,0.438)--(0.456,0.437716608271961)--(0.457,0.437135173132164)--(0.457232309494454,0.437)--(0.458,0.436553070744206)--(0.458949188589167,0.436)--(0.459,0.435970377667187)--(0.46,0.435386988828716)--(0.46066278488839,0.435)--(0.461,0.434803001298629)--(0.462,0.434218363665473)--(0.462373215847945,0.434)--(0.463,0.433633081602039)--(0.464,0.433047196344758)--(0.46408050016141,0.433)--(0.465,0.432460619635924)--(0.46578456872569,0.432)--(0.466,0.431873452753714)--(0.467,0.431285616409491)--(0.467485495191973,0.431)--(0.468,0.430697154704242)--(0.469,0.43010807288293)--(0.469183330542898,0.43)--(0.47,0.429518316159495)--(0.470878044097639,0.429)--(0.471,0.428927969615682)--(0.472,0.428336937999488)--(0.472569629292139,0.428)--(0.473,0.427745296126462)--(0.474,0.427153021055844)--(0.474258176827003,0.427)--(0.475,0.426560083575862)--(0.475943681900166,0.426)--(0.476,0.425966556502432)--(0.477,0.425372332716152)--(0.477626083199346,0.425)--(0.478,0.42477750959183)--(0.479,0.424182044251792)--(0.479305498311158,0.424)--(0.48,0.423585924716663)--(0.480981936761259,0.423)--(0.481,0.42298921568886)--(0.482,0.422391802502788)--(0.482655305945723,0.422)--(0.483,0.421793796487705)--(0.484,0.421195143528783)--(0.484325738518226,0.421)--(0.485,0.420595840084692)--(0.485993247713595,0.42)--(0.486,0.419995947122346)--(0.487,0.419395346980466)--(0.487657731204033,0.419)--(0.488,0.418794155880371)--(0.489,0.418192317628889)--(0.489319325843015,0.418)--(0.49,0.417589827867788)--(0.490978038832106,0.417)--(0.491,0.416986748436209)--(0.492,0.416382963461128)--(0.492633777912844,0.416)--(0.493,0.415778584527695)--(0.494,0.415173562990578)--(0.49428667416533,0.415)--(0.495,0.41456788395073)--(0.495936719838611,0.414)--(0.496,0.413961614961146)--(0.497,0.4133546469587)--(0.497583850869526,0.413)--(0.498,0.412747076889801)--(0.499,0.412138873759049)--(0.499228183431611,0.412)--(0.5,0.411530001925125)--(0.500869686675523,0.411)--(0.501,0.410920539734683)--(0.502,0.410310390198222)--(0.502508341293026,0.41)--(0.503,0.409699625137752)--(0.504,0.409088241794607)--(0.504144240207019,0.409)--(0.505,0.408476173097341)--(0.50577732205047,0.408)--(0.506,0.407863513508661)--(0.507,0.407250183623019)--(0.507407627358214,0.407)--(0.508,0.406636219160284)--(0.509,0.406021656679083)--(0.509035218199745,0.406)--(0.51,0.405406386494485)--(0.510659995953558,0.405)--(0.511,0.404790524754774)--(0.512,0.404174015399939)--(0.512282074703483,0.404)--(0.513,0.403556846568589)--(0.513901446884949,0.403)--(0.514,0.402939085911907)--(0.515,0.402320628868442)--(0.515518066148869,0.402)--(0.516,0.401701559668157)--(0.517,0.401081871422641)--(0.517132036913146,0.401)--(0.518,0.400461492699172)--(0.518743290274783,0.4)--(0.519,0.39984052127861)--(0.52,0.39921888469799)--(0.520351878641693,0.399)--(0.521,0.398596602169061)--(0.521957842856146,0.398)--(0.522,0.397973726547215)--(0.523,0.397350140614771)--(0.523561098643978,0.397)--(0.524,0.39672595221658)--(0.525,0.396101136188981)--(0.525161768122874,0.396)--(0.526,0.395475633902578)--(0.526759789724128,0.395)--(0.527,0.394849537440821)--(0.528,0.394222771103356)--(0.52835519955705,0.394)--(0.529,0.393595358951209)--(0.529948042833033,0.393)--(0.53,0.392967352100464)--(0.531,0.392338634236456)--(0.531538240216691,0.392)--(0.532,0.391709309809566)--(0.533,0.391079362675202)--(0.533125909946153,0.391)--(0.534,0.390448719427218)--(0.534710978250364,0.39)--(0.535,0.389817480184459)--(0.536,0.389185580256967)--(0.536293504557665,0.389)--(0.537,0.388553020171892)--(0.537873500055863,0.388)--(0.538,0.387919863439285)--(0.539,0.387286009303337)--(0.53945092766923,0.387)--(0.54,0.386651529622654)--(0.541,0.386016446762849)--(0.541025882968145,0.386)--(0.542,0.385380642754972)--(0.542598263032223,0.385)--(0.543,0.384744240586119)--(0.544,0.384107201945027)--(0.54416818498874,0.384)--(0.545,0.383469473206194)--(0.545735592742155,0.383)--(0.546,0.382831145521726)--(0.547,0.382192149486226)--(0.547300523757575,0.382)--(0.548,0.381552492903309)--(0.548862997273531,0.381)--(0.549,0.380912236540026)--(0.55,0.380271281419211)--(0.550422978868139,0.38)--(0.551,0.379629693742819)--(0.551980555513714,0.379)--(0.552,0.378987505400837)--(0.553,0.378344589425983)--(0.553535628346041,0.378)--(0.554,0.377701067269506)--(0.555,0.377056922427116)--(0.555088321376588,0.377)--(0.556,0.376412064835797)--(0.556638548681352,0.376)--(0.557,0.375766604674415)--(0.558,0.375120496861265)--(0.558186392237862,0.375)--(0.559,0.374473698623065)--(0.559731814860026,0.374)--(0.56,0.373826296792697)--(0.561,0.373178224042505)--(0.561274847500459,0.373)--(0.562,0.37252948140513)--(0.562815500394435,0.372)--(0.563,0.371880134101344)--(0.564,0.37123009436976)--(0.564353759879327,0.371)--(0.565,0.370579403439923)--(0.565889677353045,0.37)--(0.566,0.369928106716794)--(0.567,0.369276097881246)--(0.567423200662364,0.369)--(0.568,0.368623454623494)--(0.568954416389257,0.368)--(0.569,0.36797020439241)--(0.57,0.367316224251924)--(0.570483239738773,0.367)--(0.571,0.366661624487412)--(0.572,0.366006414011511)--(0.572009784383805,0.366)--(0.573,0.36535046279082)--(0.573533945626633,0.365)--(0.574,0.364693902196038)--(0.575,0.364036717634296)--(0.575055842987561,0.364)--(0.576,0.363378802438223)--(0.576575385499696,0.363)--(0.577,0.362720276543667)--(0.578,0.362061115495873)--(0.578092669929405,0.362)--(0.579,0.361401231762741)--(0.579607625213438,0.361)--(0.58,0.36074073595153)--(0.581,0.360079595938292)--(0.581120330355309,0.36)--(0.582,0.359417738958227)--(0.582630729330385,0.359)--(0.583,0.358755268464663)--(0.584,0.358092146927104)--(0.584138888132645,0.358)--(0.585,0.357428311840559)--(0.585644761144734,0.357)--(0.586,0.356763861748635)--(0.587,0.356098756048059)--(0.587148405874405,0.356)--(0.588,0.355432937844285)--(0.588649782706293,0.355)--(0.589,0.354766503086131)--(0.59,0.354099410503665)--(0.590148944962749,0.354)--(0.591,0.353431604019123)--(0.591645854843761,0.353)--(0.592,0.352763179373392)--(0.593,0.352094097109591)--(0.593140565571899,0.352)--(0.594,0.351424297026303)--(0.594633037187351,0.351)--(0.595,0.350753877116501)--(0.596,0.350082802290936)--(0.596123326690398,0.35)--(0.597,0.349411003134773)--(0.597611388190805,0.349)--(0.598,0.34873858242752)--(0.599,0.348065512078329)--(0.59909728614276,0.348)--(0.6,0.347391708217237)--(0.600580965152779,0.347)--(0.601,0.34671728102047)--(0.602,0.346042212103878)--(0.602062500610512,0.346)--(0.603,0.345366397746042)--(0.603541824237657,0.345)--(0.604,0.344689958207149)--(0.605,0.344012887596968)--(0.605019025652663,0.344)--(0.606,0.343335056788902)--(0.606494020495769,0.343)--(0.607,0.342656598892786)--(0.607966909002524,0.342)--(0.608,0.341977513570468)--(0.609,0.341297670004451)--(0.609437607883065,0.341)--(0.61,0.340617187571535)--(0.61090620543714,0.34)--(0.611,0.339936075721453)--(0.612,0.339254221637636)--(0.61237263928023,0.339)--(0.613,0.338571708321781)--(0.61383697127396,0.338)--(0.614,0.337888563534136)--(0.615,0.337204695514921)--(0.61529916651128,0.337)--(0.616,0.336520144801287)--(0.616759257855443,0.336)--(0.617,0.335834960496667)--(0.618,0.335149075039329)--(0.618217240361629,0.335)--(0.619,0.33446248024215)--(0.619673115492949,0.334)--(0.62,0.333775249669274)--(0.621,0.333087343185285)--(0.621126910595662,0.333)--(0.622,0.332398697445571)--(0.622578593484096,0.332)--(0.623,0.331709413678962)--(0.624,0.33101948249328)--(0.624028225973809,0.331)--(0.625,0.330328778776441)--(0.625475740129637,0.33)--(0.626,0.329637434714012)--(0.626921219772569,0.329)--(0.627,0.328945449826886)--(0.628,0.328252706157722)--(0.628364602749861,0.328)--(0.629,0.327559294518287)--(0.629805947128881,0.327)--(0.63,0.326865239643041)--(0.631,0.326170461064638)--(0.631245227700532,0.326)--(0.632,0.325474974385327)--(0.632682459231567,0.325)--(0.633,0.324778841989061)--(0.634,0.324082024518646)--(0.634117660388372,0.324)--(0.635,0.323384455152242)--(0.635550801050331,0.323)--(0.636,0.322686237516583)--(0.636981942135649,0.322)--(0.637,0.321987371052674)--(0.638,0.321287717193385)--(0.638411016627137,0.321)--(0.639,0.320587406411692)--(0.639838098142861,0.32)--(0.64,0.319886444151779)--(0.641,0.319184740413808)--(0.641263149090646,0.319)--(0.642,0.318482328388279)--(0.642686191973081,0.318)--(0.643,0.317779262160569)--(0.644,0.31707550424248)--(0.644107240670235,0.317)--(0.645,0.316370982681168)--(0.645526265439199,0.316)--(0.646,0.315665804118464)--(0.646943323339595,0.315)--(0.647,0.314959967910412)--(0.648,0.314253348039005)--(0.648358359474099,0.314)--(0.649,0.313546048575576)--(0.64977142848866,0.313)--(0.65,0.312838088564046)--(0.651,0.3121294027169)--(0.651182514143697,0.312)--(0.652,0.311419973585244)--(0.652591613796452,0.311)--(0.653,0.310709880924313)--(0.653998768463666,0.31)--(0.654,0.309999124027758)--(0.655,0.309287556696319)--(0.655403918076659,0.309)--(0.656,0.308575322333541)--(0.65680713119334,0.308)--(0.657,0.30786242064588)--(0.658,0.30714877494518)--(0.658208379306052,0.307)--(0.659,0.306434389618083)--(0.659607669430664,0.306)--(0.66,0.305719333794989)--(0.661,0.305003604847477)--(0.66100503463624,0.305)--(0.662,0.30428705907994)--(0.662400419941288,0.304)--(0.663,0.303569839561916)--(0.663793889360951,0.303)--(0.664,0.302851945492014)--(0.665,0.302133306488106)--(0.665185418681831,0.302)--(0.666,0.301413913496687)--(0.666575009165706,0.301)--(0.667,0.300693842583697)--(0.667962695318664,0.3)--(0.668,0.299973092908691)--(0.669,0.29925153060344)--(0.669348429604887,0.299)--(0.67,0.298529272758789)--(0.670732261918915,0.298)--(0.671,0.297806332664088)--(0.672,0.297082665331037)--(0.672114184678867,0.297)--(0.673,0.2963582102375)--(0.673494179863211,0.296)--(0.674,0.295633069316282)--(0.674872289537646,0.295)--(0.675,0.294907241654692)--(0.676,0.294180628670804)--(0.67624848201411,0.294)--(0.677,0.293453276281558)--(0.67762277835164,0.293)--(0.678,0.292725233450185)--(0.678995199810848,0.292)--(0.679,0.291996499221611)--(0.68,0.291266926392323)--(0.680365699296807,0.291)--(0.681,0.290536657416551)--(0.681734330455189,0.29)--(0.682,0.28980569321413)--(0.683,0.289073991886421)--(0.683101083370905,0.289)--(0.684,0.288341485545948)--(0.684465939534061,0.288)--(0.685,0.287608280052709)--(0.68582893809946,0.287)--(0.686,0.286874374373155)--(0.687,0.286139689220097)--(0.687190056957561,0.286)--(0.688,0.285404230866522)--(0.688549302308287,0.285)--(0.689,0.28466806826624)--(0.689906699798362,0.284)--(0.69,0.283931200339508)--(0.691,0.283193516157644)--(0.691262218497733,0.283)--(0.692,0.282455084707867)--(0.692615883823054,0.282)--(0.693,0.281715943732905)--(0.69396771069917,0.281)--(0.694,0.280976092105307)--(0.695,0.280235393295797)--(0.695317661907142,0.28)--(0.696,0.27949396697406)--(0.69666577694631,0.279)--(0.697,0.278751825658336)--(0.698,0.27800896297046)--(0.698012061131561,0.278)--(0.699,0.277265238730531)--(0.699356477795619,0.277)--(0.7,0.276520795047312)--(0.700699071250985,0.276)--(0.701,0.275775630704897)--(0.702,0.275029726993573)--(0.702039839298252,0.275)--(0.703,0.274282967982416)--(0.703378753505902,0.274)--(0.704,0.273535483712258)--(0.704715853053006,0.273)--(0.705,0.27278727291489)--(0.706,0.272038311477607)--(0.706051134424614,0.272)--(0.707,0.27128849391823)--(0.707384573149854,0.271)--(0.708,0.270537945076449)--(0.70871620544675,0.27)--(0.709,0.269786663629918)--(0.71,0.269034627329952)--(0.710046028436284,0.269)--(0.711,0.268281726668635)--(0.711374017642135,0.268)--(0.712,0.267528088486847)--(0.712700208337928,0.267)--(0.713,0.266773711406167)--(0.714,0.266018582660748)--(0.714024600084797,0.266)--(0.715,0.265262573541677)--(0.715347164731301,0.265)--(0.716,0.26450582044206)--(0.71666793847391,0.264)--(0.717,0.263748321925359)--(0.717986923542892,0.263)--(0.718,0.262990076541235)--(0.719,0.262230938931866)--(0.719304089028365,0.262)--(0.72,0.261471044500091)--(0.72061946947151,0.261)--(0.721,0.260710397901153)--(0.721933068440806,0.26)--(0.722,0.259948997623879)--(0.723,0.259186724224573)--(0.723244862032799,0.259)--(0.724,0.258423661181321)--(0.724554871842193,0.258)--(0.725,0.257659838980678)--(0.725863107068434,0.257)--(0.726,0.256895256048335)--(0.727,0.256129827695463)--(0.727169552155975,0.256)--(0.728,0.255363567866436)--(0.72847421301473,0.255)--(0.729,0.254596541640935)--(0.729777105872939,0.254)--(0.73,0.253828747379081)--(0.731,0.253060144404653)--(0.731078224742023,0.253)--(0.732,0.252290658688989)--(0.732377557355255,0.252)--(0.733,0.25152039907973)--(0.733675128243198,0.251)--(0.734,0.25074936386886)--(0.734970939224089,0.25)--(0.735,0.24997755133194)--(0.736,0.249204824422263)--(0.736264966184718,0.249)--(0.737,0.24843130110081)--(0.737557234526448,0.248)--(0.738,0.247656994342043)--(0.738847748833834,0.247)--(0.739,0.246881902349945)--(0.74,0.246105952360077)--(0.740136497793707,0.246)--(0.741,0.245329133992818)--(0.741423482042268,0.245)--(0.742,0.244551524071715)--(0.742708717822395,0.244)--(0.743,0.243773120726255)--(0.743992206730998,0.243)--(0.744,0.242993922067834)--(0.745,0.24221378040169)--(0.745273925093889,0.242)--(0.746,0.241432834650068)--(0.746553899523741,0.241)--(0.747,0.24065108698817)--(0.747832132244784,0.24)--(0.748,0.239868535448983)--(0.749,0.23908511919605)--(0.749108614976763,0.239)--(0.75,0.238300803851656)--(0.750383344264794,0.238)--(0.751,0.237515677804395)--(0.751656336699045,0.237)--(0.752,0.236729739005517)--(0.7529275936533,0.236)--(0.753,0.235942985386266)--(0.754,0.235155305489059)--(0.754197099370339,0.235)--(0.755,0.234366765838373)--(0.755464866449106,0.234)--(0.756,0.233577404239075)--(0.756730902497121,0.233)--(0.757,0.2327872185362)--(0.757995208760302,0.232)--(0.758,0.231996206553603)--(0.759,0.231204219593343)--(0.759257764848834,0.231)--(0.76,0.230411398444573)--(0.760518593640992,0.23)--(0.761,0.229617743569138)--(0.761777696689269,0.229)--(0.762,0.228823252699859)--(0.763,0.228027903248553)--(0.763035072248796,0.228)--(0.764,0.227231584543711)--(0.76429070646088,0.227)--(0.765,0.226434422134778)--(0.765544618657748,0.226)--(0.766,0.225636413659437)--(0.766796809857005,0.225)--(0.767,0.224837556731801)--(0.768,0.224037820908827)--(0.768047277320971,0.224)--(0.769,0.223237111280444)--(0.769296010049826,0.223)--(0.77,0.22243554513906)--(0.770543025095485,0.222)--(0.771,0.221633119998102)--(0.77178832334377,0.221)--(0.772,0.22082983334595)--(0.773,0.220025663261432)--(0.773031903223803,0.22)--(0.774,0.219220497996268)--(0.774273752196821,0.219)--(0.775,0.218414462788951)--(0.775513887268075,0.218)--(0.776,0.217607555021168)--(0.776752309189692,0.217)--(0.777,0.216799772047953)--(0.777989018681383,0.216)--(0.778,0.215991111197358)--(0.779,0.215181429431969)--(0.77922400014654,0.215)--(0.78,0.214370856800921)--(0.780457270128516,0.214)--(0.781,0.213559397383653)--(0.781688830048524,0.213)--(0.782,0.212747048393588)--(0.782918680489758,0.212)--(0.783,0.211933807015384)--(0.784,0.211119576115644)--(0.784146811714588,0.211)--(0.785,0.210304394502275)--(0.785373229171715,0.21)--(0.786,0.209488311171806)--(0.78659793908573,0.209)--(0.787,0.208671323187065)--(0.787820941900751,0.208)--(0.788,0.207853427580172)--(0.789,0.207034593539963)--(0.789042235175523,0.207)--(0.79,0.206214727993692)--(0.790261810317586,0.206)--(0.791,0.205393945050603)--(0.79147967985658,0.205)--(0.792,0.204572241613148)--(0.792695844094673,0.204)--(0.793,0.203749614550947)--(0.793910303299256,0.203)--(0.794,0.202926060700349)--(0.795,0.202101493241968)--(0.795123049773205,0.202)--(0.796,0.20127593133987)--(0.796334086162966,0.201)--(0.797,0.200449432296273)--(0.797543418456539,0.2)--(0.798,0.199621992807431)--(0.798751046775234,0.199)--(0.799,0.19879360953396)--(0.799956971204431,0.198)--(0.8,0.197964279100337)--(0.801,0.197133885011535)--(0.80116118188746,0.197)--(0.802,0.196302506278985)--(0.802363686404293,0.196)--(0.803,0.195470169409022)--(0.803564487394641,0.195)--(0.804,0.194636870874423)--(0.804763584793083,0.194)--(0.805,0.193802607109181)--(0.805960978496958,0.193)--(0.806,0.192967374507946)--(0.807,0.192131055911011)--(0.807156659192185,0.192)--(0.808,0.191293732462482)--(0.808350633877887,0.191)--(0.809,0.190455428525682)--(0.809542904639645,0.19)--(0.81,0.189616140330599)--(0.810733471218623,0.189)--(0.811,0.188775864064909)--(0.811922333317264,0.188)--(0.812,0.187934595873346)--(0.813,0.187092249888063)--(0.813109484495695,0.187)--(0.814,0.186248845815917)--(0.81429492640101,0.186)--(0.815,0.185404437433881)--(0.815478662983506,0.185)--(0.816,0.184559020707511)--(0.816660693783342,0.184)--(0.817,0.183712591556083)--(0.817841018300288,0.183)--(0.818,0.182865145851878)--(0.819,0.18201666422344)--(0.819019634952435,0.182)--(0.82,0.181167034915529)--(0.820196535961088,0.181)--(0.821,0.180316375875634)--(0.821371729207619,0.18)--(0.822,0.179464682780556)--(0.822545214022587,0.179)--(0.823,0.178611951256327)--(0.823716989694276,0.178)--(0.824,0.177758176877412)--(0.824887055468289,0.177)--(0.825,0.176903355165875)--(0.826,0.176047436940209)--(0.826055407783939,0.176)--(0.827,0.175190371415484)--(0.827222043125427,0.175)--(0.828,0.174332244351098)--(0.828386966291958,0.174)--(0.829,0.173473051051815)--(0.829550176350362,0.173)--(0.83,0.172612786765638)--(0.830711672322606,0.172)--(0.831,0.171751446682866)--(0.831871453185321,0.171)--(0.832,0.17088902593513)--(0.833,0.170025494807895)--(0.833029516486877,0.17)--(0.834,0.169160765514616)--(0.834185856642377,0.169)--(0.835,0.168294940207486)--(0.835340478567539,0.168)--(0.836,0.167428013775594)--(0.836493381049132,0.167)--(0.837,0.166559981044281)--(0.837644562826066,0.166)--(0.838,0.165690836774041)--(0.83879402258885,0.165)--(0.839,0.164820575659393)--(0.839941758979024,0.164)--(0.84,0.163949192327726)--(0.841,0.163076603923133)--(0.841087766774898,0.163)--(0.842,0.162202831044244)--(0.842232045982069,0.162)--(0.843,0.161327919120936)--(0.84337459764558,0.161)--(0.844,0.160451862503901)--(0.84451542020188,0.16)--(0.845,0.159574655470613)--(0.845654512035458,0.159)--(0.846,0.158696292223996)--(0.846791871478193,0.158)--(0.847,0.157816766891063)--(0.84792749680869,0.157)--(0.848,0.156936073521516)--(0.849,0.156054149140761)--(0.849061383783411,0.156)--(0.85,0.155170977622745)--(0.850193530279768,0.155)--(0.851,0.154286619553477)--(0.85132393735622,0.154)--(0.852,0.153401068665)--(0.852452603068124,0.153)--(0.853,0.152514318604695)--(0.853579525413995,0.152)--(0.854,0.151626362933658)--(0.854704702334728,0.151)--(0.855,0.150737195125038)--(0.855828131712811,0.15)--(0.856,0.149846808562332)--(0.8569498113715,0.149)--(0.857,0.148955196537632)--(0.858,0.148062283594288)--(0.858069736515796,0.148)--(0.859,0.147168082405963)--(0.859187905707724,0.147)--(0.86,0.146272635027344)--(0.86030431844686,0.146)--(0.861,0.145375934376113)--(0.861418972308045,0.145)--(0.862,0.144477973269341)--(0.862531864802508,0.144)--(0.863,0.143578744421448)--(0.863642993376918,0.143)--(0.864,0.142678240442092)--(0.864752355412405,0.142)--(0.865,0.141776453834015)--(0.865859948223552,0.141)--(0.866,0.140873376990811)--(0.866965769057367,0.14)--(0.867,0.13996900219465)--(0.868,0.139063247942353)--(0.868069812793033,0.139)--(0.869,0.138156144273539)--(0.869172077837433,0.138)--(0.87,0.137247718975518)--(0.870272562366529,0.137)--(0.871,0.136337963875421)--(0.871371263343434,0.136)--(0.872,0.135426870676956)--(0.872468177658112,0.135)--(0.873,0.134514430957702)--(0.873563302126161,0.134)--(0.874,0.133600636166333)--(0.874656633487563,0.133)--(0.875,0.132685477619759)--(0.875748168405401,0.132)--(0.876,0.131768946500179)--(0.876837903464542,0.131)--(0.877,0.130851033852039)--(0.877925835170269,0.13)--(0.878,0.129931730578909)--(0.879,0.129011013757217)--(0.879011959599823,0.129)--(0.88,0.128088803984031)--(0.880096271378113,0.128)--(0.881,0.127165175894213)--(0.881178768940486,0.127)--(0.882,0.126240119817771)--(0.882259448457063,0.126)--(0.883,0.125313625927403)--(0.883338306011001,0.125)--(0.884,0.12438568423476)--(0.88441533759687,0.124)--(0.885,0.123456284586583)--(0.885490539118994,0.123)--(0.886,0.122525416660715)--(0.886563906389736,0.122)--(0.887,0.121593069961983)--(0.887635435127726,0.121)--(0.888,0.120659233817943)--(0.888705120956041,0.12)--(0.889,0.119723897374482)--(0.889772959400328,0.119)--(0.89,0.118787049591276)--(0.890838945886869,0.118)--(0.891,0.117848679237088)--(0.891903075740577,0.117)--(0.892,0.116908774884908)--(0.892965344182945,0.116)--(0.893,0.115967324906927)--(0.894,0.115024284225151)--(0.894025745712446,0.115)--(0.895,0.114079630717758)--(0.895084275196929,0.114)--(0.896,0.113133396510777)--(0.896140928375321,0.113)--(0.897,0.112185569160951)--(0.897195700031309,0.112)--(0.898,0.111236135997779)--(0.898248584833831,0.111)--(0.899,0.110285084117372)--(0.899299577334585,0.11)--(0.9,0.109332400376094)--(0.900348671965444,0.109)--(0.901,0.108378071383962)--(0.90139586303579,0.108)--(0.902,0.107422083497799)--(0.902441144729757,0.107)--(0.903,0.106464422814131)--(0.903484511103375,0.106)--(0.904,0.105505075161816)--(0.904525956081626,0.105)--(0.905,0.104544026094386)--(0.905565473455378,0.104)--(0.906,0.103581260882092)--(0.906603056878233,0.103)--(0.907,0.102616764503642)--(0.907638699863247,0.102)--(0.908,0.10165052163761)--(0.908672395779538,0.101)--(0.909,0.100682516653495)--(0.909704137848775,0.1)--(0.91,0.0997127336024338)--(0.910733919141527,0.099)--(0.911,0.0987411562075184)--(0.91176173257349,0.098)--(0.912,0.0977677678537282)--(0.912787570901569,0.097)--(0.913,0.0967925515774355)--(0.913811426719807,0.096)--(0.914,0.0958154900554712)--(0.914833292455165,0.095)--(0.915,0.094836565593726)--(0.915853160363143,0.094)--(0.916,0.0938557601152623)--(0.91687102252322,0.093)--(0.917,0.092873055147907)--(0.917886870834126,0.092)--(0.918,0.0918884318113019)--(0.918900697008925,0.091)--(0.919,0.0909018708033774)--(0.919912492569893,0.09)--(0.92,0.0899133523862181)--(0.920922248843201,0.089)--(0.921,0.0889228563712866)--(0.921929956953371,0.088)--(0.922,0.0879303621039684)--(0.922935607817508,0.087)--(0.923,0.086935848447396)--(0.923939192139288,0.086)--(0.924,0.0859392937655134)--(0.924940700402695,0.085)--(0.925,0.0849406759053339)--(0.925940122865488,0.084)--(0.926,0.0839399721783411)--(0.926937449552386,0.083)--(0.927,0.0829371593409831)--(0.927932670247956,0.082)--(0.928,0.0819322135742023)--(0.928925774489183,0.081)--(0.929,0.0809251104619413)--(0.929916751557711,0.08)--(0.93,0.0799158249685594)--(0.930905590471734,0.079)--(0.931,0.0789043314150888)--(0.931892279977503,0.078)--(0.932,0.0778906034542586)--(0.932876808540456,0.077)--(0.933,0.0768746140441976)--(0.933859164335915,0.076)--(0.934,0.0758563354207384)--(0.934839335239349,0.075)--(0.935,0.0748357390682173)--(0.935817308816164,0.074)--(0.936,0.0738127956886744)--(0.936793072310988,0.073)--(0.937,0.0727874751693372)--(0.937766612636438,0.072)--(0.938,0.0717597465482708)--(0.938737916361308,0.071)--(0.939,0.0707295779780593)--(0.939706969698161,0.07)--(0.94,0.0696969366873777)--(0.940673758490283,0.069)--(0.941,0.0686617889402968)--(0.941638268197949,0.068)--(0.942,0.0676240999931549)--(0.942600483883961,0.067)--(0.943,0.0665838340488044)--(0.943560390198401,0.066)--(0.944,0.0655409542080382)--(0.944517971362557,0.065)--(0.945,0.0644954224179688)--(0.94547321115195,0.064)--(0.946,0.0634471994171204)--(0.946426092878403,0.063)--(0.947,0.0623962446769658)--(0.947376599371085,0.062)--(0.948,0.0613425163396163)--(0.94832471295645,0.061)--(0.949,0.0602859711513422)--(0.949270415436986,0.06)--(0.95,0.0592265643915685)--(0.950213688068689,0.059)--(0.951,0.0581642497969542)--(0.951154511537159,0.058)--(0.952,0.0570989794801207)--(0.952092865932207,0.057)--(0.953,0.05603070384255)--(0.953028730720856,0.056)--(0.953962084431167,0.055)--(0.954,0.0549592719361167)--(0.954892905269116,0.054)--(0.955,0.0538846432213404)--(0.955821170878023,0.053)--(0.956,0.0528068358937856)--(0.956746857927659,0.052)--(0.957,0.051725791758312)--(0.957669942300767,0.051)--(0.958,0.050641450310309)--(0.95859039905312,0.05)--(0.959,0.0495537485949304)--(0.959508202370704,0.049)--(0.96,0.0484626210546054)--(0.960423325523749,0.048)--(0.961,0.0473679993635752)--(0.961335740817314,0.047)--(0.962,0.0462698122480443)--(0.962245419538089,0.046)--(0.963,0.0451679852903452)--(0.963152331897026,0.045)--(0.964,0.0440624407152977)--(0.964056446967392,0.044)--(0.964957732395664,0.043)--(0.965,0.0429529590177959)--(0.96585615470952,0.042)--(0.966,0.0418393895142421)--(0.966751679451603,0.041)--(0.967,0.0407218220886629)--(0.967644270415671,0.04)--(0.968,0.0396001612439605)--(0.968533889860038,0.039)--(0.969,0.0384743065434712)--(0.969420498405485,0.038)--(0.97,0.0373441522146747)--(0.970304054923515,0.037)--(0.971,0.0362095867091469)--(0.971184516413751,0.036)--(0.972,0.0350704922125429)--(0.972061837869099,0.035)--(0.972935971888801,0.034)--(0.973,0.0339264877426514)--(0.973806869017558,0.033)--(0.974,0.0327774171520775)--(0.974674477511656,0.032)--(0.975,0.0316233784440684)--(0.975538742682437,0.031)--(0.976,0.030464218799159)--(0.976399606910271,0.03)--(0.977,0.0292997751659442)--(0.977257009391322,0.029)--(0.978,0.0281298731866888)--(0.978110885852983,0.028)--(0.978961168129481,0.027)--(0.979,0.0269541371320868)--(0.979807783830768,0.026)--(0.98,0.0257719681402492)--(0.980650656485268,0.025)--(0.981,0.024583665965154)--(0.981489704269786,0.024)--(0.982,0.0233889887038175)--(0.982324839777082,0.023)--(0.983,0.0221876734596757)--(0.983155969387502,0.022)--(0.983982992506828,0.021)--(0.984,0.0209793317153514)--(0.98480580053328,0.02)--(0.985,0.0197627459744404)--(0.985624276608731,0.019)--(0.986,0.0185384538230091)--(0.986438293561597,0.018)--(0.987,0.0173060550760535)--(0.987247712780257,0.017)--(0.988,0.016065102967552)--(0.988052382475469,0.016)--(0.988852135385559,0.015)--(0.989,0.0148139365925185)--(0.989646786625294,0.014)--(0.99,0.0135525406527875)--(0.99043612999663,0.013)--(0.991,0.0122806123349915)--(0.991219933791318,0.012)--(0.991997935397535,0.011)--(0.992,0.0109973254688588)--(0.992769833934328,0.01)--(0.993,0.00969931831539631)--(0.993535280922614,0.009)--(0.994,0.00838740140762081)--(0.994293866105562,0.008)--(0.995,0.00706003515130964)--(0.995045097886243,0.007)--(0.995788373488889,0.006)--(0.996,0.00571190997044866)--(0.996522931764036,0.005)--(0.997,0.00434184009260091)--(0.997247771823045,0.004)--(0.997961499516484,0.003)--(0.998,0.00294503966262914)--(0.998661989286787,0.002)--(0.999,0.00150542327964402)--(0.999345413389119,0.001)--(1,0);